\documentclass[10pt, a4paper]{amsart}

\usepackage{amssymb}
\usepackage{amsmath}
\usepackage{latexsym}
\usepackage{amscd}
\usepackage{amsfonts}
\usepackage{pb-diagram}
\usepackage[all]{xypic}
\usepackage[mathcal,mathscr]{eucal}
\usepackage{amsthm}
\usepackage{graphicx}
\usepackage{url}
\usepackage{pdfsync}
\usepackage{bbm}
\usepackage{xypic}
\usepackage[pdftex,colorlinks]{hyperref}

\hypersetup{linkcolor = blue, citecolor=blue}

\usepackage{pifont}
\usepackage{amsbsy}
\usepackage{epstopdf}
\usepackage{times}
\usepackage{epic,eepic}

\newtheorem{theorem}{Theorem}[section]
\newtheorem{lemma}{Lemma}[section]
\newtheorem{definition}{Definition}[section]
\newtheorem{corollary}{Corollary}[section]
\newtheorem{proposition}{Proposition}[section]
\theoremstyle{definition}
\newtheorem{remark}{Remark}[section]
\newtheorem{example}{Example}[section]

\newcommand{\intav}[1]{\mathchoice {\mathop{\vrule width 6pt height 3 pt depth  -2.5pt
\kern -8pt \intop}\nolimits_{\kern -6pt#1}} {\mathop{\vrule width
5pt height 3  pt depth -2.6pt \kern -6pt \intop}\nolimits_{#1}}
{\mathop{\vrule width 5pt height 3 pt depth -2.6pt \kern -6pt
\intop}\nolimits_{#1}} {\mathop{\vrule width 5pt height 3 pt depth
-2.6pt \kern -6pt \intop}\nolimits_{#1}}}

\newcommand{\intavl}[1]{\mathchoice {\mathop{\vrule width 6pt height 3 pt depth  -2.5pt
\kern -8pt \intop}\limits_{\kern -6pt#1}} {\mathop{\vrule width 5pt
height 3  pt depth -2.6pt \kern -6pt \intop}\nolimits_{#1}}
{\mathop{\vrule width 5pt height 3 pt depth -2.6pt \kern -6pt
\intop}\nolimits_{#1}} {\mathop{\vrule width 5pt height 3 pt depth
-2.6pt \kern -6pt \intop}\nolimits_{#1}}}

\begin{document}

\title[Inhomogeneous Hopf-Ole\u{\i}nik Lemma and Krylov's boundary gradient estimates]{Inhomogeneous Hopf-Ole\u{\i}nik Lemma and Applications. Part IV: Sharp Krylov Boundary Gradient Type Estimates for Solutions to Fully Nonlinear Differential Inequalities with unbounded coefficients and $C^{1,Dini}$ \\ boundary data}

\author{J. Ederson M. Braga}
   \address{Departamento  de Matem\'atica, Universidade Federal do Cear\'a \newline
Campus do Pici - Bloco 914, CEP 60455-760, Fortaleza, Cear\'a, Brazil.}
   \email{eder\_mate@hotmail.com}

\author{Diego Moreira}
   \address{Departamento  de Matem\'atica, Universidade Federal do Cear\'a \newline
Campus do Pici - Bloco 914, CEP 60455-760, Fortaleza, Cear\'a, Brazil.}
   \email{dmoreira@mat.ufc.br}
  
 \author{Lihe Wang}
 \address{Department of Mathematics, University of Iowa City,  \newline 
   Iowa City, IA 52242, USA}
\email{lwang@math.uiowa.edu}
   
\begin{abstract}
In this paper we provide another application of the Inhomogeneous Hopf-Ole\u{\i}nik Lemma (IHOL) proved in \cite{BM-IHOL-PartI} or \cite{Boyan-2}. As a matter of fact, we also provide a new and simpler proof of a slightly weaker version IHOL  for the uniformly elliptic fully nonlinear case which is sufficient for most purposes.  The paper has essentially two parts. In the first part, we use IHOL for unbounded RHS to develop a Caffarelli's ``Lipschitz implies $C^{1,\alpha}$" approach to prove Ladyzhenskaya-Uraltseva boundary gradient type estimates for functions in $S^{*}(\gamma, f)$ that vanishes on the boundary. Here, unbounded RHS  means that $f\in L^{q}$  with $q>n$. This extends the celebrated Krylov's boundary gradient estimate proved in \cite{Krylov}. A Phragm\'en-Lindel\"of classification result for solutions in half spaces is recovered from these estimates. Moreover, a H\"older estimate up to the boundary (in the half-ball) for $u(x)/x_{n}$ is obtained.  In the second part, we extend the previous results for functions in $S^{*}(\gamma, \sigma, f)$ where $\gamma,f\in L^{q}$ with $q>n$ that have a $C^{1,Dini}$ boundary data on a $W^{2,q}$ domain.  Here, we use an ``improvement of flatness" strategy suited to the unbounded coefficients scenario.  As a consequence of that, a quantitative version of IHOL under pointwise $C^{1,Dini}$ boundary regularity is obtained.  \end{abstract}

\thanks{Corresponding author: Diego Moreira - dmoreira@mat.ufc.br}
\maketitle

%\tableofcontents

\section{Introduction}
We start by recalling a celebrated result due to N. Krylov.
\begin{theorem}[Krylov, \cite{Krylov}]\label{krylov-trundiger-statements-intro} Let $u\in W^{2,n}(B_{R_{0}}^{+})\cap C^{0}(\overline{B}_{R_{0}}^{+})$ be a strong solution to $Lu=Trace(A(x)D^2u)=f$ in $B_{R_{0}}^{+}$ and $u=0$ in $B'_{R_{0}}:=\partial B_{R_{0}}^{+}\cap\big\{x_{n}=0\}$ where $A$ is a $(\lambda, \Lambda)\footnote{This means that $A$ is a symmetric matrix of order $n$ such that $ \lambda|\xi|^{2}\leq \langle A(x)\xi, \xi \rangle \leq \Lambda|\xi|^{2} \quad \forall x\in B_{1}^{+}, \ \ \forall \xi\in\mathbb{R}^{n}.$}-$uniformly elliptic matrix  of order $n$ and $f\in L^{\infty}(B_{R_{0}}^{+})$. Then, for any $r\leq R_{0},$ we have  
\begin{equation}\label{original-Krylov-oscillation-estimate}
\underset{B_{r}^{+}}{~osc ~}\bigg(\frac{u}{x_{n}}\bigg)\leq C\bigg(\frac{r}{R_{0}}\bigg)^{\alpha}\Bigg(\underset{B_{R_{0}}^{+}}{~osc ~}\bigg(\frac{u}{x_{n}}\bigg) + ||f||_{L^{\infty}(B_{R_{0}}^{+})}\Bigg)
\end{equation}
where $\alpha\in(0,1)$ and $C>0$  are universal constants depending only on $n, \lambda, \Lambda.$
\end{theorem}
One can prove that the estimate above actually implies the existence of the classical gradient of the solution on the flat boundary $B'_{R_{0}}$ and that the gradient is H\"older continuous there. As a matter of fact, the following estimate holds
\begin{equation}\label{modulus-continuity-gradient-krylov}
\underset{B_{r}^{+}}{~osc ~} |\nabla u(x',0)| \leq C\bigg(\frac{r}{R_{0}}\bigg)^{\alpha}\Bigg(\underset{B_{R_{0}}^{+}}{~osc ~}\bigg(\frac{u}{x_{n}}\bigg) + ||f||_{L^{\infty}(B_{R_{0}}^{+})}\Bigg).
\end{equation}
\indent Krylov's result is indeed impressive. It is known from the Krylov-Safonov theory in \cite{KS1,KS2} and \cite{Sa} that solutions of uniformly elliptic equations with bounded measurable coefficients are at most H\"older continuous inside the domain and thus the classical gradient may not even exist in the interior. Krylov's result lines up with the observation that solutions to nondivergence type equations tend to behave better on the boundary. 

Now, we start describing the goals of this paper. The first one (which is the central one) is to extend Krylov's result above for up to the boundary continuous $L^{n}-$viscosity solutions to the following Dirichlet problem 
\begin{eqnarray}
\mathcal{M}_{\lambda, \Lambda}^{+}(D^2u) +\gamma(x)|\nabla u| +\sigma(x)u &\geq &-|f| \quad \textnormal { in } \ \Omega \\ \label{DI+}
\mathcal{M}_{\lambda, \Lambda}^{-}(D^2u) -\gamma(x)|\nabla u| +\sigma(x)u &\leq &|f| \quad \textnormal { in } \ \Omega \\ \label{DI-}
u&=&\varphi \ \textnormal { on } \ \partial\Omega. \label{bdry-data-intro}
\end{eqnarray}
\noindent Here,  $0\leq \gamma \in L^{q}(\Omega)$ and $ \sigma, f\in L^{q}(\Omega)$ with $q>n.$ Moreover, $\varphi$ is a  $C^{1,Dini}$ boundary data along $\partial\Omega$ where $\Omega\subset\mathbb{R}^{n}$ is a bounded $W^{2,q}$ domain. This paper can be divided into two macro parts that we now start to describe. In the first one, we deal with Krylov's result in  the zero boundary data case (Theorem \ref{boundary krylov thm Lq version}). Two ingredients here come into play, namely, Lipschitz estimates up to the boundary and the Inhomogeneous Hopf-Ole\u{\i}nik Lemma that we cal IHOL for short. The interplay of these estimates (properly normalized and applied to every scale) allows us to implement a L. Caffarelli ``Lipschitz implies $C^{1,\alpha}$"  type approach that renders Krylov's result. In a certain way, the ``Lipschitz implies $C^{1,\alpha}$" approach developed here can be thought as the analogue (for this context) to the free boundary regularity theory developed by L. Caffarelli in \cite{C1}. 

IHOL was proven in \cite{BM-IHOL-PartI} for fully nonlinear and quasilinear type as well as uniformly elliptic fully nonlinear equations.  We point out that an earlier and slightly stronger version of IHOL for the uniformly elliptic fully nonlinear case is due to B. Sirakov. To the best of our knowledge, this stronger version was first stated in \cite{Boyan-1} and a full proof appeared quite recently in \cite{Boyan-2} (see Theorems 2 and 11 in \cite{Boyan-2}).  As a matter  of fact, the papers \cite{Boyan-1, Boyan-2} contain quite nice new estimates. 

In this paper, Krylov's boundary gradient type estimates concern only uniformly elliptic fully nonlinear equations. For this context only, we present a simpler and alternative proof of a slightly weaker version of IHOL found in \cite{Boyan-1, Boyan-2}  ~(see Theorem \ref{IHOL} and Remark \ref{IHOL-supersolutions} below) which is the same version found in \cite{BM-IHOL-PartI}. In reality, this version of IHOL is sufficient for our purposes here. This is the second goal of this paper. We highlight that this alternative proof is elementary in nature and of geometric flavour. It relies on the construction of barriers for the Pucci extremal operators with unbounded RHS. These barriers enjoy some geometric properties that easily yield simple proofs of IHOL as well as the Lipschitz estimates up to the boundary\footnote{We observe that Theorem 3 in \cite{Boyan-2} also gives the Lipschitz estimates up to the boundary for bounded solutions.}. The ideas of the proofs are transportable to other situations. In fact, these Pucci barriers  are of multiple use and thus interesting on their own right. 

We remark that Lipschitz type estimates up to the boundary obtained here are sharp in the sense that they do not hold for $q=n$ (see section 9). Moreover, Krylov's $C^{1,\alpha}$ type estimates allow us to recover a classification result of Phragm\'en-Lindel\"of type for solutions to homogeneous equations in half-spaces (see Remark \ref{PLT}). We also point out that Krylov's $C^{1,\alpha}$ type estimates along the boundary obtained here also implies a H\"older control up to the boundary for $u(x)/d(x)$ where $d(x)$ represents the distance to the boundary for the unbounded RHS case.

In the second part, we discuss Krylov's result under the $C^{1,Dini}$ (nonzero) boundary data case (Theorem \ref{general-pointwise-boundary-krylov-general-boundary-data}). Here, the previous strategy of ``Lipschitz implies $C^{1,\alpha}$" becomes more delicate to implement since one has to account for the ``wiggling" oscillation of the boundary data.  Here, we follow an ``improvement of flatness" type method (see for instance the one implemented in \cite{SS} for the bounded coefficients case). In our case however, differently from \cite{SS}, the differential inequalities we deal with do not have an ``envelope class", i.e, if $\gamma, f\in L^{\infty}$ then $S^{*}(\gamma; f) \subset S^{*}(||\gamma||_{L^{\infty}}; ||f||_{L^{\infty}}).$ Furthermore, in our case the iteration becomes more delicate since now the Dini character of the boundary data is indeed what drives the convergence of the tangent plane approximations at every scale. Here, we need some additional assumptions on the modulus of continuity that are discussed in the next section. We mention however that they seem to be  weaker and more natural than the ones presented by J. Kovats in \cite{Kovats-1, Kovats-2} once they are more aligned with the conditions that appeared earlier in the classical works of G. Lieberman, K. Widman and M. Borsuk (see Remark \ref{comments-beta-compatibility}). Our estimates hold in $W^{2,q}$ domains by ``flattening out" the boundary type arguments. We leave the details of these computation to the readers. As a consequence of that, we prove a version of IHOL under $C^{1,Dini}$ boundary data  regularity (see Corollary \ref{hopf-C-1-alpha-tangentially}) which complements Theorem \ref{IHOL}.  We also observe that the $C^{1,\alpha}$ estimates for the zero boundary data case presented here (first part) are in a sharper form when compared to the general boundary data case obtained in the second part.\\

Now, we mention some historical accounts and recent developments related to Krylov's result with no intention of being complete or exhaustive. Shortly after the Krylov's paper \cite{Krylov}, M. Safonov in \cite{Sa-simp} and L. Caffarelli (unpublished work) simplified Krylov's original proof.  To the best of our knowledge, L. Caffarelli's simplification appeared for the first time in J. Kazdan's book \cite{Ka}. 

The statement of Theorem \ref{krylov-trundiger-statements-intro} above is taken from Theorem 9.31 in D. Gilbarg and N. Trudinger's book \cite {GT} (which like \cite{Ka}) incorporates L. Caffarelli's argument on Krylov's proof. 
 As described at the end of Chapter VII in \cite{Lieberman-Book},  Krylov studied the quotient $u(x)/x_{n}$ by introducing new variables to the problem. L. Caffarelli's idea was to look directly to $u$ $($or even, perturbations of $u$ by linear functions, i.e, $u-Ax_{n}).$ This allows the elimination of the Krylov's added independent variables. The quotient $v(x):=u/x_{n}$ is proven to satisfy a Harnack type inequality.  This proof has a geometric flavour and it depends on the construction of a clever comparison barrier (a quadratic polynomial with precise curvature control in orthogonal directions)\footnote{Another instance where beautiful geometrical considerations implied Harnack type inequalities appeared in an earlier paper (\cite{Serrin}) due to J. Serrin. This was brought to our attention by L. Caffarelli.}. This construction seems, in principle,  delicate to reproduce in the unbounded RHS case.  Caffarelli-Krylov's approach fostered a number of interesting variants like, for instance, in \cite{Lieberman-CPDE, Lieberman-PISA} due to G. Lieberman (see section 5 in \cite{Lieberman-CPDE} and section 4 in \cite{Lieberman-PISA}). We indicate here the clear presentation of Caffarelli-Krylov's approach in Theorem 1.2.16 in the first chapter of the nice and recent book \cite{QH} written by Q.~Han.

 Recently,  the remarkable paper \cite{LU-Steklov}\footnote{It seems that results from this paper were obtained earlier by the same authors in a preprint. There is a nice note \cite{U-ICM} that contains precisely the statement of the results in \cite{LU-Steklov}.} due to O. Ladyzhenskaya and N. Uraltseva came to our knowledge. There, the authors extend Krylov's result by considering strong solutions to differential inequalities involving second order quasilinear equations in nondivergence form with lower order terms with coefficients in $L^{q}$ for $q>n.$
 
  In \cite{LU-Steklov}, the authors used barriers for homogeneous equations to explore boundary estimates for solutions to the inhomogeneous problem. They used ABP estimate that allows one to compare the homogeneous barriers with the actual solutions to the inhomogeneous problems. Lipschitz type estimates on the boundary are obtained by developing a quite delicate iteration scheme. An ``oscillation decay" type estimate is also needed. This is done by the use of barriers that are ``pieces of the fundamental solution type", i.e,  of the form $C_{1}+C_{2}|x|^{-\alpha}$ and a  Landis boundary growth type lemma for the quotient $v(x)=u(x)/x_{n}$. We remark that  their results also apply to $W^{2,q}-$domains. 
 
  We highlight the nice paper of B. Barcelo, L. Escauriaza and E. Fabes \cite{BEF}, where Krylov's boundary gradient type estimates were obtained for solutions to linear 2nd order uniformly elliptic equations in nondivergence form (like the one in Theorem \ref{krylov-trundiger-statements-intro}) via estimates on the Green's function. Some of our results here are extensions of the estimates in \cite{BEF} for the fully nonlinear case involving unbounded coefficients. We use however a completely different approach here.
  
   Another development on Krylov's result was done in the third author thesis. He extended Krylov's result for the fully nonlinear parabolic equations $(u\in S^{*}(g))$ in the case the RHS $g$ belongs to $L^{q}$ with $q>n+1$ (parabolic case) and the boundary and the boundary data are pointwise $C^{1,\alpha}$. There, it was developed an iteration scheme that finds a linear approximation for the solution at every dyadic scale. In order to estimate the decay of the Lipschitz constant, Gaussian type barriers were used. In fact, the strategy in \cite{Lihe-CPAM-II} is quite delicate and involves a combination of these type of barriers. Some of them are ``tilted" in order to capture at the same time the oscillation of the boundary and the boundary data at small scales. The iteration process then goes on by using the ABP estimate to measure the ``deviation" or the ``error" in the linear approximation with respect to the solution when one goes from one scale to the next. Roughly speaking, these ``errors" are controlled by the RHS, oscillations of the boundary data and boundary. Since $q>n+1$ and boundary data and boundary oscillates in a $C^{1,\alpha}$ fashion, these accumulated errors ``pile up". Krylov's $C^{1,\alpha}$ (pointwise) boundary estimates were proven on the lateral boundary. 
   
   It seems to us that Krylov's boundary gradient type estimates for viscosity solutions to fully nonlinear (parabolic) equations with unbounded RHS in (pointwise) $C^{1,\alpha}$ domains with (pointwise) $C^{1,\alpha}$ boundary data first appeared in the third author thesis (Theorem 2.1 in \cite{Lihe-CPAM-II}). By using the same type of ideas  present in \cite {Lihe-CPAM-II}, the third author and F. Ma studied the elliptic case with unbounded RHS and (pointwise) $C^{1,Dini}$ boundary and boundary data in \cite{Lihe-Ma} .
   
    In the papers \cite{Lihe-CPAM-II, Lihe-Ma} lower order terms were not considered. Furthermore, although the regularity results were proven in some detail, the gradient estimates were not written in the most precise way. A careful inspection in the proofs in \cite{Lihe-CPAM-II, Lihe-Ma} reveals that a simpler iteration scheme towards Krylov's result can be obtained for the case where the RHS in $L^{\infty}$ (at least in the zero boundary data case).  The simpler iteration scheme appears since the use of the  ABP estimate to control the error between scales is no longer necessary in the proof. In fact, ABP can be directly replaced by a more precise estimate, namely, the Hopf-Ole\u{\i}nik Lemma (HOL).  As a matter of fact, this observation was indeed the main motivation for the development of the method in \cite{Lihe-CPAM-II, Lihe-Ma}.  Now, under the possession of IHOL, we can simplify some of the delicate arguments in \cite{Lihe-CPAM-II} and obtain precise gradient estimates in a clear and direct way.

  The ideas surrounding Krylov's result still permeates the field of nonlinear elliptic and parabolic PDEs. To mention some recent and important examples, O. Savin and N.Q. Le proved  nice results on affine analogues of Krylov's and Ladyzhenskaya-Uraltseva's $C^{1,\alpha}$ results for the linearized Monge-Amp\`ere equation in \cite{Savin-Le-1, Savin-Le-2}. Even more recently, X. Ros-Oton and J. Serra followed some ideas of the original Caffarelli-Krylov's approach in order to prove Krylov's result in the context of nonlinear Integro-Differential operators with bounded RHS\footnote{As a matter of fact, we obtained (for the unbounded RHS case) a similar estimate for the H\"older norm up to the boundary for $u(x)/x_{n}$ as done in \cite{R-O-Se-1, R-O-Se-2} (see estimate (\ref{krylov-uraltseva-oscillation-estimate})). In fact,  for that matter, we used some nice ideas from \cite{R-O-Se-1, R-O-Se-2}.}. This is the content of the excellent papers \cite{R-O-Se-1, R-O-Se-2}. We suspect that some ideas in the present  paper may eventually be useful to explore estimates in the Integro-Differential operators with unbounded RHS setting.

Our paper is organized as follows: In section 2, we introduce some notation and section 3 is devoted to preliminares and definitions of the special classes of modulus of continuity considered here. Section 4 is destined to present the main results of this paper. In section 5, we introduced the structural conditions required for the PDEs. In section 6, we provide examples and properties of the modulus of continuity considered here. The purpose of section 7 is to present the alternative proof of IHOL and the new construction of the inhomogeneous Pucci barriers for the fully nonlinear case. In section 8, we give the proofs of Lipschitz estimates and IHOL on flat boundaries, namely, Propositions \ref{bdry-lip-type-estimate-scaled-version} and \ref{Boundary Behaviour ffb} respectively. In section 9, we discuss an  example that shows that the estimates we proved are sharp with respect to the RHS. ~Section 10 deals with the ``Lipschitz implies $C^{1,\alpha}$" approach to prove Ladyzhenskaya-Uraltseva estimates for the class $S^{*}(\gamma;f)$ in the zero boundary data case. In Section 11, we state and prove a new version of the ``Improvement of flatness Lemma" for the unbounded coefficients case. In sections 12, 13 and 14,   we prove the Krylov's result under the $C^{1,Dini}$ boundary data regularity assumptions, i.e, the proof of Theorem \ref{general-pointwise-boundary-krylov-general-boundary-data} and its Corollaries. In the Appendix, for completeness, we present some lemmas that related pointwise Taylor's expansion and $C^{1,\omega}$ regularity. These estimates are known (specially in the $C^{1,\alpha}$ case) but  it is not so easy to find a reference for their proofs.

\section{Notation}
\begin{itemize}
	\item $n\geq 2$ indicates the dimension of the Euclidean space.\vspace{.2cm}
	\item If $\Omega \subseteq {\mathbb{R}}^n$ we set $\Omega^+ := \Omega \cap \Big\{ (x_1, \cdots, x_n) \in {\mathbb{R}}^n : x_n \geq 0 \Big\}$. \vspace{2mm}
	\item$|A|$ is the $n-$dimensional Lebesgue measure of the set $A$.\vspace{.2cm}
	\item We denote sometimes $x\in\mathbb{R}^{n}$ as $x=(x',x_{n}),$ where $x'=(x_{1},\cdots, x_{n-1})\in\mathbb{R}^{n-1};$\vspace{.2cm}
	\item $H_{n-1}:=\partial\mathbb{R}_{+}^{n}:=\big\{(x',0); x'\in\mathbb{R}^{n-1} \big\};$ \vspace{.2cm}
	\item  $B_{r}(x_{0})=\big\{x\in \mathbb{R}^{n}; |x-x_{0}|<r\big\};$\vspace{.2cm}
	\item For $x_{0}\in H_{n-1}$, \ $B'_{r}(x_{0}):=\big\{x=(x',0); |x'-x_{0}|<r\big\}; $ \vspace{.2cm}
	\item $B'_{r}=B'_{r}(0), \quad B_{r}=B_{r}(0);$
	\end{itemize}
\section{Definitions and Preliminares}
\begin{definition}\label{definition-MC} A modulus of continuity is a nondecreasing continuous function $\omega:[0,\delta_{\omega}]\to[0,\infty)$ such that  $\omega(0)=0$ and $\omega(t)>0$ for $t>0$. Here $\delta_{\omega}\in (0,1].$ Additionally, we say 
\begin{itemize}
\item[$a)$] $\omega$ satisfies the $Q-$decreasing quotient property for $Q\geq 1$ if for $0\leq h\leq r\leq \delta_{\omega},$
$$q(r)\leq Q \cdot q(h),  \quad \textnormal {where }\quad q(t):=\frac{\omega(t)}{t}\quad \textnormal { is defined for } \ t\in(0,\delta_{\omega}].$$ 
\item[$b)$] $\omega$ is Dini continuous if 
$$ \int_{0}^{\delta_{\omega}}\frac{\omega(t)}{t}dt <\infty; $$
\item[$c)$] $\omega$ has the $\beta-$compatibility property between scales for some $\beta\in(0,1)$ if 
$$ \exists \ \delta_{\omega}^{*}\in (0,1] \quad \textnormal { so that } $$
$$ \forall \mu\in\big(0,\delta_{\omega}^{*}), \quad \mu^{\beta}\cdot \omega(\mu^{k}\cdot\delta) \leq \omega(\mu^{k+1}\cdot\delta) \quad\forall \delta\in(0,\delta_{\omega}] \ \textnormal { and } \ \forall k\geq k_{0} \ \textnormal { where } k_{0}\in\mathbb{N}. 
$$
 \end{itemize}
We indicate that a modulus of continuity $\omega$ satisfies all the properties above by writing $\omega\in\mathcal{DMC}(Q,\beta).$
\end{definition}
\noindent See also Remark \ref{comments-beta-compatibility} for the comments on $\beta-$compatibility condition. Also, we refer to section 7, where examples of modulus of continuity in the class $\mathcal{DMC}(Q,\beta)$ are given.
\begin{definition}\label{defining-pointwise-c1-dini} Let $u:B_{r}\to\mathbb{R}$ be a bounded function,  $x_{0}\in B_{r}$ and $\omega$ a modulus of continuity. We say that $u\in C^{1,\omega}(x_{0})$ if there exist an affine function $L_{x_{0}}$ such that 
 \begin{equation}\label{C^{1,alpha}-pointwise}
[u]_{C^{1,\omega}}(x_{0}):= \sup\limits_{x\in B_{r}\atop {0<|x-x_{0}|\leq \delta_{\omega}}} \frac{|u(x) - L_{x_{0}}(x)|}{|x-x_{0}|\omega(|x-x_{0}|)} <\infty.
 \end{equation}
\noindent It is easy to verify that the affine function $L_{x_{0}}$ is unique and 
$$ |u(x) - L_{x_{0}}(x)|\leq [u]_{C^{1,\omega}}(x_{0})|x-x_{0}|\omega(|x-x_{0}|) \quad \forall x\in B_{r} \ \textnormal { such that } \ |x-x_{0}|\leq \delta_{\omega}.
 $$
  \noindent We define the (first order) Taylor's polynomial of $u$ at $x_{0}$ to be the affine function $L_{x_{0}}$. We then write, 
 \begin{equation}\label{taylor-polynomial}
 L_{x_{0}}(x):= \nabla u(x_{0})\cdot (x-x_{0})+ u(x_{0}), \quad x\in\mathbb{R}^{n}.
 \end{equation}
and set 
 \begin{equation}\label{pointwise-C1-dini-norm}
||u||_{C^{1,\omega}(x_{0})}:= |u(x_{0})| + |\nabla u(x_{0})| + [u]_{C^{1,\omega}(x_{0})}.
\end{equation}
\end{definition}
\begin{remark}\label{taylor-on-the-fiber} Assume that $\varphi:B'_{r}\to\mathbb{R}.$ We define $\varphi\in C^{1,\omega}(x_{0})$  where $x_{0}\in B'_{r}$  exactly as in the previous definition imposing only that the affine function $L$ over $\mathbb{R}^{n}$  satisfies additionally that ${\partial L_{x_{0}}}/{\partial x_{n}}\equiv 0.$ Observe that in this case, (\ref{taylor-polynomial}) and (\ref{pointwise-C1-dini-norm}) are defined likewise and that now $\nabla\varphi(x_{0})\in \mathbb{R}^{n-1}\times\{0\}.$
\end{remark}
\begin{definition}  Let $u:B_{r}\to\mathbb{R}$ be a bounded function, $x_{0}\in B_{r}$ and $\alpha \in(0,1].$ We say that  $u\in C^{1,\alpha}(x_{0})$ if 
   \begin{equation}\label{C^{1,alpha}-pointwise}
[u]_{C^{1,\alpha}}(x_{0}):= \sup\limits_{x\in B_{r}} \frac{|u(x) - L_{x_{0}}(x)|}{|x-x_{0}|^{1+\alpha}} <\infty.
 \end{equation}
 \begin{equation}\label{C^{1,alpha}-semi-norm-pointwise-weighted}
[u]_{C^{1,\alpha}}^{*}(x_{0}):= r^{1+\alpha} \cdot [u]_{C^{1,\alpha}}(x_{0})
 \end{equation}
 \begin{equation}\label{weighted-pointwise-c1-alpha-norm}
||u||_{C^{1,\alpha}(x_{0})}^{*}:= |u(x_{0})| + r\cdot |\nabla u(x_{0})| + r^{1+\alpha}\cdot [u]_{C^{1,\alpha}(x_{0})}
\end{equation}
\end{definition}
\begin{remark} We observe that the concepts introduced above coincide. Indeed, 
$$  \sup\limits_{x\in B_{r}\atop {|x-x_{0}|\leq \delta_{\omega}}} \frac{|u(x) - L_{x_{0}}(x)|}{|x-x_{0}|^{1+\alpha}}\leq [u]_{C^{1,\alpha}}(x_{0}) \leq \sup\limits_{x\in B_{r}\atop {|x-x_{0}|\leq \delta_{\omega}}} \frac{|u(x) - L_{x_{0}}(x)|}{|x-x_{0}|^{1+\alpha}} + \Bigg(\frac{1}{\delta_{\omega}}\Bigg)^{1+\alpha} \Big(||u||_{L^{\infty}(B_{r})} + ||L_{x_{0}}||_{L^{\infty}(B_{r})}\Big).$$
\end{remark}
 \begin{remark}\label{classical-dini-taylor} Let $\omega$ be a modulus of continuity. We recall that $u\in C^{1,\omega}({B}_{r})$ if $u\in C^{1}(B_{r})$ and  $$[\nabla u]_{C^{0,\omega}(B_{r})}=\sup\limits_{x, y\in B_{r}, x\neq y \atop {|x-y|\leq \delta_{\omega}}} \frac{|\nabla u(x)-\nabla u(y)|}{\omega(|x-y|)} < \infty.
$$ By classical Taylor's expansion, it is easy to see that for any $x_{0}\in \overline{B}_{r/2}$ we have
$$ |u(x) -u(x_{0}) -\nabla u(x_{0})\cdot(x-x_{0})| \leq [\nabla u]_{C^{0,\omega}(\overline{B}_{r})} |x-x_{0}|\omega(|x-x_{0}|) \quad \forall x\in \overline{B}_{\min\{r/2, \delta_{\omega}\}}(x_{0}).$$
\end{remark}
\noindent Now,  we introduce  the following notation: If  $V:\Omega\to\mathbb{R}^{n}$ is a $C^{1,\kappa}(B_{r})$ vector field with $0<\kappa \leq 1$
$$||V||_{C^{0,\kappa}(B_{r}^{+})} ^{*} = ||V||_{L^{\infty}(B_{r}^{+})} + r^{\kappa} \cdot [V]_{C^{0,\kappa}(B_{r}^{+})},$$
$$[V]_{C^{0,\kappa}(B_{r}^{+})}=\sup\limits_{x,y\in B_{r}^{+}\atop {x\neq y}} \frac{|V(x)-V(y)|}{|x-y|^{\kappa}}.
$$
If $\varphi \in C^{1,\kappa}({B}_{r})$  where $\kappa\in(0,1]$, we denote 
$$ || \varphi ||_{C^{1, \kappa}(B_{r})}^{*}= ||\varphi||_{L^{\infty}(B_{r})} + r\cdot  ||\nabla \varphi||_{L^{\infty}(B_{r})} + r^{1+\kappa}\cdot [\nabla\varphi]_{C^{0,\kappa}(B_{r}^{+})}.$$

\begin{definition} Let $a, b, c>0.$ We set the following notation
\begin{equation}  \min\big\{a, b^{-}\big\}:= \left \{
\begin{array}{ll}\label{} 
 a & \textnormal { if } \quad a < b, \\\\
 \kappa &  \textrm{ for any } \kappa \in (0,b) \ \textnormal { if } \ b\leq a.
\end{array}
\right.
\end{equation}
Additionally, 
$$\min\big\{a,b,c^{-}\big\}:= \min\big\{\min\{a,b\big\}, c^{-}\big\}.$$
\end{definition}
\noindent We end up this section with the following 
\begin{definition}\label{definition-curved-domain}
Let $\Omega \subset {\mathbb{R}}^n$ a bounded domain and $q>n$. We say $\Omega$ is $W^{2,q}-$domain if for each point $x_0 \in \partial \Omega$ there corresponds a coordinate system $(x',x_{n})\in \mathbb{R}^{n-1}\times\mathbb{R}$ together with a $W^{2,q}$ function $h:\mathbb{R}^{n-1}\to\mathbb{R}$ and $r_{0} > 0$ such that
$$\Omega \cap B_{r_{0}}(x_0) = \Big\{x=(x',  x_n): x_n > h(x_1, x_2, \cdots, x_{n - 1}) \Big\}=\Omega\cap B_{r_{0}}(x_{0}).$$
In the case $q=\infty$, i.e,  $h\in C^{1,1}(\mathbb{R}^{n-1}),$ we say that $\Omega$ is a $C^{1,1}-$domain. This is equivalent to say that $\Omega$ satisfies a uniform interior and exterior ball condition $($see Lemma 2.2 in \cite{Nages}$).$

\end{definition}

\section{Main results} 
In this section, we present the main results on the paper. In the sequel, we use the notation
\begin{equation}\label{gamma-r-zero}
 \gamma_{R_{0}} := \max\Big\{\gamma, \gamma\cdot R_{0}\Big\},
 \end{equation}
 for the case where $\gamma$ is a nonnegative real constant.  We observe that in some results below, we allow $\gamma$ to be a function in the Lebesgue space $L^{q}$. Whenever this is the case, this will be indicated in the statements of the corresponding results. Unless explicitly stated otherwise, $\gamma$ is a nonnegative constant. We refer the reader to section 5 to check definitions and structural conditions for the PDEs below.
 We use the notation for $r>0$
 
 $$ \mathcal{A}_{\frac{r}{2}, r}:=\Big\{x\in \mathbb{R}^{n}; \ \frac{r}{2}<|x|<r \Big\} .$$
 \begin{proposition}[{\bf Inhomogeneous Pucci Barriers - I}]\label{Pucci-Barriers}  Let  us consider the constants $0\leq \gamma\leq \gamma_{0}, ~M\geq 0 $ and $0<r\leq R_{0}$. Assume $f\in L^{q}(\mathcal{A}_{r})$ with $q>n$. Then,  there exist a unique $L^{n}-$viscosity solutions in $C^{0}(\overline{\mathcal{A}_{r}})$ to the following Dirichlet problem 
%%%%%%%
 \begin{equation} \label{DP-positive-pucci-barrier-displaced}
 \left \{
    \begin{array}{rcll}
\mathcal{P}_{\gamma}^{-}[u]&=& f & \textrm{ in } \:\mathcal{A}_{\frac{r}{2}, r} \\\\
     u &=& 0& \textrm{ on } \partial B_{r}\\\\
    u &=& M & \textrm { on } \partial B_{\frac{r}{2}}.
    \end{array}
    \right.
\end{equation}
This solution is also a $L^{n}-$strong solution to (\ref{DP-positive-pucci-barrier-displaced}).\\

\noindent  Furthermore,  denoting $d_{r}(x)=dist(x,\partial B_{r}),$ we have \ $\forall x\in \overline{\mathcal{A}_{r/2,r}}$ that 

\begin{equation}\label{estimate-pucci-geometry-barrier}
\Bigg(A_{1}\frac{M}{r} - A_{2}\cdot r^{1-n/q}||f||_{L^{q}(\mathcal{A}_{r/2,r})}\Bigg)\cdot d_{r}(x) \leq u(x) \leq \Big(A_{3}\frac{M}{r} + A_{4}\cdot r^{1-n/q}||f||_{L^{q}(\mathcal{A}_{r/2,r})}\Big)\cdot d_{r}(x)
\end{equation}
\noindent Here,  $A_{1},A_{3}$ depend on $n,\lambda, \Lambda, \gamma_{R_{0}}$ and $A_{2}, A_{4}$ depend only on $n,q, \lambda, \Lambda, \gamma_{R_{0}}$ are positive universal constants.
 \end{proposition}
 \noindent Following exactly the same strategy of the proof of Proposition \ref{Pucci-Barriers}, we can obtain a ``symmetric" result for $\mathcal{P}_{\gamma}^{+}$
  \begin{proposition}[{\bf Inhomogeneous Pucci Barriers - II}]\label{Pucci-Barriers+}  Let  us consider the constants $0\leq \gamma\leq \gamma_{0}, ~M\geq 0 $ and $0<r\leq R_{0}$. Assume $f\in L^{q}(\mathcal{A}_{r})$ with $q>n$. Then,  there exist a unique $L^{n}-$viscosity solutions in $C^{0}(\overline{\mathcal{A}_{r}})$ to the following Dirichlet problem 
%%%%%%%
 \begin{equation} \label{DP-positive-pucci-barrier-displaced-+}
 \left \{
    \begin{array}{rcll}
\mathcal{P}_{\gamma}^{+}[v]&=& f & \textrm{ in } \:\mathcal{A}_{\frac{r}{2}, r} \\\\
     v &=& M& \textrm{ on } \partial B_{r}\\\\
    v &=& 0 & \textrm { on } \partial B_{\frac{r}{2}}.
    \end{array}
    \right.
\end{equation}
This solution is also a $L^{n}-$strong solution to (\ref{DP-positive-pucci-barrier-displaced}).\\

\noindent  Furthermore, by denoting ${d}_{r}^{*}(x)=dist(x,\partial B_{r/2}),$ we have \ $\forall x\in \overline{\mathcal{A}_{r/2,r}}$

\begin{equation}\label{estimate-pucci-geometry-barrier}
\Bigg(\overline{A}_{1}\frac{M}{r} - \overline{A}_{2}\cdot r^{1-n/q}||f||_{L^{q}(\mathcal{A}_{r/2,r})}\Bigg)\cdot d_{r}^{*}(x) \leq u(x) \leq \Big(\overline{A}_{3}\frac{M}{r} + \overline{A}_{4}\cdot r^{1-n/q}||f||_{L^{q}(\mathcal{A}_{r/2,r})}\Big)\cdot d_{r}^{*}(x)
\end{equation}
\noindent Here,  $\overline{A}_{1},\overline{A}_{3}$ depend on $n,\lambda, \Lambda, \gamma_{R_{0}}$ and $\overline{A}_{2}, \overline{A}_{4}$ depend only on $n,q, \lambda, \Lambda, \gamma_{R_{0}}$ are positive universal constants.
 \end{proposition}

 \begin{theorem}[{{\bf IHOL - Fully Nonlinear Case}}]\label{IHOL} Let $0\leq u\in C^{0}(\overline{B_{r}})$, $f\in L^{q}(B_{r})$ with $q>n$ and $0 < r \leq R_0$. Assume that $u \in {S}^*(\gamma,f)$ in $B_{r}$. Then, there exist positive and universal constants $C_{1}, C_{2}$ such that $ \forall x\in \overline{B_{r}}$
\begin{equation}\label{hopf-mean-value-inequality-FN}
u(x) \geq \Bigg(\frac{C_{1}}{r} u(0) - C_{2} \Big( r^{1 - n/q} ||f||_{L^{q}(B_{r})} \Big) \Bigg)\cdot dist(x,\partial B_{r}).
\end{equation}
Assume moreover the inwards $\frac{\partial u}{\partial\nu} (x_{0})$ unit normal derivative  exists at $x_{0}\in\partial B_{r}$ and $u(x_{0})=0.$ Then, 
\begin{equation}\label{BCN-Hopf} \frac{u(0)}{r} \leq C_3 \cdot \Bigg(  \frac{\partial u}{\partial \nu}(x_{0}) +  r^{1 - n/q}\cdot ||f||_{L^{q}(B_{r})}\Bigg).
\end{equation}
Here, $C_{1}, C_{2}$ and $C_{3}$ are positive universal constants depending only on $n, q, \lambda, \Lambda, \gamma_{R_{0}}.$
\end{theorem}
\begin{remark}\label{IHOL-supersolutions} In the Theorem \ref{IHOL} above, if condition $u\in {S}^{*}(\gamma;f)$ is replaced by $u\in \overline{S}(\gamma;f)$ in $B_{r}$ instead, then $u(0)$ can be replaced by the average value $|B_{r}|^{-n/\varepsilon}||\frac{u}{dist(x, \partial B_{r})}||_{L^{\varepsilon}(B_{r/2})}$ in both estimates (\ref{hopf-mean-value-inequality-FN}) and (\ref{BCN-Hopf}). This follows from Theorems 2 and 11 in \cite{Boyan-2}. As a consequence of that,  it follows that one can also replace $u(0)$ by $|B_{r}|^{-n/\varepsilon}||{u}||_{L^{\varepsilon}(B_{r/2})}$ since this is a smaller quantity when compared to $|B_{r}|^{-n/\varepsilon}||\frac{u}{dist(x, \partial B_{r})}||_{L^{\varepsilon}(B_{r/2})}$.

 This latter fact however has an much simpler proof. One can just follow exactly the proof of Theorem \ref{IHOL} as presented here using weak Harnack inequality in (\ref{harnack-in-proof-Hopf}) instead of Harnack inequality just replacing $u(0)$ by $||u||_{L^{\varepsilon}(B_{1/2})}$ in the proof (see Theorem  2.1 in \cite{BM-IHOL-PartI}). The proof is essentially a consequence of the geometry of the Pucci barriers.
\end{remark}

\begin{proposition} [{{\bf Boundary Ladyzhenskaya-Uraltseva Lipschitz type estimate}}] \label{bdry-lip-type-estimate-scaled-version}
Let $u \in C^{0}(\overline{B^{+}_{r}})$ and $f\in L^{q}(B^{+}_{r})$ with $q>n$ and $0 < r \leq R_0$. Assume that $u \in S^{*}(\gamma,f)$ in $B^{+}_{r}$. Then, there exists a positive universal constant $D_{1}>0$ such that
\begin{equation} \label{Lip estimate fully - scaled}
\vert u(x) \vert \leq D_{1} \Bigg( \frac{{\Vert u \Vert}_{L^{\infty}(B_r^+)}}{r} + r^{1 - n/q} ||f||_{L^{q}(B_{r}^{+})} \Bigg) \cdot x_n + \sup_{B'_r} \vert u \vert, \ \ \ \ \forall \ x \in \overline{B_{\frac{r}{2}}^{+}},
\end{equation}
 Moreover, this estimate is sharp and does not hold for $q=n.$ Here, $D_{1} = D_{1}(n, q, \lambda, \Lambda, \gamma_{_{R_{0}}}).$
\end{proposition}
\begin{remark}  We observe also that Proposition \ref{bdry-lip-type-estimate-scaled-version} extends some of the results of the classical paper of H. Berestycki, L. Caffarelli and L. Nirenberg \cite{BCN} (see Lemma 2.1 in \cite{BCN} for instance) to the case of unbounded RHS and more general (nonlinear) operators. 
\end{remark}
\begin{remark}\label{lips-sub-super} It follows from the proof of Proposition \ref{bdry-lip-type-estimate-scaled-version} the following (see section 8). Let $f\in L^{q}(B_{r}^{+})$ with $q>n$ and $u\in C^{0}(\overline{B_{r}^{+}}).$ Thus, for $u \in \underline{S}(\gamma, -|f|) $  in $ B_{r}^{+}$ we have 
$$ u(x) \leq D_{1}\Bigg( \frac{{\Vert u \Vert}_{L^{\infty}(B_r^+)}}{r} + r^{1 - n/q} ||f||_{L^{q}(B_{r}^{+})} \Bigg) \cdot x_n + \sup_{B'_r} u , \ \ \ \ \forall \ x \in \overline{B_{\frac{r}{2}}^{+}}.$$
\noindent  Analogously, for  $u \in \overline{S}(\gamma, |f|) $  in $ B_{r}^{+}$ we have
$$u(x) \geq - D_{1}\Bigg( \frac{{\Vert u \Vert}_{L^{\infty}(B_r^+)}}{r} + r^{1 - n/q} ||f||_{L^{q}(B_{r}^{+})} \Bigg) \cdot x_n - \inf_{B'_r} u , \quad \forall \ x \in \overline{B_{\frac{r}{2}}^{+}}.$$
\end{remark}
\begin{remark}[{\bf Carleson's estimate for nonnegative solutions}]\label{carleson-estimates-upgrade} In the Lipschitz type estimate provided by Proposition \ref{bdry-lip-type-estimate-scaled-version} if we additionally assume that $u$ is nonnegative and vanishes on the flat boundary $u=0$ in $B'_{r}$, the estimate (\ref{Lip estimate fully - scaled}) takes a sharper form and becomes
\begin{equation} \label{Lip estimate fully-scaled-Carleson}
 u(x) \leq {D_{1}} \Bigg( \frac{{u(\frac{r}{2}e_{n})}}{r} + r^{1 - n/q} ||f||_{L^{q}(B_{r}^{+})} \Bigg) \cdot x_n \ \ \ \ \forall \ x \in \overline{B_{\frac{r}{2}}^{+}}.
\end{equation}
This is consequence of the inhomogeneous version of the Carleson's estimate for nonnegative functions in $S^{*}(\gamma, f)$ with $f\in L^{n-\varepsilon_{0}}$ that is proven in a forthcoming work \cite{Braga-Moreira-Carleson}. 
Here, $\varepsilon_{0}>0$  ~is a Escuriaza type exponent. Again, Lipschitz type estimate (\ref{Lip estimate fully-scaled-Carleson}) is sharp in the sense it does not hold for $q=n.$ 
\end{remark}
\begin{remark}\label{Lipschitz-C^{1,1}} It follows from the proof of Proposition \ref{bdry-lip-type-estimate-scaled-version} that this result can be immediately extended to $C^{1,1}$ domains. So, let $\Omega\subset\mathbb{R}^{n}$ be a bounded $C^{1,1}$domain and $u\in C^{0}(\overline{\Omega})\cap S^{*}(\gamma;f)$ in $\Omega.$ Suppose $f\in L^{q}(\Omega)$ with $q>n$. Then, we have the following estimate with $\overline{D}_{1}=\overline{D}_{1}(n,q,\lambda,\Lambda,\gamma,\partial\Omega)$
$$\vert u(x) \vert \leq \overline{D}_{1} \Bigg( {{\Vert u \Vert}_{L^{\infty}(\Omega)}}+ ||f||_{L^{q}(\Omega)} \Bigg) \cdot dist(x,\partial \Omega) + \sup_{\partial\Omega} \vert u \vert, \ \ \ \ \forall \ x \in \overline{\Omega}.$$ 
Likewise, similar observations apply to the context of Remarks \ref{lips-sub-super} and \ref{carleson-estimates-upgrade}.
\end{remark}
\begin{remark} It follows from Theorem \ref{general-pointwise-boundary-krylov-general-boundary-data} (or even Corollary \ref{scaled-general-pointwise-krylov}) and Theorem \ref{global-krylov-w2q-domains} that in the case the function $u$ in Remark \ref{Lipschitz-C^{1,1}} and Proposition \ref{bdry-lip-type-estimate-scaled-version} vanishes on the boundary or flat boundary respectively, we can allow the drift term to be in $\gamma\in L^{q}$ with $q>n$. Moreover, we can allow $W^{2,q}$ domains instead of $C^{1,1}$. In this case, $D_{1}$ depends on $||\gamma||_{L^{q}(\Omega)}$ and $R_{0}^{1-n/q}||\gamma||_{L^{q}(B_{R_{0}}^{+})}$ respectively.

\end{remark}
\begin{proposition}[{{\bf Inhomogeneous Boundary Hopf-Ole\u{\i}nik principle on flat boundaries}}] \label{Boundary Behaviour ffb} Consider $0\leq u\in C^{0}(\overline{B}^{+}_{r})$ such that $u = 0$ in $B'_r$, $f\in L^{q}(B_{r})$ with $q>n$ and $0 < r \leq R_0$. Assume that $u \in S^{*}(\gamma,f)$ in $B^{+}_{r}$. Then,
\begin{equation} \label{hopf-principle-inequality-fully-scaled}
u(x) \geq \Bigg( D_{2} \frac{u \left( \frac{r}{2} e_n \right)}{r} - D_{3} \cdot r^{1 - n/q} ||f||_{L^{q}(B_{r}^{+})} \Bigg) \cdot x_n, \quad \forall x \in \overline{B}_{r/2}^{+}.
\end{equation}
\noindent Here,  $D_{2} = D_{2}(n, q, \lambda, \Lambda, \gamma_{_{R_{0}}}) \in (0, 1)$ and $D_{3} = D_{3}(n,q, \lambda, \Lambda, \gamma_{_{R_{0}}}) > 0$.
\end{proposition}
\begin{remark}\label{hopf-C1,1} As before, it follows from the proof of Proposition \ref{Boundary Behaviour ffb} that this result can extended to $C^{1,1}$ domains in the following way: Let $\Omega\subset\mathbb{R}^{n}$ be a bounded $C^{1,1}-$domain and $0\leq u\in C^{0}(\overline{\Omega})\cap S^{*}(\gamma;f)$ in $\Omega$ where $f\in L^{q}(\Omega)$ with $q>n.$  Suppose $u\equiv 0$ along $\partial\Omega.$ Then, for some $x_{0}\in\Omega$
$$u(x) \geq \Bigg(\overline{D}_{2} \cdot u(x_{0}) - \overline{D}_{3} \cdot ||f||_{L^{q}(\Omega)} \Bigg) \cdot dist(x,\partial\Omega), \quad \forall x \in \Omega.$$
Here,  $\overline{D}_{2} = \overline{D}_{2}(n, q, \lambda, \Lambda, \gamma,\partial\Omega) \in (0, 1)$ and $\overline{D}_{3} = \overline{D}_{3}(n,q, \lambda, \Lambda, \gamma, \partial\Omega) > 0$.
\end{remark}
\begin{theorem}[{{\bf Boundary $C^{1,\alpha_{0}}-$Krylov-Ladyzhenskaya-Uraltseva type estimate}}] \label{boundary krylov thm Lq version}
Let $u \in C^{0}({\overline{B}}_{r}^{+}) \cap S^*(\gamma, f)$ in $B_{r}^{+}$ where $f\in L^q(B_{r}^+)$ with $q > n$ and $r\leq R_{0}.$ Assume that $u$ vanishes on $B'_{r}$. Then, there exists a unique H\"older continuous function $A$ on $B'_{r/2}$ and $\alpha_0 \in (0, 1)$ such that for all $x_0 \in B'_{r/2}$ and $x \in B_{3r/4}^{+}$ we have
\begin{equation} \label{taylor-general-krylov}
\left\vert u(x) - A(x_0) \cdot x_n \right\vert \leq r^{- \alpha_0} E_{1} \Bigg(\frac{||u||_{L^{\infty}(B_{r}^{+})}}{r}   + r^{1-n/q}\cdot ||f||_{L^{q}(B_{r}^{+})}\Bigg) \vert x-x_{0} \vert^{\alpha_0}x_{n},
\end{equation}
and 
\begin{equation} \label{holder-estimate-krylov}
 {\Vert A\Vert}_{C^{0, \alpha_0}(B_{r/2}^{'})}^{*} \leq E_{1}  \Bigg(\frac{||u||_{L^{\infty}(B_{r}^{+})}}{r}   + r^{1-n/q}\cdot ||f||_{L^{q}(B_{r}^{+})}\Bigg).
 \end{equation}
 Furthermore, 
 \begin{equation}\label{krylov-uraltseva-oscillation-estimate}
\Big\| \frac{u}{x_{n}}\Big\|_{C^{0,\alpha_{0}}(B_{r/2}^{+})}^{*}\leq E_{1} \Bigg(\frac{||u||_{L^{\infty}(B_{r}^{+})}}{r}   + r^{1-n/q}\cdot ||f||_{L^{q}(B_{r}^{+})}\Bigg).
 \end{equation}
Precisely, $\alpha_0 = \alpha_0(n, \lambda, \Lambda, q)$ and $E_{1} = E_{1}(n, q, \lambda, \Lambda, \gamma_{R_{0}})>0$.
\end{theorem}
\begin{remark} The vector field $A: B'_{r/2} \to {\mathbb{R}}^n$ above should be thought as the gradient of $u$ along $B'_{r/2}.$
\end{remark}
\begin{remark} In the case we consider nonnegative functions in Theorem \ref{boundary krylov thm Lq version}, we can indeed replace the expression inside the parenthesis in (\ref{taylor-general-krylov}), (\ref{holder-estimate-krylov}) and (\ref{krylov-uraltseva-oscillation-estimate}) by the following one

$$  \bigg(\frac{{u(\frac{r}{2}e_{n})}}{r} + r^{1 - n/q} ||f||_{L^{q}(B_{r}^{+})} \bigg).$$

\noindent As before, this follows from the results in \cite{Braga-Moreira-Carleson}. In any case (independent of the sing of $u$),  we can replace the expression $|x|^{\alpha_0}\cdot x_{n}$ by $|x|^{1+\alpha_0}$ in (\ref{taylor-general-krylov}). In fact, in this case, the new taylor expansion in (\ref{taylor-general-krylov}) holds for all $ x_{0}\in B'_{1/2}$ and ~for all $ x\in B_{1}^{+}$ (see Remark \ref{extension-of-C-{1,alpha}-inequality-to-B1}). 
\end{remark}
\begin{definition}\label{definition-alpha00} We denote $\alpha_{00}\in (0,1)$ the exponent $\alpha_{0}$ of Theorem \ref{boundary krylov thm Lq version} in the case where $f\equiv 0,$ i.e, 
 \begin{equation}\label{alpha-zero}
\alpha_{00} = \alpha_{00}(n, \lambda, \Lambda) \in (0,1).
\end{equation}
 \end{definition}
\begin{remark}[{\bf Recovering Phargm\'en-Lindel\"of type result in half spaces}]\label{PLT}Let $u\in C^{0}(\mathbb{R}_{+}^{n})\cap S_{\lambda,\Lambda}(0)$ in $\mathbb{R}_{+}^{n}$ that vanishes on the flat boundary, i.e, $u=0$ in $\partial\mathbb{R}_{+}^{n}.$ Let $\alpha_{00}\in (0,1)$ given in Definition \ref{definition-alpha00}. Suppose that there exists a number $0<\beta<1+\alpha_{00}$ such that 
\begin{equation}\label{growth-condition-PL}
 |u(x)| \leq C_{0}|x|^{1+\beta}\ \forall x\in \mathbb{R}_{+}^{n}.
 \end{equation}
Then, 
\begin{equation}\label{classification}
u(x)=u(e_{n})\cdot x_{n} \quad \forall x\in \mathbb{R}_{+}^{n}.
\end{equation}
Additionally, if $u$ above is nonnegative (no growth condition apriori), it has to be of the form (\ref{classification}). We give a proof  (see the end of section 10) of these facts inspired by the very nice ideas of \cite{R-O-Se-2} in the nonlinear integral-differential operators context. For the nonnegative case, we simply observe that from the results of \cite{Braga-Moreira-Carleson}, we have $|u(x)|\leq C|x|$ for all $x\in\mathbb{R}_{+}^{n}.$ Thus, we can just take $\beta=1$ in the statement.  We point out here that this fact was also proven in \cite{boyan} in a more general situation. The context there was done for domains in conical shape. We are thankful to B. Sirakov that brought this fact to our attention and kindly explain to us the details of their (different) proof in \cite{boyan}. 
\end{remark}

We now state of the main result of this paper

\begin{theorem}[{{\bf Krylov boundary gradient type estimate under $C^{1,Dini}$ pointwise boundary regularity}}] \label{general-pointwise-boundary-krylov-general-boundary-data}
Let $u \in C^0(\overline{B}_1^+) \cap S^*(\gamma, f)$ in $B_1^+$ where $\gamma, f \in L^{q}(B_1^+)$ with $q > n$ and $\beta_{*}:=\min\{1-n/q, \alpha_{00}^{-} \}$. Assume the boundary data  $\varphi=u_{\mid_{B'_1}} $ satisfies $\varphi\in C^{1,\omega}(0)$ where $\omega\in \mathcal{DMC}(Q,\beta_{*})$ and let $L$ be the Taylor's polynomial of $\varphi$ at zero. Then, there exists a unique $\Psi_0 \in {\mathbb{R}}$ such that for all $x=(x',x_{n})\in B_{\delta_{\omega}}^+$,
\begin{equation}\label{eq-1-main-pointwise-result}
\left\vert u(x) - L(x',0) - \Psi_0 \cdot x_n \right\vert \leq T_{0}\left( {\Vert u \Vert}_{L^{\infty}(B_1^+)} +||\varphi||_{C^{1, \omega}}(0) + {\Vert f \Vert}_{L^{q}(B_1^+)} \right)  {\vert x \vert}\vartheta(|x|),
\end{equation}
\begin{equation}\label{eq-2-main-pointwise-result}
\left\vert \Psi_0 \right\vert \leq T_{0} \left( {\Vert u \Vert}_{L^{\infty}(B_1^+)} +||\varphi||_{C^{1, \omega}}(0) + {\Vert f \Vert}_{L^{q}(B_1^+)} \right),
\end{equation}
\noindent where
$$\vartheta(t):= t^{\beta_{*}}+\int_{0}^{t}\frac{\omega(s)}{s}ds \quad \textnormal { for  }\ t\in[0,\delta_{\omega}].$$ 
In particular, there exists the normal derivative
\begin{equation}\label{normal-derivative}
D_{x_{n}}u(0) := \lim\limits_{t\to 0^{+}} \frac{u(te_{n})-u(0)}{t}=\Psi_{0}.
\end{equation}
Here, the universal constant $T_{0}=T_{0}(n,q, \lambda, \Lambda, ||\gamma||_{L^{q}(B_{1}^{+})}, \alpha_{00}, \delta_{\omega}^{*}, \beta_{*},Q, \delta_{\omega},  \int_{0}^{\delta_{\omega}}\omega(s)s^{-1}ds).$ 
\end{theorem}
\begin{remark}\label{dependence-T0}
In fact, $T_{0}$ in the Theorem \ref{general-pointwise-boundary-krylov-general-boundary-data} above can be taken of the form 
$$T_{0}=J_{1}\cdot \delta_{\omega}^{-(1+\beta_{*})}\Big(1+||\gamma||_{L^{q}(B_{1}^{+})}\Big)\bigg(1+||\gamma||_{L^{q}(B_{1}^{+})}^{\frac{2+\beta_{*}}{1-{n}/{q}}}\bigg)$$
where  $J_{1}=J_{1}(n,q, \lambda, \Lambda, \alpha_{00}, \delta_{\omega}^{*}, \beta_{*},Q,  \int_{0}^{1}\omega(s)s^{-1}ds)>0$ is universal (see proof of Theorem \ref{general-pointwise-boundary-krylov-general-boundary-data}).
\end{remark}
\begin{remark} It is long known that even for Harmonic functions in half spaces the ``Dini character" of the boundary data is required to obtain at least a finite gradient on boundary (see \cite{widman}, Remark 1).
\end{remark}
\begin{remark}[{\bf Comments on the $\beta_{*}$ compatibility condition}]\label{comments-beta-compatibility} We observe that the condition ``$\varphi\in C^{1,\omega}(0)$ where $\omega\in \mathcal{DMC}(Q,\beta_{*})$" in the Theorem \ref{general-pointwise-boundary-krylov-general-boundary-data} is a natural one. Indeed, first, if the boundary data $\varphi$ is say $C^{1,\beta_{0}}$ and the RHS is in $L^{q}$ with $q>n$, we know that solution is $C^{1,\beta}$ (along the boundary) only for $\beta\leq1-n/q$ and $\beta<\alpha_{00}$. This way, this condition imposed on the modulus of continuity of the boundary data in Theorem  \ref{general-pointwise-boundary-krylov-general-boundary-data} is capturing this obstruction (see detailed discussion on the H\"older continuous boundary data in the sequel). Second, this condition is weaker than requiring Dini continuity of $\omega$ together with $\omega(t)/t^{\alpha}$ is decreasing (see Lemma \ref{properties-modulus-of-continuity} item $iii))$. This type of condition on the monotonicity of $\omega(t)/t^{\alpha}$ together with Dini continuity appeared in classical previous works on this subject (i.e, boundary regularity with $C^{1,Dini}$ boundary data) as one can see in the papers of G. Lieberman (\cite{Lieberman-CPDE}), K. Widman (\cite{widman}), M. Borsuk (\cite{Borsuk}) and M. Borsuk and V. Kondratiev's book \cite{Borsuk-Kondratiev} for linear equations in nondivergence and divergence forms. Moreover, our condition here seems to be weaker than the one imposed in the works if J. Kovats \cite{Kovats-1, Kovats-2} and more natural since it resembles the ones in K. Widman and M. Borsuk works. 
\end{remark}

\begin{corollary}\label{smooth-bdry-data}
Let $u \in C^0(\overline{B}_{1}^+) \cap S^*(\gamma, f)$ in $B_{1}^{+}$ with $\gamma, f \in L^{q}(B_{1}^+)$ where $q > n$. Assume the boundary data  $\varphi=u_{\mid_{B'_{1}}} \in C^{1, \omega}(B'_{1})$ with $\omega\in \mathcal{DMC}(Q,\beta_{*})$ where $\beta_{*}:=\min\{1-n/q, \alpha_{00}^{-}\}.$ Thus, there exist a unique vector field, $A: B'_{1/2} \to {\mathbb{R}}^n$  such that for every $x_0 \in B'_{1/2}$ \\
\begin{equation} \label{main-estimate-boundary-data-smooth}
\Big\vert u(x) - u(x_0) - A(x_0) (x - x_0) \Big\vert \leq T_{1} \cdot M \cdot \vert x - x_0 \vert\vartheta(|x-x_{0}|)\ \textnormal { for } \ |x-x_{0}|\leq \min\{\delta_{\omega},1/2 \},
\end{equation}
%\begin{equation*}
%M := \Bigg(\frac{||u||_{L^{\infty}(B_{r}^{+})}}{r} + \frac{|| \varphi ||_{C^{1, \beta_{0}}(B'_r)}^{*}}{r}   + r^{1-n/q}\cdot ||f||_{L^{q}(B_{r}^{+})}\Bigg).
%\end{equation*}
\begin{equation} \label{gradient Holder estimate at the boundary in half balls dini}
 {\Vert A \Vert}_{C^{0, \vartheta}(\overline{B'}_{1/2})}\leq T_{1} \cdot M,
\end{equation}
where
$$\vartheta(t):= t^{\beta_{*}}+\int_{0}^{t}\frac{\omega(s)}{s}ds \quad \textnormal { for  }\ t\in[0,\delta_{\omega}] \quad \textnormal { and }$$ 

$$  M:=\left( {\Vert u \Vert}_{L^{\infty}(B_1^+)} +{\Vert f \Vert}_{L^{q}(B_1^+)} + ||\varphi||_{C^{1, \omega}(B'_{1})}  \right).$$
Here $T_{1}=T_{1}\Big(n,q, \lambda, \Lambda, \gamma, \alpha_{00},\delta_{\omega}^{*}, \beta_{*},Q, \delta_{\omega},  \int_{0}^{\delta_{\omega}}\omega(s)s^{-1}ds, \min\{\delta_{\omega},1/2\}, \omega\big(\min\{1/8, \delta_{\omega}/4\}\big)\Big).$
\end{corollary}
\begin{remark}\label{dependence-T1}The constant $T_{1}$ in Corollary \ref{smooth-bdry-data} above can be taken of the form
\begin{equation}\label{invocada-dependence}
T_{1}:=C\bigg(1+\frac{1}{\omega(\delta_{*}/4)}\bigg)\big(T_{0}\delta_{*}^{-(1+\beta_{*})}+1)
\end{equation}
where $C>2$ is a dimensional constant and $\delta_{*}=\min\{\delta_{\omega},1/2\}.$ The constant $T_{0}$ is the one given in Theorem \ref{general-pointwise-boundary-krylov-general-boundary-data} or Remark \ref{dependence-T0} with $\delta_{\omega}$ replaced by $\delta_{*}$. For details, see the proof of Corollary \ref{smooth-bdry-data}.
\end{remark}
Next, we state the global version of Krylov's result in $W^{2,q}-$domains and $C^{1,Dini}$ boundary data 
\begin{theorem}[{{\bf Global Krylov boundary gradient type estimate in $W^{2,q}-$domains $(q>n)$}}] \label{global-krylov-w2q-domains}
Let $u \in C^0(\overline{\Omega}) \cap S^*(\gamma, f)$ in $\Omega$ where $\gamma, f \in L^q(\Omega)$ with $q > n$ and $\Omega \subseteq {\mathbb{R}}^n$ be a bounded  $W^{2,q}$-domain. There exists a universal $\kappa_{\partial\Omega}=\kappa_{\partial\Omega}(n,\lambda, \Lambda, \partial\Omega)\in(0,1)$ such that if $u_{\mid_{\partial \Omega}} = \varphi \in C^{1, \omega}(\partial \Omega)$ where $\omega\in \mathcal{DMC}(Q,\beta_{*})$ for  
$${\beta}_{*} := \min \Big\lbrace 1 - \frac{n}{q}, \kappa_{\partial\Omega}^- \Big\rbrace, $$

\noindent  then $u \in C^{1, \vartheta}(\partial \Omega)$. More precisely, there exist a unique vector field $A : \partial \Omega \to {\mathbb{R}}^{n}$ and positive universal constants $T_{2}$ and $r_{0}=r_{0}(\partial\Omega)\leq \delta_{\omega}$ such that for every $b \in \partial \Omega\cap B_{r_{0}}(x_{0})$ and $x\in B_{r_{0}}(x_{0})\cap\Omega$ we have
\begin{equation} \label{main estimate w2q}
\Big\vert u(x) - u(b) - A (b) (x - b) \Big\vert \leq T_{2} \cdot \Big( {\Vert \varphi \Vert}_{C^{1, \omega}(\partial \Omega)} + {\Vert f \Vert}_{L^q(\Omega)} \Big) \vert x - b \vert\vartheta(|x-b|).
\end{equation}
Moreover,
\begin{equation} \label{gradient Holder estimate at the boundary w2q}
{\Vert A \Vert}_{C^{0, \vartheta}(\partial \Omega)} \leq T_{2} \cdot \Big( {\Vert \varphi \Vert}_{C^{1, \omega}(\partial \Omega)} + {\Vert f \Vert}_{L^q(\Omega)} \Big).
\end{equation}
Here, 
$$\vartheta(t):= t^{\beta_{*}}+\int_{0}^{t}\frac{\omega(s)}{s}ds \quad \textnormal { for  }\ t\in \left[0, {\delta_{\omega}} \right].$$ 
\noindent  Additionally,  
$$T_2 = T_2\Big(n,q,\lambda, \Lambda, ||\gamma||_{L^{q}(\Omega)}, \delta_{\omega}^{*}, \delta_{\omega}, \beta_{*}, Q, \int_{0}^{\delta_{\omega}}\omega(s)s^{-1}ds, \min\{1/2, \delta_{\omega}/K_{\partial\Omega}\}, w\big(\min\{1/8, \delta_{\omega}/4K_{\partial\Omega}\})\Big)$$
where $K_{\partial\Omega}>1$ depends on $\partial\Omega.$
\end{theorem}
 \begin{remark} The constant $T_{2}$ in Theorem \ref{global-krylov-w2q-domains} above is of the form appearing in (\ref{invocada-dependence}) with  $\delta^{*}$ is replaced by $\min\{1/2, \delta_{\omega}/4K_{\partial\Omega}\}$ where as above $K_{\partial\Omega}>1$.
 \end{remark}
\begin{remark}[{\bf Krylov's boundary type estimates for all the coefficients unbounded}] We observe that all the results above (and the following $C^{1,\alpha}$ counterparts below) work for the classes $S^{*}(\gamma, \sigma, f)$ that has the zeroth order term $\sigma\in L_{+}^{q},$  drift term $\gamma\in L_{+}^{q}$ and RHS $f\in L^{q}$ for $q>n$~(see Remark \ref{general-classes}).
\end{remark}
From the previous Theorem, we conclude the existence of the normal derivative on a smooth boundary for $C^{1,Dini}$ boundary data. Thus,  we conclude immediately from Theorem \ref{IHOL} the following Corollary.
\begin{corollary}[{{\bf IHOL under pointwise $C^{1,Dini}-$boundary regularity}}]\label{hopf-C-1-alpha-tangentially} Let  $0\leq u\in C^{0}(\overline{B_{r}})$, $f\in L^{q}(B_{r})$ with $q>n$ and $0\leq r\leq R_{0}.$ Assume that $u \in S^{*}(\gamma,f)$ in $B_{r}$ with $x_{0}\in \partial B_{r}, u(x_{0})= 0$ and also that the boundary data $\varphi=u_{\mid_{B'_{r}}} \in C^{1, \omega}(x_{0})$ with
$$\omega\in \mathcal{DMC}(Q,\beta_{*}) \quad \textnormal {where} \quad \beta_{*}:=\min\{1-n/q, \alpha_{00}^{-}\}.$$
Then, there exists the inner normal derivative $\frac{\partial u}{\partial \nu}(x_{0})$ and the following estimate holds
\begin{equation}\label{hopf+krylov}
\frac{u(0)}{r} \leq T_{3} \cdot \Bigg(  \frac{\partial u}{\partial \nu}(x_{0}) + r^{1-n/q}||f||_{L^{q}(B_{r})}\Bigg).
\end{equation}
\noindent Here, $T_{3}$ is a positive universal constant that depends only on $n, q, \lambda, \Lambda, \gamma_{R_{0}}$. 
\end{corollary}

\subsection{Scaled H\"older regularity versions of the previous results} In what follows, we present the statements of the previous results in a scaled form for the case boundary data is $C^{1,\beta_{0}}.$ They may be useful in some circumstances to the readers and they also make contact with the estimates with the zero boundary data case estimates (\ref{taylor-general-krylov}) and (\ref{holder-estimate-krylov}).  Even for the case where boundary data is $C^{1,\beta_{0}}$ these estimates are new, since they involve unbounded coefficients $(\gamma, f).$ The relevant observation here is that if $\varphi\in C^{1,\beta_{0}}$ and $\beta_{**}:=\min\{1-n/q, \beta_{0}, \alpha_{00}^{-}\}$ then we clearly have $\varphi\in C^{1,\beta_{**}}.$ This way (see example \ref{example-holder}) the modulus of continuity $\omega(t)=t^{\beta_{**}} \in \mathcal{DMC}(1, \beta_{*})$ where $\beta_{*}=\min\{1-n/q,\alpha_{00}^{-}\}$ since $\beta_{*}\geq \beta_{**}.$ The next results follow immediately from the previous $C^{1,Dini}$ estimates and the scaling Remark \ref{scaling-remark}.

\begin{corollary}[{\bf H\"older scaled version of Theorem \ref{general-pointwise-boundary-krylov-general-boundary-data}}]\label{scaled-general-pointwise-krylov} Let  $ u \in C^0(\overline{B}_{r}^+) \cap S^*(\gamma, f) $  in $B_{r}^{+}$ where $f \in L^{q}(B_{r}^{+})$ with $q > n$ and $0\leq r\leq R_{0}$. Let $ \beta_{*}=\min\{1-n/q, \beta_{0}, \alpha_{00}^{-}\}$ and assume that $\varphi=u_{\mid_{B'_r}}\in C^{1,\beta_{0}}(0)$. Then $u\in C^{1,\beta_{*}}(0)$. More precisely, the estimates (\ref{eq-1-main-pointwise-result}) and (\ref{eq-2-main-pointwise-result}) become respectively
\begin{equation}\label{eq-1-main-pointwise-result-scaled}
\left\vert u(x) - L(x',0) - \Psi_0 \cdot x_n \right\vert \leq \overline{T_{0}}\cdot r^{-(1+\beta_{*})}\cdot M \cdot {\vert x \vert}^{1 + \beta_{*}} \quad \forall x\in B_{r}^{+}, 
\end{equation}
\begin{equation}\label{eq-2-main-pointwise-result-scaled}
r\cdot \left\vert \Psi_0 \right\vert \leq \overline{T_{0}}\cdot M,
\end{equation}
$$ M:=\left( {\Vert u \Vert}_{L^{\infty}(B_r^+)} + r^{2-n/q}{\Vert f \Vert}_{L^{q}(B_r^+)} + ||\varphi||_{C^{1, \beta_{0}}}^{*}(0)  \right). 
$$
Here, $\overline{T_{0}}=\overline{T_{0}}(n,q, \lambda, \Lambda, \alpha_{00}, \beta_{0}, R_{0}^{1-n/q}||\gamma||_{L^{q}(B_{R_{0}}^{+})})$ is a universal constant.
\end{corollary}
\begin{corollary}[{\bf H\"older scaled version of Corollary \ref{smooth-bdry-data}}]\label{scaled-general-pointwise-krylov} Let  $ u \in C^0(\overline{B}_{r}^+) \cap S^*(\gamma, f) $  in $B_{r}^{+}$ where $f \in L^{q}(B_{r}^{+})$ with $q > n$ and $0\leq r\leq R_{0}$. Let $ \beta_{*}=\min\{1-n/q, \beta_{0}, \alpha_{00}^{-}\}$ and assume that $\varphi=u_{\mid_{B'_r}}\in C^{1,\beta_{0}}(B'_{r})$. Then $u\in C^{1,\beta_{*}}(B_{r}^{+})$. More precisely, there exists a unique vector field, $A:B'_{r}\to\mathbb{R}^{n}$ such that the estimates (\ref{main-estimate-boundary-data-smooth}) and (\ref{gradient Holder estimate at the boundary in half balls dini}) become respectively
\begin{equation} \label{main estimate}
\Big\vert u(x) - u(x_0) - A(x_0) (x - x_0) \Big\vert \leq \overline{T_{1}} \cdot r^{-(1+ {\beta}_{*})} \cdot M \cdot \vert x - x_0 \vert^{1 + {\beta}_{*}},
\end{equation}
\begin{equation} \label{gradient Holder estimate at the boundary in half balls}
r\cdot {\Vert A \Vert}_{C^{0, \beta_{*}}(B'_{1/2})}^{*} \leq \overline{T_{1}} \cdot M,
\end{equation}
where
$$  M:=\left( {\Vert u \Vert}_{L^{\infty}(B_{r}^{+})} + r^{2-n/q}{\Vert f \Vert}_{L^{q}(B_{r}^{+})} + ||\varphi||_{C^{1, \beta_{0}}(B'_{r})}^{*}  \right).$$
Here $\overline{T_{1}}=\overline{T_{1}}(n,q, \lambda, \Lambda, \alpha_{00}, \beta_{0},R_{0}^{1-n/q}||\gamma||_{L^{q}(B_{R_{0}}^{+})})$ is a universal constant.
\end{corollary}
\begin{theorem}[{{\bf Global Krylov boundary H\"older gradient type estimate in $W^{2,q}-$domains $(q>n)$}}] \label{global-krylov-w2q-domains-holder}
Let $u \in C^0(\overline{\Omega}) \cap S^*(\gamma, f)$ in $\Omega$ where $\gamma, f \in L^q(\Omega)$ with $q > n$ and $\Omega \subseteq {\mathbb{R}}^n$ be a bounded  $W^{2,q}$-domain. There exists a universal $\kappa_{\partial\Omega}=\kappa_{\partial\Omega}(n,\lambda, \Lambda, \partial\Omega)\in(0,1)$ such that if $u_{\mid_{\partial \Omega}} = \varphi \in C^{1, \beta_{0}}(\partial \Omega)$ and ${\beta}_{*} := \min \Big\lbrace 1 - \frac{n}{q}, \beta_{0}, \kappa_{\partial\Omega}^- \Big\rbrace, $

\noindent  then $u \in C^{1, \beta_{*}}(\partial \Omega)$. Precisely, there exist a unique vector field $A : \partial \Omega \to {\mathbb{R}}^{n}$ and positive universal constants $\overline{T_{2}}$ and $r_{0}=r_{0}(\partial\Omega)\leq \delta_{\omega}$ such that for every $b \in \partial \Omega\cap B_{r_{0}}(x_{0})$ and $x\in B_{r_{0}}(x_{0})\cap\Omega$ we have
\begin{equation} \label{main estimate w2q}
\Big\vert u(x) - u(b) - A (b) (x - b) \Big\vert \leq \overline{T_{2}} \cdot \Big( {\Vert \varphi \Vert}_{C^{1, \beta_{0}}(\partial \Omega)} + {\Vert f \Vert}_{L^q(\Omega)} \Big) \vert x - b \vert^{1+\beta_{*}}.
\end{equation}
Moreover,
\begin{equation} \label{gradient Holder estimate at the boundary w2q}
{\Vert A \Vert}_{C^{0, \beta_{*}}(\partial \Omega)} \leq \overline{T_{2}} \cdot \Big( {\Vert \varphi \Vert}_{C^{1, \beta_{0}}(\partial \Omega)} + {\Vert f \Vert}_{L^q(\Omega)} \Big).
\end{equation}
Here,  $\overline{T_2} = \overline{T_2}(n,q,\lambda, \Lambda, ||\gamma||_{L^{q}(\Omega)}, \beta_{0}, \kappa_{\Omega}, \partial\Omega).$
\end{theorem}
\begin{remark} $\overline{T_{1}},\overline{T_{2}},\overline{T_{3}}$ have similar structural dependence that the corresponding $T_{1},T_{2}, T_{3}$ in the Dini case.  

\end{remark}

\section{Structural Conditions for the PDEs}
We now introduce the structural conditions for the PDEs that appear in this paper. We start by recalling the Pucci extremal operators. Let us denote $\mathcal{S}^{n\times n}$ the space of symmetric matrices of order $n$. For $0<\lambda\leq\Lambda$, ~~$\mathcal{M}_{\lambda, \Lambda}^{\pm}\colon \mathcal{S}^{n\times n}\to\mathbb{R}$ are given by

\begin{equation}
\mathcal{M}_{\lambda, \Lambda}^{-}(M) = \lambda\cdot\sum\limits_{e_{i}>0} e_{i} + \Lambda\cdot\sum\limits_{e_{i}<0} e_{i}, \quad \mathcal{M}_{\lambda, \Lambda}^{+}(M) = \Lambda\cdot\sum\limits_{e_{i}>0} e_{i} + \lambda\cdot\sum\limits_{e_{i}<0} e_{i}, 
\end{equation}
\noindent where $e_{i}$ are the eigenvalues of $M$.  We recall that 
\begin{equation}\label{Pucci-attained}
\mathcal{M}_{\lambda, \Lambda}^{-}(M)=\inf\limits_{M\in\mathcal{A}_{\lambda, \Lambda}}Trace(AM), \quad \mathcal{M}_{\lambda, \Lambda}^{+}(M)=\sup\limits_{M\in\mathcal{A}_{\lambda, \Lambda}}Trace(AM)
\end{equation}
 \begin{equation}\label{A-lambda-Lambda}
 \mathcal{A}_{\lambda, \Lambda}:=\Big\{A\in\mathcal{S}^{n\times n}; ~~\lambda I_{n} \leq A\leq \Lambda I_{n} \Big\}.
 \end{equation}
 It is easy to verify that the infimum and supremum above are attained. Also for $\gamma\geq 0$ a measurable function, we define the Pucci operators $\mathcal{P}_{\lambda, \Lambda, \gamma}^{\pm}\colon\mathcal{S}^{n\times n}\times\mathbb{R}^{n}\times\Omega\to\mathbb{R}$ are defined by 
\begin{equation}
\mathcal{P}_{\lambda, \Lambda, \gamma}^{-}(M,p,x) =\mathcal{M}_{\lambda, \Lambda}^{-}(M)-\gamma(x)|p|, \quad \mathcal{P}_{\lambda, \Lambda, \gamma}^{+}(M,p,x) =\mathcal{M}_{\lambda, \Lambda}^{+}(M)+\gamma(x)|p|.\\
\end{equation}

Here we often make use of scaling arguments. These scalings in general affects the function $\gamma$ of the Pucci operators $\mathcal{P}_{\lambda, \Lambda, \gamma}^{\pm}$, but leave the parameters $\lambda, \Lambda$ untouched. For this reason and to simplify notation, we  denote the Pucci operators $\mathcal{P}_{\lambda, \Lambda, \gamma}^{\pm}$ by $\mathcal{P}_{\gamma}^{\pm}$  to keep track of these  changes. Whenever clarification becomes necessary, we mention all the ellipticity constants explicitly. To simplify matters, we also make use of the following notation, (we drop the dependence on $\Omega$ in the symbols since the context is clear)

$$\mathcal{P}_{\gamma}^{\pm}[u](x)=\mathcal{P}_{\lambda, \Lambda, \gamma}^{\pm}[u](x):= \mathcal{P}_{\lambda, \Lambda, \gamma}^{\pm}(D^{2}u(x), \nabla u(x)),$$ 
 
$$\overline{S}(\gamma, f) = \overline{S}(\lambda, \Lambda, \gamma, f) =\Big\{ u\in C^{0}(\Omega); \:\mathcal{P}_{\gamma}^{-}[u](x)\leq f(x) \: \textnormal { in } \: \Omega \: \textnormal { in the } L^{n}-\textnormal{viscosity sense }  \Big\}, $$

$$\underline{S}(\gamma, f) = \underline{S}(\lambda, \Lambda, \gamma, f) =\Big\{ u\in C^{0}(\Omega); \:\mathcal{P}_{\gamma}^{+}[u](x)\geq f(x) \: \textnormal { in } \: \Omega \: \textnormal { in the } L^{n}-\textnormal{viscosity sense }\Big\},$$

$$ S(\gamma, f) = \underline{S}(\gamma, f) \cap \overline{S}(\gamma, f), \quad  S^{*}(\gamma, f) =  \underline{S}(\gamma, -|f|) \cap \overline{S}(\gamma, |f|).$$\vspace{1mm}

As said before, $\gamma$ sometimes denotes a nonnegative constant and sometimes a nonnegative measurable  function. In the latter case, this will be always indicated in the context or in the statements of the corresponding results. So, unless indicated otherwise, $\gamma$ is a nonnegative constant.
\begin{remark}[{\bf $L^{n}-$viscosity solutions}]\label{type-of-viscosity-solution-involved} We recall that $u\in C^{0}(\Omega)$ is a solution to $\mathcal{P}_{\gamma}^{+}[u]\geq f$ in $\Omega$ in the $L^{n}-$viscosity sense if for any $\phi\in W_{loc}^{2,n}(\Omega)$ such that $u-\phi$ has a local maximum at $x_{0}\in\Omega$ we have 

$$ess\limsup\limits_{x\to x_{0}} \Big(\mathcal{P}_{\gamma}^{+}[\varphi] - f(x)\Big) = ess\limsup\limits_{x\to x_{0}} \Big(\mathcal{M}_{\lambda, \Lambda}^{+}(D^2\phi(x))+\gamma(x)|\nabla\phi(x)| - f(x) \Big)\geq 0.$$

\noindent We say that  $\mathcal{P}_{\gamma}^{-}[u]\leq f$ in $\Omega$ in the $L^{n}-$viscosity sense if for any $\phi\in W_{loc}^{2,n}(\Omega)$ such that $u-\phi$ has a local mimimum at $x_{0}\in\Omega$ we have 
$$ess\liminf\limits_{x\to x_{0}} \Big(\mathcal{P}_{\gamma}^{-}[\varphi] - f(x)\Big) = ess\liminf\limits_{x\to x_{0}} \Big(\mathcal{M}_{\lambda, \Lambda}^{-}(D^2\phi(x))-\gamma(x)|\nabla\phi(x)| - f(x) \Big)\leq 0.$$

Unless otherwise explicitly stated, all the viscosity concepts in this paper are considered in the $L^{n}-$viscosity sense (even when $f\in L^{\infty}(\Omega)$). We point out however that for all the classes defined above, in the case $\gamma$ is a nonnegative  number,  the $L^{n}$ and $L^{q}$ viscosity concepts coincide whenever $f\in L^{q}(\Omega)$ with $n\leq q <\infty$. The same apply to the of $L^{n}$ and $C-$viscosity concepts  whenever $f\in C^{0}(\Omega)$. Clearly, if $f\in L^{\infty}(\Omega)$, $L^{n}$ and $L^{q}$ viscosity concepts coincide for all $n < q < \infty$. These results are consequence of the Theorem 2.1 in \cite{CKSS}. In this paper, we freely use $L^{n}-$viscosity theory as it appears in \cite{CCKS} and \cite{Koike-Swiech-WHUnbounded}. For the $C-$viscosity theory see \cite{CC}. 
\end{remark}
\begin{remark}\label{general-classes} We observe that the results of this paper also apply to classes involving the zeroth order term, as long as the involved functions are bounded. In fact, we recall the following general Pucci extremal operators, $\mathcal{P}_{\Lambda, \lambda, \gamma, \sigma}^{\pm}\colon \mathcal{S}^{n\times n}\times \mathbb{R}^{n}\times\mathbb{R}\to\mathbb{R}\times\Omega$  given by 

$$ \mathcal{P}_{\gamma, \sigma}^{-}(M,p,z,x)=\mathcal{P}_{\Lambda, \lambda, \gamma, \sigma}^{-}(M,p,z,x):=\mathcal{M}_{\lambda,\Lambda}^{-}(M) - \gamma(x)\cdot |p| +\sigma(x) \cdot z,  \quad \gamma \in L_{+}^{q}(\Omega), \ \sigma \in L^{q}(\Omega), \quad q>n,$$

$$ \mathcal{P}_{\gamma, \sigma}^{+}(M,p,z,x)=\mathcal{P}_{\Lambda, \lambda, \gamma, \sigma}^{-}(M,p,z,x):=\mathcal{M}_{\lambda, \Lambda}^{+}(M) + \gamma(x)\cdot |p| + \sigma(x) \cdot z,   \quad\gamma \in L_{+}^{q}(\Omega), \ \sigma \in L^{q}(\Omega),\quad q>n.$$

\noindent  Similarly as before, we can define  
$$\mathcal{P}_{\gamma,\sigma}^{\pm}[u](x):=\mathcal{P}_{\gamma,\sigma}^{\pm}(D^2u(x), \nabla u(x), u(x))$$ 
$$ \overline{S}(\gamma, \sigma, f), \quad \underline{S}(\gamma, \sigma, f), \quad S(\gamma, \sigma, f) \ \ \textrm{ and } \ \  S^{*}(\gamma, \sigma, f).$$
\noindent Now, 
$$ \overline{S}(\gamma, \sigma, f)\subset \overline{S}(\gamma, f+\sigma\cdot u), \quad  \underline{S}(\gamma, \sigma, f)\subset \underline{S}(\gamma, f+\sigma\cdot u), $$ \vspace{.2mm}
$${S}(\gamma, \sigma, f)\subset {S}(\gamma, f+\sigma\cdot u), \quad S^{*}(\gamma, \sigma, f) \subset S^{*}(\gamma, f+\sigma|u|).$$\\
\noindent since if $u\in L^{\infty}$ then $||f+\sigma u||_{L^{q}}\leq ||f+\sigma |u| ||_{L^{q}} \leq ||f||_{L^{q}} + ||\sigma||_{L^{q}}\cdot ||u||_{L^{\infty}}.$ In the case $\gamma$ is a nonnegative constant and $\sigma$ is a nonpositive constant then similar considerations as in Remark \ref{type-of-viscosity-solution-involved} regarding the equivalence of $L^{n}, L^{q} (n\leq q<\infty)$ viscosity concepts for the classes above apply here. This also follows from Theorem 2.1 in \cite {CKSS} since the Pucci operators $\mathcal{P}_{\gamma, \sigma}^{\pm}$ are monotone decreasing in the variable $z$.
\begin{remark}[{{\bf Monotonicity in $\gamma$}}]\label{monotonicity-in-gamma} Whenever $\gamma_{0}$ is a nonnegative constant we have, in the respective domains, 
$$ \gamma(x) \leq \gamma_{0} \Longrightarrow \underline{S}(\gamma, f) \subseteq \underline{S}(\gamma_{0}, f), \quad \overline{S}(\gamma, f) \subseteq \overline{S}(\gamma_{0}, f), \quad S^{*}(\gamma, f) \subseteq S^{*}(\gamma_{0}, f)$$
 Thus, whenever $\gamma$ is a nonnegative constant, we may replace whenever convenient the dependence on $\gamma$ in the universal constants by any  $\gamma_{0} \geq \gamma$.
\end{remark}

\end{remark}
\begin{remark}[{\bf Scaling remark}]\label{scaling-remark} Suppose $u\in \underline{S}(\gamma(x); f)$ in $\Omega$. Let us set
$$\Omega_{\beta}=\beta^{-1}\cdot\Omega:=\Big\{x; ~~\beta x \in\Omega \Big\}.$$
Then, defining $v(x):=\alpha u(\beta x)$ for $x\in\Omega_{\beta}$, we conclude that $v\in \underline{S}(\overline{\gamma}, \overline{f})$ in $\Omega_{\beta}$ where 
$$ \overline{\gamma}(x):=\beta\gamma(\beta x), \quad \overline{f}(x):=\alpha\beta^{2}f(\beta x).$$
Clearly, similar observations hold for the classes $\underline{S}(\gamma; f), {S}(\gamma; f)$ and ${S}^{*}(\gamma; f).$ These scaling properties are easy to check. The proof follow the spirit of Lemma 2.12 in \cite{CC} for $C-$viscosity solutions. Many times in this paper, we use the scaling $v(x)=u(rx)/r$. In this case, $\overline{f}(x)=rf(rx)$.
 \end{remark}
 \begin{remark}\label{scaling-pointwise-dini} Suppose $u:B_{1}\to\mathbb{R}$ is such that $u\in C^{1,\omega}(0).$ Let $0 <r< 1$ and define $v(x):=u(rx)$ for $x\in B_{1}$. In order to simplify our discussion, suppose the tangent plane of $u$ at zero is zero (i.e, Taylor's polynomial at zero is zero). Then, for $x\in B_{1}$ such that $|x|\leq\delta_{\omega}$ we have
 
 $$ |v(x)|=|u(rx)| \leq [u]_{C^{1,\omega}(0)}|rx|\omega(|rx|)\leq[u]_{C^{1,\omega}(0)} |x|\omega(|x|).$$
 This implies that $[v]_{C^{1,\omega}(0)}\leq [u]_{C^{1,\omega}(0)}.$
 \end{remark}

\section{Examples and Properties of modulus of continuity}
We now discuss some examples of modulus of continuity for which our Theorem applies. 

\begin{example}[{\bf H\"older modulus of continuity}]\label{example-holder} Let $\alpha\in(0,1]$ and consider the following function $\omega(t):=t^{\alpha}.$ Clearly, $\omega$ is nonnegative, increasing, continuous in $[0,1]$ and $\omega(0)=0.$ Let us set $\delta_{\omega}^{*}=\delta_{\omega}=1.$ The quotient $q(t)=\omega(t)/t = t^{\alpha-1}$  is decreasing. It is elementary to check that the Dini continuity is satisfied. Finally, for any $\mu, \delta \in(0,1]$ and $\beta\in [\alpha,1]$ we find  $\mu^{\beta}\omega(\delta\mu^{k})\leq \mu^{\alpha}\omega(\delta\mu^{k})=\mu^{\alpha}\delta^{\alpha}\mu^{k\alpha} \leq \delta^{\alpha}\mu^{k\alpha}\mu^{\alpha}=\omega(\delta\mu^{k+1}).$ Thus, $\omega$ satisfies the $\beta-$compatibility condition between scales. 
\end{example}
\begin{example}[{\bf Pure Dini modulus of continuity}] Let us consider the function $\omega(t)=(\ln(t^{-1}))^{\gamma}$ where $\gamma<-1.$ Clearly, $\lim_{t\to 0^{+}} \omega(t) =0$. Also, for $q(t)=\omega(t)/t$, we have $q'(t)<0 \Longleftrightarrow t\in(0,e^{\gamma})$ and $\omega'(t)>0$ in $(0,1)$. Moreover, setting $\delta_{\omega}=e^{\gamma}\in(0,1)$ for $0<\mu<1$ we have for any $\delta\in(0,1]$

$$\frac{\omega(\delta\mu^{k+1})}{\omega(\delta\mu^{k})} = \bigg(\frac{ln(1/\mu)}{ln(1/\delta_{\omega} \mu^{k})} + 1\bigg)^{\gamma} \geq 2^{\gamma} \geq \mu^{\beta} \quad \textnormal { for } \mu\in (0,2^{\gamma/\beta}],$$
since the expression inside the parenthesis is less than 2 and $\gamma<0.$ This way, if $\beta\in (0,1)$, we define $\delta_{\omega}^{*}:=2^{\gamma/\beta}\in (0,1).$ Thus, $\omega$ satisfies the $\beta-$compatibility condition between scales. Also, making the change of variables $ln(1/t)=s$, we conclude that $\omega$ is a Dini modulus of continuity since for $0<r<\delta_{\omega}$
$$\int_{0}^{r} t^{-1}\big(ln(t^{-1})\big)^{\gamma}dt = \int_{ln (r^{-1})}^{\infty}s^{\gamma}ds=-\frac{(ln(r^{-1}))^{\gamma+1}}{(\gamma+1)}. $$
Furthermore, $\omega(t)$ cannot be controlled by any H\"older type modulus of continuity, say $\leq Ct^{\alpha}$ with $\alpha\in(0,1]$. Indeed, suppose by contradiction that, $\omega(t)\leq Ct^{\alpha}$ for all $t$ small enough and $C>0$. Set $\gamma=-\mu$ where $\mu>0$. Thus,  making the change of variables $\ln(1/t)=x$ we find
\begin{equation}\label{tricky-change-of-variables-2}
0<C^{-1}\leq\limsup\limits_{t\to 0} t^{\alpha}\big(\omega(t)\big)^{-1}=\limsup\limits_{t\to 0} t^{\alpha}\bigg(\ln(t^{-1})\bigg)^{\mu} = \lim\limits_{x\to \infty}e^{-\alpha x}x^{\mu}=0,
\end{equation}
which is clearly a contradiction.
\end{example}

\begin{example}[{\bf Mixed type modulus of continuity}] Let us define $\omega(t)=t^{\alpha}\Big(\ln(t^{-1})\Big)^{\gamma}$ where $\alpha\in (0,1)$ and $\gamma\in\mathbb{R}.$ We already analyzed in the first example the case where $\gamma=0.$ Now we divide our analysis in two cases:\\

\noindent {\bf {\underline {Case 1:}} $\gamma >0.$}\  In this case, by doing the change of variables as in (\ref{tricky-change-of-variables-2}) we arrive to $\lim\limits_{t\to 0}w(t)=0.$ Furthermore, $\omega(t)$ is nonnegative and continuous for $0\leq t\leq 1$. The quotient $q(t)=\omega(t)/t$ is decreasing in $[0,1)$. We observe that  $\omega'(t) \geq 0 \Longleftrightarrow t\in (0, e^{-\gamma/\alpha}]$ and  $e^{-\gamma/\alpha} <1.$  In particular, we can take $\delta_{\omega}=e^{-\gamma/\alpha}.$ Now, we have for $\alpha\leq \beta\in (0,1], \delta\in(0,1]($since $\gamma >0$ and $0\leq A_{k}\leq 1)$ (below)

\begin{equation}\label{tricky-computation-beta-compatibility}
\frac{\omega(\delta\mu^{k+1})}{\omega(\delta\mu^{k})} = \mu^{\alpha} \bigg(\underbrace{\frac{ln(1/\mu)}{ln(1/\delta_{\omega} \mu^{k})} }_{A_{k}}+ 1\bigg)^{\gamma} =\mu^{\alpha} \geq \mu^{\beta} \quad \textnormal { for } \mu\in (0,1).
\end{equation}

\noindent So, we can take $\delta_{\omega}^{*}=1$ and $\omega$ satisfies the $\beta-$compatibility condition. Furthermore, $0<r<\delta_{\omega}$ we obtain by the change of variables $\ln(t^{-1})=s$ that 
\begin{equation}\label{dini-double-analysis}
\int_{0}^{r}t^{\alpha-1}\Big(\ln(t^{-1})\Big)^{\gamma}dt=\int_{\ln(r^{-1})}^{\infty} x^{\gamma} e^{-\alpha x}dx <\infty. 
\end{equation}

\noindent Thus, expression (\ref{dini-double-analysis}) implies that $\omega$ is Dini continuity independently of the sign of $\gamma.$ \\

\noindent {\bf {\underline {Case 2:}} $\gamma <0.$}\ It is clear that $\lim\limits_{t\to 0}\omega(t)=0.$ Again, $\omega(t)$ is nonnegative and continuous for $0\leq t\leq 1.$ We observe that $\omega(t)$ is nondecreasing for all $t\in[0,1).$ Furthermore, for the quotient $q(t)=\omega(t)/t$, we have $q'(t)\leq 0 \Longleftrightarrow t\in [0,e^{\gamma/(1-\alpha)}]$ and $e^{\gamma/(1-\alpha)}<1.$ The Dini continuity of $\omega$ was already established.

 Let $\alpha<\beta <1$. We set $\delta_{\omega}:=e^{\gamma/(1-\alpha)}\in(0,1).$ Now since $\gamma <0$ and $0\leq A_{k}\leq 1$  we have by (\ref{tricky-computation-beta-compatibility}) for all $\delta\in(0,1]$

$$ \frac{\omega(\delta\mu^{k+1})}{\omega(\delta\mu^{k})} = \mu^{\alpha} \bigg(\frac{ln(1/\mu)}{ln(1/\delta \mu^{k})} + 1\bigg)^{\gamma} \geq 2^{\gamma}\mu^{\alpha} =2^{\gamma}\mu^{\alpha-\beta}\mu^{\beta} \geq  \mu^{\beta} \quad \textnormal { for } \mu\in (0,2^{\frac{\gamma}{\beta-\alpha}}).$$
We then define $\delta_{\omega}^{*}=2^{\frac{\gamma}{\beta-\alpha}}.$ This way, $\omega$ satisfies the $\beta-$compatibility condition.
\end{example}
\begin{remark} Unlike the papers \cite{Kovats-1, Kovats-2},  we observe that our assumptions on the modulus of continuity do include the case $\omega\equiv 0$. 
\end{remark}

We now prove some properties of the modulus of continuity. 

\begin{lemma}[{\bf Properties of the modulus of continuity}]\label{properties-modulus-of-continuity} Let $\omega$ be a modulus of continuity. \begin{itemize}
\item[$i$)] If $\omega$ satisfies the $Q-$decreasing quotient property $\big($a) in Definition \ref{definition-MC}$\big)$ then
$$ \omega(t) \leq Q\int_{0}^{t}\frac{\omega(s)}{s}ds \quad \forall t\in[0,\delta_{\omega}];$$ 
$$ \int_{0}^{\Theta t}\frac{\omega(s)}{s}ds \leq {2Q^{2}}\Theta\int_{0}^{t}\frac{\omega(s)}{s}ds \quad \forall t\in [0,{\delta_{\omega}}/\Theta] \ \textnormal{ whenever } \ \Theta \geq 1.$$
\end{itemize}
\indent In particular, 
\begin{itemize}
\item [] $$ \int_{0}^{\mu^{k}}\frac{\omega(s)}{s}ds \leq \frac{2Q^{2}}{\mu}\int_{0}^{\mu^{k+1}}\frac{\omega(s)}{s}ds \quad \forall \mu\in [0,\delta_{\omega}], \ \ \forall k\geq 0.$$
\item[$ii)$] If $\omega$ is a Dini modulus of continuity then for $\mu\in\big(0,\min\{\delta_{\omega},1/e\}\big]$ we have for $k\geq 0$
 $$ \sum\limits_{j=k}^{\infty}\omega(\mu^{j}) \leq \omega(\mu^{k}) +  \int_{0}^{\mu^{k}}\frac{\omega(t)}{t}dt \leq  (2Q) \int_{0}^{\mu^{k}}\frac{\omega(t)}{t}dt.$$
 \item[$iii)$] Assume that $\omega(t)/t^{\alpha}$ is decreasing in $(0,\delta_{\omega}]$ for some $\alpha\in(0,1)$. Then, $\omega$ satisfies properties $(a)$ and $(c)$ $($with $\beta=\alpha$ and $\delta_{\omega}^{*}=1)$ in Definition \ref{definition-MC}. If additionally, $\omega$ is Dini continuous $($i.e, it satisfies $(b)$ in Definition \ref{definition-MC}$)$ then $\omega\in \mathcal{DMC}(1,\alpha).$

\end{itemize}
\end{lemma}
\begin{proof} Let us denote
$$ \omega_{1}(t)=\int_{0}^{t}\frac{\omega(s)}{s}ds.$$ From the $Q-$decreasing quotient property we see that for $t\in[0,\delta_{\omega}]$

$$\omega_{1}(t)=\int_{0}^{t}\frac{\omega(s)}{s}ds\geq Q^{-1}\frac{\omega(t)}{t}\int_{0}^{t}ds =Q^{-1}\omega(t). $$

\noindent Thus, by $Q-$decreasing property and the result proven above (recalling that $\mu\in (0,1))$

\begin{eqnarray*}
 \omega_{1}(\Theta t) = \int_{0}^{\Theta t} \frac{\omega(s)}{s}ds& =& \int_{0}^{t} \frac{\omega(s)}{s}ds  + \int_{t}^{\Theta t} \frac{\omega(s)}{s}ds \\
 & = & \omega_{1}(t) + Q  \int_{t}^{\Theta t}\frac{\omega(t)}{t}ds\\
 & \leq &\omega_{1}(t) + {Q}\Theta\omega(t)\\ \
 & \leq & {\omega_{1}(t)} + Q^{2}\Theta\omega_{1}(t)\\
 & =& 2Q^{2}\Theta\omega_{1}(t).
\end{eqnarray*}
The last inequality follows from the previous one just by taking $\Theta=\mu^{-1}$ and $t=\mu^{k+1}\in (0,\delta_{\omega}/\Theta).$
This proves $i).$ Now, by monotonicity and from the fact  $\mu\in(0,\min\{\delta_{\omega}, 1/e\}]$ and $j\geq 0$
$$\int_{\mu^{j+1}}^{\mu^{j}}\frac{\omega(s)}{s}ds\geq\int_{\mu^{j+1}}^{\mu^{j}} \frac{\omega(\mu^{j+1})}{s}ds = \omega(\mu^{j+1})\ln(\mu^{-1}) \geq \omega(\mu^{j+1}). $$
Now, we estimate for $k\geq 0$
\begin{eqnarray*}
\sum\limits_{j=k}^{N}\omega(\mu^{j})& \leq& \omega(\mu^{k}) + \sum\limits_{j=k}^{N}\omega(\mu^{j+1}) \\
&\leq &  \omega(\mu^{k}) +\sum\limits_{j=k}^{N}\int_{\mu^{j+1}}^{\mu^{j}} \frac{\omega(s)}{s}ds \\
&=&  \omega(\mu^{k}) + \int_{\mu^{N+1}}^{\mu^{k}} \frac{\omega(s)}{s}ds\\
& \leq & (2Q)\omega_{1}(\mu^{k}). 
\end{eqnarray*}
Letting $N\to\infty$, we finish the of $ii)$ in the Lemma.\\

\noindent  Finally, to prove $iii)$, we  observe that if we define $q_{\alpha}(t):=\omega(t)/t^{\alpha}$ for $t\in(0,\delta_{\omega}]$ then $q(t)=\frac{q_{\alpha}(t)}{t^{1-\alpha}}$ which is clearly decreasing since the denominator is increasing. Moreover, for any $ \mu\in(0,1)$ and any $\delta\in(0,\delta_{\omega}]$ we clearly have $0<\delta\mu^{k+1}\leq \delta \mu^{k}\leq \delta_{\omega}$ and 
 $$ \frac{\omega(\delta \mu^{k+1})}{(\delta\mu^{k+1})^{\alpha}} \geq  \frac{\omega(\delta \mu^{k})}{(\delta\mu^{k})^{\alpha}} \Longrightarrow \omega(\delta\mu^{k+1}) \geq \mu^{\alpha}\omega(\delta \mu^{k}).$$ 
\end{proof}
\begin{remark}[{\bf Restriction, Renormalization and monotonicity of modulus of continuity}]\label{restriction-renormalization-MC} 
 \item \underline{Restriction:} Suppose $\omega\in\mathcal{DMC}(Q,\beta)$ and $\tau\in(0,\delta_{\omega})$ then the restriction of $\omega$ to $ [0,\tau]$, i.e, $\widehat{\omega}:[0,\tau]\to[0,\infty)$ given by $\widehat{\omega}(t) = \omega(t)$ for $t\in[0,\tau]$ is such that $\hat{\omega}\in\mathcal{DMC}(Q,\beta).$ In this case, $\delta_{\widehat{\omega}}=\tau$ and $\delta_{\widehat{\omega}}=\delta_{\omega^{*}}.$ Furthermore, if a bounded function $u$ defined on $B_{1}$ is such that $u\in C^{1,\omega}(0)$ and $L$ is its Taylor's polynomial at zero, we have that for $x\in B_{1}$ with $|x|\leq \tau$ 
$$ |u(x) - L(x)|\leq [u]_{C^{1,\omega}(0)} |x|\omega(|x|).$$
In particular $ [u]_{C^{1,\widehat{\omega}}(0)}\leq  [u]_{C^{1,\omega}(0)}.$\\

\item \underline{ Renormalization:} Suppose $\omega\in\mathcal{DMC}(Q,\beta)$ and set $\overline{\omega}(t):=w(Kt)$ where $t\in[0,\delta_{\omega}/K]$ ND $K\geq1.$ It is easy to see that $\overline{\omega}\in \mathcal{DMC}(Q,\beta)$ with $\delta_{\overline{\omega}} =\delta_{\omega}/K$ and $\delta_{\overline{\omega}}^{*}:=\delta_{\omega}^{*}$. Furthermore, $\int_{0}^{\delta_{\omega}}{w(s)}{s}^{-1}ds=\int_{0}^{1}{\overline{w}(s)}{s}^{-1}ds.$
As before, let $u$ be a bounded function defined on $B_{1}$ such that $u\in C^{1,\omega}(0)$ and $L$ its Taylor's polynomial at zero. Setting $\overline{u}(x)=K^{-1}u(Kx)$ and $\overline{L}(x):=K^{-1}L(Kx)$ for $x\in B_{1}$ and $|x|\leq \delta_{\omega}/K$ we have 

\begin{eqnarray*}
|\overline{u}(x) - \overline{L}(x)| & = & \big|K^{-1}\big(u(Kx) - L(Kx)\big)\big|\\\\
& \leq & K^{-1}[u]_{C^{1, \omega}}(0)|Kx|\omega(K|x|)\\\\
& \leq & [u]_{C^{1, \omega}}(0)|x|\overline{\omega}(|x|).
\end{eqnarray*}
This implies that $\overline{L}$ is the Taylor's polynomial of $\overline{u}$ at the origin and  $ [\overline{u}]_{C^{1,\overline{\omega}}(0)}\leq  [u]_{C^{1,\omega}(0)}.$\\

\item \underline{Monotonicity:} Suppose $\omega_{1}(t)\leq A\cdot\overline{\omega_{2}}(t)=\omega_{2}(t)$ for all $t\in [0,\delta]$. Then,
$$ A^{-1}\cdot[u]_{C^{0,\overline{\omega}_{2}}(\Omega)}= [u]_{C^{0,\omega_{2}}(\Omega)}\leq [u]_{C^{0,\omega_{1}}(\Omega)} \Longrightarrow C^{0,\overline{\omega}_{2}}(\Omega) =C^{0,\omega_{2}}(\Omega) \subset C^{0,\omega_{1}}(\Omega). $$ \end{remark}
\noindent In particular, when $\omega_{1}\sim\omega_{2}$ in $[0,\delta]$ then the respective induced norms are equivalent.

\section{Proof of Proposition \ref{Pucci-Barriers} and IHOL - Theorem \ref{IHOL}}
\indent We start by stating and proving a series of Lemmas of independent interest. The proof of Proposition \ref{Pucci-Barriers} is a direct consequence of them. In the sequel, we prove Theorem \ref{IHOL} which has  Proposition \ref{Pucci-Barriers} as the main ingredient. We will use throughout this section the scaling $v(x)=u(rx)/r$ as pointed out in Remark (\ref{scaling-remark}) to reduced the proofs to the case where $r=1$.The first Lemma is indeed the homogeneous version of Proposition \ref{Pucci-Barriers}.

\begin{lemma}\label{classical} Let us consider the following Dirichlet Problem for $0\leq \gamma \leq \gamma_{0}$ and $0<r\leq R_{0}$
 \begin{equation} \label{DP-positive-pucci-barrier-homogeneous}
 \left \{
    \begin{array}{rcll}
\mathcal{P}_{\gamma}^{-}[u]&=& 0& \textrm{ in } \:\mathcal{A}_{\frac{r}{2}, r} \\\\
     u &=& 0& \textrm{ on } \partial B_{r}\\\\
    u &=& M & \textrm { on } \partial B_{\frac{r}{2}}.
    \end{array}
    \right.
\end{equation}
There exists a unique classical solution $\Gamma_{0}\in C^{\infty}(\overline{\mathcal{A}_{r}})$ to the problem (\ref{DP-positive-pucci-barrier-homogeneous}). Furthermore, $\forall x\in\overline{\mathcal{A}_{\frac{r}{2}, r}}$, we have

\begin{equation}\label{estimate-pucci-geometry-barrier-homogeneous}
\overline{A_{1}}\cdot \frac{M}{r} \cdot dist(x,\partial B_{r}) \leq u(x) \leq \overline{A_{3}}\cdot \frac{M}{r} \cdot dist(x,\partial B_{r}) 
\end{equation}
\noindent Here,  $\overline{A_{1}}, \overline{A_{3}}$ are positive universal constants depending only on $n, \lambda, \Lambda, \gamma_{R_{0}}.$
\end{lemma}
\begin{proof} This is actually a particular case of Proposition 4.1 in \cite{BM-IHOL-PartI}. We indicate the details. By scaling, it is enough to discuss the case where $r=1$. The uniqueness follows from the validity of the comparison principle for $C-$viscosity solutions (see \cite{CCKS}). In fact, for the problem (\ref{DP-positive-pucci-barrier-homogeneous}), all the concepts ($C$-viscosity, $L^{n}$-viscosity, strong and classical solutions) coincide.  Regarding the existence, once solution must be radial (equations is invariant under rotations), the PDE in (\ref{DP-positive-pucci-barrier-homogeneous}) can be reduced to an ODE. Thus, we define
$$ \Gamma_{0}(x) := M\cdot\phi(|x|) \quad \textnormal { for } \quad  \frac{1}{2}\leq |x| \leq 1, $$
where $\phi:[1/2, 1]\to\mathbb{R}$ is given by
$$  \phi(r) = \Bigg(\int_{1/2}^{1} t^{-\mathcal{E}_{0}}e^{-(\gamma/\lambda)t}dt\Bigg)^{-1}\cdot \int_{r}^{1} t^{-\mathcal{E}_{0}}e^{-(\gamma/\lambda) t}dt, \quad \mathcal{E}_{0}:= \frac{\Lambda\cdot(n-1)}{\lambda}\geq 1.$$
One can check that $\phi$ is decreasing and convex. From this, it is easy to check that $\Gamma_{0}$ is a (classical) solution to (\ref{DP-positive-pucci-barrier-homogeneous}) and also satisfies the indicated properties.
\end{proof}
\begin{lemma} \label{difference-s-star-class} 
Let $U\subset \mathbb{R}^{n}$ be an open set. Suppose that $  \mathcal{P}_{\gamma}^{-}[u]= f$ in $U$ in the $L^{n}$-viscosity sense and that $v\in W_{loc}^{2,n}(U)$ is a strong solution to $ \mathcal{P}_{\gamma}^{-}[v]= g $ in $U$. Here, $f,g\in L_{loc}^{n}(U)$.  Then, $u-v\in S(\gamma, f-g)$ in $U$. 
\end{lemma}
\begin{proof} Indeed, let $\phi\in W_{loc}^{2,n}(U)$ be such that $(u-v)-\phi = u-(v+\phi)=u-\phi^{*}$ has a local maximum at $x_{0}\in U$ where $\phi^{*}=v+\phi$. This way, for $a.e.~x$ in $U$ we have
\begin{eqnarray*}
\mathcal{P}_{\gamma}^{+}[\phi](x)&=&\mathcal{P}_{\gamma}^{+}[\phi^{*}-v](x)\\
&\geq &\mathcal{P}_{\gamma}^{-}[\phi^{*}](x) + \mathcal{P}_{\gamma}^{+}[-v](x)\\
& = & \mathcal{P}_{\gamma}^{-}[\phi^{*}](x) -\mathcal{P}_{\gamma}^{-}[v](x)\\
& =& \mathcal{P}_{\gamma}^{-}[\phi^{*}](x) -g(x).
\end{eqnarray*}
\noindent In particular, by the estimate above
\begin{eqnarray*}
ess\limsup\limits_{x\to x_{0}} \Big(\mathcal{P}_{\gamma}^{+}[\phi] - \Big(f(x) - g(x)\Big) \Big) &\geq & ess\limsup\limits_{x\to x_{0}} \Big(\mathcal{P}_{\gamma}^{-}[\phi^{*}](x) -g(x) - \Big(f(x) - g(x)\Big) \Big)\\
& = & ess\limsup\limits_{x\to x_{0}} \Big(\mathcal{P}_{\gamma}^{-}[\phi^{*}](x) -f(x)\Big)\\
& \geq & 0.
\end{eqnarray*} 
So, $\mathcal{P}_{\gamma}^{+}[u]\geq f-g$ in $U$ in the $L^{n}$-viscosity sense. Conversely, suppose $\phi\in W_{loc}^{2,n}(U)$ and $u-v-\phi=u-(v+\phi)=u-\phi^{*}$ has a local minimum at $x_{0}\in U$, where as before $\phi^{*}:=v+\phi.$  This way, for $a.e.~x$ in $U$ we have
\begin{eqnarray*}
\mathcal{P}_{\gamma}^{-}[\phi](x)&=&\mathcal{P}_{\gamma}^{-}[\phi^{*}-v](x)\\
&\leq &\mathcal{P}_{\gamma}^{-}[\phi^{*}](x) + \mathcal{P}_{\gamma}^{+}[-v](x)\\
& = & \mathcal{P}_{\gamma}^{-}[\phi^{*}](x) -\mathcal{P}_{\gamma}^{-}[v](x)\\
& =& \mathcal{P}_{\gamma}^{-}[\phi^{*}](x) -g(x).
\end{eqnarray*}
\noindent In particular, as before, by the estimate above
\begin{eqnarray*}
ess\liminf\limits_{x\to x_{0}} \Big(\mathcal{P}_{\gamma}^{-}[\phi] - \Big(f(x) - g(x)\Big) \Big) &\leq & ess\liminf\limits_{x\to x_{0}} \Big(\mathcal{P}_{\gamma}^{-}[\phi^{*}](x) -g(x) - \Big(f(x) - g(x)\Big) \Big)\\
& = & ess\liminf\limits_{x\to x_{0}} \Big(\mathcal{P}_{\gamma}^{-}[\phi^{*}](x) -f(x)\Big)\\
& \leq & 0.
\end{eqnarray*} 
So, $\mathcal{P}_{\gamma}^{-}[u]\leq f-g$ in $U$ in the $L^{n}$-viscosity sense. This finishes the proof of the Lemma.
\end{proof}
Now, we present the last Lemma we need to prove the results of this section. 
\begin{lemma}\label{lemma-control-by-the-distance} Let $u\in C^{0}(\overline{\mathcal{A}_{\frac{r}{2},r}})\cap S^{*}(\gamma; f)$ in $\mathcal{A}_{\frac{r}{2},r}$ where $f\in L^{q}(\mathcal{A}_{\frac{r}{2},r})$ for $q>n$ and $0\leq \gamma \leq \gamma_{0}$ and $0<r\leq R_{0}$. Additionally, assume that $u=0$ on $\partial \mathcal{A}_{\frac{r}{2}, r}.$ Then, there exists a positive universal constant $C=C(n, q, \lambda, \Lambda, \gamma_{R_{0}})>0$ such that
\begin{equation}\label{control-by-the-distance}
|u(x)| \leq Cr^{1-n/q}||f||_{L^{q}(\mathcal{A}_{r/2,r})}dist(x,\partial B_{r}) \quad \textnormal{ for every } \ x\in\overline{\mathcal{A}_{\frac{r}{2},r}}.
\end{equation}
\end{lemma}
\begin{proof} By scaling, it is enough to study the case where $r=1$. Let us consider $\Gamma_{a}$ and $\Gamma_{b}$ the $L^{n}$-viscosity solutions to the following Dirichlet problems
\begin{align} \label{Dirichlet pucci problems}
\left \{
    \begin{array}{rcll}
\mathcal{P}_{\gamma}^{-}[\Gamma_{b}]&=& |f| & \textrm{ in } \:\mathcal{A}_{1/2, 1} \\\\
     u &=& 0 & \textrm{ on } \partial \mathcal{A}_{1/2,1},
    \end{array}
    \right.
 && \left \{
    \begin{array}{rcll}
\mathcal{P}_{\gamma}^{+}[\Gamma_{a}]&=& -|f| & \textrm{ in } \:\mathcal{A}_{1/2, 1} \\\\
     u &=& 0 & \textrm{ on } \partial \mathcal{A}_{1/2,1}.
    \end{array}
    \right.
\end{align}
The existence of such solutions can be directly quoted from Theorem 4.1 in \cite{CKLS}. A function is a $L^{n}-$viscosity solution to the Dirichlet Problem above if and only if it is a $L^{n}-$strong solution of the same problem. Indeed, $L^{n}-$strong solutions are $L^{n}-$viscosity solutions by Theorem 2.1 in \cite {CKSS}. The converse follows from the fact that $W_{loc}^{2,n}$ interior regularity is available for $L^{n}-$viscosity solutions to the Dirichlet problems (\ref{Dirichlet pucci problems}) by Theorem 4.2 in \cite{NW}. The result then follows from Corollary 3.7 in \cite{CCKS}. Now, we can apply the comparison principle in the presence of strong solutions (Theorem 2.10 in \cite{CCKS}) to obtain

\begin{equation}\label{comparison-pucci-barriers-constructed}
\Gamma_{b}(x) \leq u(x) \leq \Gamma_{a}(x) \quad \forall x\in \overline{\mathcal{A}_{1/2,1}}.
\end{equation}
Since $q>n$, we have $C^{1,\alpha}$ estimates up to the boundary for problems (\ref{Dirichlet pucci problems}) (see Theorem 4.5 in \cite{NW}). More precisely, 
\begin{equation}\label{control-c1-alpha}
\max\Big\{||\Gamma_{a}||_{C^{1,\alpha}(\mathcal{A}_{1/2,1})}, ||\Gamma_{b}||_{C^{1,\alpha}(\mathcal{A}_{1/2,1})} \Big\} \leq C\cdot ||f||_{L^{q}(\mathcal{A}_{1/2,1})}.
\end{equation} 
In particular, for any $x\in\mathcal{A}_{1/2,1}\setminus \partial B_{1},$ we have by (\ref{control-c1-alpha}) that
\begin{eqnarray}
|\Gamma_{a}(x)| &=& |\Gamma_{a}(x)-\Gamma_{a}(x/|x|)| \\ \nonumber\label{Lipschitz-gamma-a}
&\leq& ||\nabla \Gamma_{a}||_{L^{\infty}(\mathcal{A}_{1/2,1})}\cdot||x-x/|x|||  \\ \nonumber
&\leq &C||f||_{L^{q}(\mathcal{A}_{1/2,1})}dist(x,\partial B_{1})
\end{eqnarray}
The inequality above trivially holds in $\partial B_{1}.$ Similarly, we prove that for all $x\in\mathcal{A}_{1/2,1}$ we have 

\begin{equation}\label{Lipschitz-gamma-b}
|\Gamma_{b}(x)|\leq C||f||_{L^{q}(\mathcal{A}_{1/2,1})}dist(x,\partial B_{1}).
\end{equation}
Combining the inequalities (\ref{comparison-pucci-barriers-constructed}), (\ref{Lipschitz-gamma-a}) and (\ref{Lipschitz-gamma-b}), the Lemma is proven.
\end{proof}
\begin{center}
{\bf \underline{ Proof of Proposition \ref{Pucci-Barriers}}} 
\end{center}
\begin{proof} Existence, uniqueness and the fact that $L^{n}$-viscosity solutions are also $L^{n}$-strong solutions to the Dirchlet problem in (\ref{DP-positive-pucci-barrier-displaced}) follow exactly from the same arguments we used the same assertions in the proof of Lemma \ref{lemma-control-by-the-distance}. It remains to prove inequality in (\ref{estimate-pucci-geometry-barrier}). By scaling arguments, it is enough to show the inequality only when $r=1.$ So, let $v\in C^{\infty}(\overline{\mathcal{A}_{1/2,1}})$ be the unique classical solution (and thus also $L^{n}$-viscosity solution by the equivalence of these notions proved in Theorem 2.1 in \cite{CKSS}) of the problem 
\begin{equation} \label{DP-positive-pucci-barrier-homogeneous-extra}
 \left \{
    \begin{array}{rcll}
\mathcal{P}_{\gamma}^{-}[v]&=& 0& \textrm{ in } \:\mathcal{A}_{1/2,1} \\\\
     v &=& 0& \textrm{ on } \partial B_{r}\\\\
    v &=& M & \textrm { on } \partial B_{\frac{r}{2}}.
    \end{array}
    \right.
\end{equation}
given by Lemma \ref{classical}. From the same Lemma, we know that
\begin{equation}\label{estimate-homogeneous-IHOL-proof}
\overline{A_{1}}\cdot {M} \cdot dist(x,\partial B_{1}) \leq v(x) \leq \overline{A_{3}}\cdot {M} \cdot dist(x,\partial B_{1}) \qquad \forall x\in \overline{\mathcal{A}_{1/2,1}},
\end{equation}
where $\overline{A_{1}}, \overline{A_{2}}$ depends only on $n,\lambda, \Lambda, \gamma.$ Now, we consider $w:=u-v \in C^{0}(\overline{\mathcal{A}_{1/2,1}}).$ Since $v$ is smooth solution to (\ref{DP-positive-pucci-barrier-homogeneous-extra}) we obtain from Lemma \ref{difference-s-star-class} that $w\in S(\gamma; f)$ in $\mathcal{A}_{1/2,1}.$ Clearly, $w=0$ along $\partial\mathcal{A}_{1/2,1}.$ This way, we conclude from Lemma \ref{lemma-control-by-the-distance} that
$$ |w(x)|\leq C||f||_{L^{q}(\mathcal{A}_{1/2,1})}dist(x,\partial B_{1}) \qquad \forall x\in\overline{\mathcal{A}_{\frac{1}{2},1}}.$$
where $C=C(n,q,\lambda, \Lambda)>0.$ This implies, 
$$ v(x) -C||f||_{L^{q}(\mathcal{A}_{1/2,1})}dist(x,\partial B_{1}) \leq u(x) \leq v + C||f||_{L^{q}(\mathcal{A}_{1/2,1})}dist(x,\partial B_{1}) \qquad \forall x\in\overline{\mathcal{A}_{\frac{1}{2},1}}.$$
Taking into account the estimate (\ref{estimate-homogeneous-IHOL-proof}), we finally conclude that $\forall x\in \overline{\mathcal{A}_{1/2,1}}$ we have
$$\Big(\overline{A_{1}}M -C||f||_{L^{q}(\mathcal{A}_{1/2,1})}\Big)dist(x,\partial B_{1})\leq u(x) \leq  \Big(\overline{A_{3}}M +C||f||_{L^{q}(\mathcal{A}_{1/2,1})}\Big)dist(x,\partial B_{1}).$$
This finishes the proof of the Proposition.
\end{proof}
\begin{center}
{\bf \underline{ Proof IHOL - Theorem \ref{IHOL}}} 
\end{center}
\begin{proof} Once more, by scaling, it is enough to prove only the case $r=1$. In order to prove the Theorem, we observe that suffices to prove the result for ``small RHS". More precisely, it suffices to prove that there exist positive universal constants $H_{1}, H_{2}$ depending on $n,q, \lambda,\Lambda, \gamma$ such that 
\begin{equation}\label{suffices}
||f||_{L^{q}(B_{1})} \leq H_{2}u(0) \Longrightarrow u(x) \geq H_{1}u(0)dist(x,\partial B_{1}) \quad \forall x\in \overline{B}_{1}.
\end{equation}
Indeed, we observe that if (\ref{suffices}) holds then 
\begin{equation}\label{desired}
u(x) \geq \Big(H_{1}u(0) - \frac{2H_{1}}{H_{2}}||f||_{L^{q}(B_{1})}\Big)dist(x,\partial B_{1}) \quad \forall x\in \overline{B_{1}}.
\end{equation}
This is easy to see. If $||f||_{L^{q}(B_{1})} \leq H_{2}u(0)$ then (\ref{suffices}) implies (\ref{desired}). Otherwise,  (\ref{desired}) holds trivially since 
$$u(x) \geq 0 \geq   \Big(H_{1}u(0) - \frac{2H_{1}}{H_{2}}||f||_{L^{q}(B_{1})}\Big)dist(x,\partial B_{1}) \quad \forall x\in\overline{B_{1}}.$$
So, we will just prove (\ref{suffices}). By Harnack inequality (Corollary 5.12 in \cite{Fok}), there exists a universal constant $C=C(n,q, \lambda,\Lambda,\gamma )>0$ such that
\begin{equation}\label{harnack-in-proof-Hopf}
u(0)\leq \sup\limits_{B_{1/2}}\leq C\Big(\inf\limits_{B_{1/2}}u + ||f||_{L^{q}(B_{1})}\Big).
\end{equation}
Thus, 
\begin{equation}\label{solution-pop-up}
||f||_{L^{q}(B_{1})}\leq \frac{1}{2C}u(0) \Longrightarrow \inf\limits_{B_{1/2}} u \geq \frac{1}{2C}u(0).
\end{equation}
Now, we consider the following barrier 
 \begin{equation} \label{barrier-Hopf}
 \left \{
    \begin{array}{rcll}
\mathcal{P}_{\gamma}^{-}[v]&=& |f|& \textrm{ in } \:\mathcal{A}_{\frac{1}{2}, 1} \\\\
     v&=& 0& \textrm{ on } \partial B_{1}\\\\
    v &=& M:=\frac{u(0)}{2C} & \textrm { on } \partial B_{\frac{1}{2}}.
    \end{array}
    \right.
\end{equation}
Now, Proposition \ref{Pucci-Barriers} gives $\forall x\in\overline{\mathcal{A}_{1/2,1}}$
\begin{equation}\label{geometry-barrier-in-hopf-proof} 
v(x) \geq \Bigg(\frac{A_{1}}{2C}u(0) - A_{2}||f||_{L^{q}(\mathcal{A}_{1/2,1})}\Bigg)\cdot dist(x,\partial B_{1}).
\end{equation}
Now, suppose that 
\begin{equation}\label{crucial-smallness}
\|f\|_{L^{q}(\mathcal{A}_{1/2,1})} \leq \min\Bigg\{\frac{1}{2C}, \frac{A_{1}}{4A_{2}C}\Bigg\}u(0).
\end{equation}
This way, estimate (\ref{solution-pop-up}) combined  with the comparison principle (applied for the operator $\mathcal{P}_{\gamma}^{-}$) and (\ref{geometry-barrier-in-hopf-proof}) imply that
\begin{equation}\label{almost-final}
u(x) \geq v(x) \geq \frac{A_{1}}{4C}u(0)dist(x,\partial B_{1})\quad \forall x\in \overline{\mathcal{A}_{1/2,1}}.\end{equation}
Also, under the assumption in (\ref{crucial-smallness}), we also conclude by (\ref{solution-pop-up}) that
\begin{equation}\label{getting-the-rest-of-ball} 
u(x) \geq \frac{1}{2C}u(0)\geq \frac{1}{2C}u(0)dist(x,\partial B_{1}) \quad \forall x\in \overline{B_{1/2}}. 
\end{equation}
Thus, (\ref{crucial-smallness}), (\ref{almost-final}) and (\ref{getting-the-rest-of-ball}) together imply that (\ref{suffices}) holds for 
$$H_{1}:= \min\Bigg\{\frac{1}{2C}, \frac{A_{1}}{4C}\Bigg\} \qquad \textnormal { and  } \quad H_{2}:=\min\Bigg\{\frac{1}{2C}, \frac{A_{1}}{4A_{2}C}\Bigg\}.$$
Estimate (\ref{BCN-Hopf}) follows by direct computations. This finishes the proof of IHOL.
\end{proof}
\section{Proof of Propositions \ref{bdry-lip-type-estimate-scaled-version} and \ref{Boundary Behaviour ffb} }

\begin{center}
{\bf \underline{ Proof of Propostion \ref{Boundary Behaviour ffb}}} 
\end{center}

\begin{proof}
By scaling it is enough to prove the Proposition when $r = 1$. Let $x_0 \in \overline{B}_{1/2}^+$. If $(x_0)_n \geq 1/16$, the Harnack inequality in $D_{0}:=\overline{B}_{3/4}^{+} \cap \lbrace x_n \geq  1/16 \rbrace \ni x_{0}$ and since $(x_{0})_{n} \leq |x_{0}| \leq 1,$ we have that there exists $c_{0} = c_{0}(n, q, \lambda, \Lambda, \gamma) \in (0, 1)$ such that
\begin{equation}\label{Inh-HI-x0}
u(x_0) \geq \left( c_{0}\cdot  u\left(\frac{1}{2} e_n \right) - ||f||_{L^{q}(B_{1}^{+})} \right) \cdot (x_0)_n.
\end{equation}

\noindent Now, if $(x_0)_n < 1/16$ we take $y_0$ to be the projection of $x_0$ on $\lbrace x_n = 1/16 \rbrace$, i.e, if  $x_{0}=(x_{0}',(x_{0})_{n})$ we set $y_{0}:=(x_{0}',1/16).$ Let us observe that $B_{1/16}(y_{0}) \subset B_{3/4}^{+}$ with $x_{0}\in \overline{B}_{1/16}(y_{0})$. Then, by estimate (\ref{hopf-mean-value-inequality-FN}) in Theorem \ref{IHOL} applied to the ball $B_{1/16}(y_{0})$,  we have that
\begin{eqnarray} \label{hopf principle estimate}
u(x_0) & \geq & \Bigg( \overline{C_{1}}\cdot u(y_0) -\overline{C_{2}} \cdot ||f||_{L^{q}(B_{1}^{+})} \Bigg) \cdot \textrm{dist}(x_0, \partial B_{\frac{1}{16}}(y_0)) \\
& =  & \Bigg( \overline{C_{1}} \cdot u(y_0) -\overline{C_{2}} \cdot ||f||_{L^{q}(B_{1}^{+})} \Bigg) \cdot (x_0)_n. \nonumber
\end{eqnarray}
Since $y_0\in D_{0},$ we can use Harnack inequality once more (like in (\ref{Inh-HI-x0}))  to obtain

\begin{eqnarray} \label{main estimate in the prop}
 u(x_0) & \geq & \Bigg( \overline{C_{1}} \cdot \Bigg(c_{0} \cdot u\left(\frac{1}{2} e_n \right) - ||f||_{L^{q}(B_{1}^{+})} \Bigg) -\overline{C_{2}} \cdot ||f||_{L^{q}(B_{1}^{+})} \Bigg) \cdot (x_0)_{n} \\
& = & \Bigg( c_{0} \cdot \overline{C_{1}} \cdot u \left(\frac{1}{2}e_{n}\right) -(\overline{C_{1}} + \overline{C_{2}}) \cdot||f||_{L^{q}(B_{1}^{+})} \Bigg) \cdot (x_0)_{n}. \nonumber
\end{eqnarray}
The proof is finished by taking $D_{2}:=\min\Big\{c_{0}, c_{0}\cdot \overline{C_{1}} \Big\}$ and $D_{3}:=\max \Big\{1,  \overline{C_{1}} + \overline{C_{2}}\Big\}$.
\end{proof}
\begin{center}
{\bf \underline{ Proof of Propostion \ref{bdry-lip-type-estimate-scaled-version}}} 
\end{center}

\begin{proof}
As before, by scaling, it is enough to prove the case $r=1$. We divide the proof in two case: first, suppose $x_0 \in \overline{B}_{1/2}^{+}\cap \lbrace x_n \geq 1/16 \rbrace.$ We have
$$\vert u(x_0) \vert \leq 16 \Big( {\Vert u \Vert}_{L^{\infty}(B_{1}^+(0))} + {\Vert f \Vert}_{L^{q}(B_{1}^+(0))} \Big) \cdot (x_0)_n + \sup_{B'_1} \vert u \vert,$$
proving desired estimate. If $x_0 \in \lbrace x_n < 1/16 \rbrace\cap B_{1/2}^{+}$ we take $y_0$ to be its projection onto the the hyperplane $\{ x_n = - 1/16 \}$, i.e, if $x_{0}=(x_{0}',{(x_{0})}_{n})$ then
\begin{equation}\label{simplify-perception-comparison-FN} 
y_{0}:=(x_{0}', -1/16) \in \Big\{ x_n = - 1/16 \Big\} \Longrightarrow |x_{0}-y_{0}| = |(x_{0})_{n} + 1/16| < 1/8.
\end{equation} 
Clearly, $ \overline{B}_{1/8}(y_{0}) \subset \overline{B}_{3/4}.$ Set 
$$ \mathcal{A}_{0}:={\mathcal{A}}_{\frac{1}{16}, \frac{1}{8}}(y_0) :=\Big\{x\in\mathbb{R}^{n};\ 1/16<|x-y_{0}|<1/8 \Big\}.$$

\noindent Now, we  consider $\Gamma^+$ the (unique) $L^{n}-$strong solution to Dirichlet problem (\ref{pucci+}) (as in Proposition \ref{Pucci-Barriers+})
\begin{equation}\label{pucci+}
\left \{
    \begin{array}{rcll}
{\mathcal{P}}_{\gamma}^{+}[\Gamma^{+}] & = & - \vert \widetilde{f} \vert & \textrm{ in } \mathcal{A}_{0}\\\\
\Gamma^{+} & = & 0 & \textrm{ on } \partial B_{1/16}(y_{0})\\\\
\Gamma^{+} & = & {\Vert u \Vert}_{L^{\infty}(B_1^+)}  & \textrm { on } \partial B_{1/8}(y_{0}),
    \end{array}
    \right.
\end{equation}
\noindent where  $\widetilde{f}$ is the extension of $f$ given by 
\begin{equation}\label{extension of f}
\widetilde{f}(x', x_{n}) :=  \left \{
\begin{array}{lc}
f(x', x_{n}), & \textrm{ if } \ (x',x_{n}) \in B_{1}^{+} \\ \\
0, & \textrm{ if } (x',x_{n}) \in B'_{1}\\ \\
f(x', - x_{n}), & \textrm{ if } \ (x',x_{n}) \in B_{1}^{-}.
\end{array}
\right.
\end{equation}
This way, $||\widetilde{f}||_{L^{q}(B_{1})} = 2||f||_{L^{q}(B_{1}^{+})}$. By the maximum principle (i.e, ABP estimate (Theorem 3.3 in \cite{CCKS})),
$$\Gamma^{+}\geq 0 \ \textnormal { in } \ \overline{\mathcal{A}_{0}}.$$
Moreover, Proposition \ref{Pucci-Barriers+} gives
\begin{equation}\label{estimate-from-above-pucci-barrier+}
\Gamma^{+}(x) \leq A_{0}\Big(||u||_{L^{\infty}(B_{1}^{+})}+ ||f||_{L^{q}(B_{1}^{+})} \Big)dist(x,\partial B_{1/16}(y_{0})) \quad \forall x\in\overline{\mathcal{A}_{0}}.
\end{equation}
\noindent where $A_{0}=A_{0}(n,q,\lambda,\Lambda, \gamma)>0$ is universal. It is easy to verify  that 
$$u, - u \in \underline{S}(\gamma, - \vert {f} \vert) \quad \textnormal { in } \quad B_{\frac{1}{8}}(y_0) \cap \lbrace x_n \geq 0 \rbrace.$$
We observe also that 
$$V:=B_{1/8}(y_{0})\cap \{ x_{n}> 0\} \subset B_{3/4}^{+}, $$ 
$$u, - u \leq \mathcal{B}:=\Gamma^+ + \sup_{B'_1} \vert u \vert \ \ \textrm{ on } \partial (B_{1/8}(y_{0})\cap \{ x_{n} >0  \})=\partial V.$$
By comparison principle with the strong solution $\mathcal{B}$ applied to $\mathcal{P}_{\gamma}^{+}$ (Theorem 2.10 of \cite{CCKS}), we obtain
$$u, - u \leq \mathcal{B} \textrm{ in } \ B_{\frac{1}{8}}(y_0) \cap \lbrace x_n \geq 0 \rbrace.$$
Now, estimate (\ref{estimate-from-above-pucci-barrier+}) gives (once $x_0 \in B_{\frac{1}{8}}(y_0) \cap \lbrace x_n \geq 0 \rbrace$, by (\ref{simplify-perception-comparison-FN}))
\begin{eqnarray} \label{main estimate in the prop}
\vert u(x_0) \vert & \leq & \Gamma^+(x_{0}) + \sup_{B'_1} \vert u \vert \\
& \leq & A_{0}\Big(||u||_{L^{\infty}(B_{1}^{+})}+ ||f||_{L^{q}(B_{1}^{+})} \Big)dist(x,\partial B_{1/16}(y_{0}))  + \sup_{B'_1} \vert u \vert \nonumber \\
& = & A_{0} \left( {\Vert u \Vert}_{L^{\infty}(B_1^+)} + {\Vert f \Vert}_{L^{q}(B_1^+)} \right) \cdot (x_0)_n + \sup_{B'_1} \vert u \vert, \nonumber
\end{eqnarray}
\noindent The proof is completed by taking  ${D}_{1}:=\max\big\{16, A_{0} \big\}$.
\end{proof}

\begin{remark}\label{bdry-lip-type-estimate-scaled-version-extended} It is clear from the proof of Proposition  \ref{bdry-lip-type-estimate-scaled-version} that in the estimate \ref{Lip estimate fully - scaled} we can replace $B_{r/2}^{+}$ by $B_{3r/4}^{+}$ perhaps changing slightly the universal constants. We remark however that in this case the new universal constants will have exactly the same dependence as the old ones.
\end{remark}
\begin{remark} \label{dependence of universal constants of gamma}
By the monotonicity dependence of the universal constants $A_{1}$ and $A_{2}, A_{3}$ in the previous results on the parameter $\gamma_{R_{0}}$ (see Remark \ref{monotonicity-in-gamma}), we conclude that in the particular case where $\gamma \leq  1$ and $R_{0} \leq 1$ the universal constants $D_{2}$, $D_{3}$ in the Proposition \ref{Boundary Behaviour ffb} and $D_{1}$ in Proposition \ref{bdry-lip-type-estimate-scaled-version} do not depend on $\gamma_{_{R_{0}}}$. More precisely, in this case, $D_{1}, D_{2}$ and $D_{3}$ depend only on $n,q,\lambda, \Lambda.$
\end{remark}

\section{Sharpness of Lipschitz regularity up to the boundary with respect to RHS}
\noindent Let us consider the function given by 

$$u(x,y) = \left \{
\begin{array}{ll}\label{function-krylov-counter-example} 
 y\cdot \bigg|\ln\sqrt{(x^2+y^2)}\bigg|^{1/4} & \forall (x,y) \in B_{1/2}^{+}\subset\mathbb{R}^{2}, \\\\
 0 &  \textrm{ on } \big\{y=0\big\}\cap \overline{B}_{1/2}^{+}\subset \mathbb{R}^{2}.
\end{array}
\right.
$$
Using polar coordinates in $\mathbb{R}^{2},$ i.e, $r=\sqrt{(x^2+y^2)}$, we see that for every $(x,y)\in \overline{B}_{1/2}^{+}$
$$0\leq u(x,y) \leq  r\cdot |\ln r|^{1/4} \to 0 \quad \textnormal { as } \quad r\to 0.$$
This way, $u\in C^{0}(\overline{B}_{1/2}^{+})\cap C^{\infty}(B_{1/2}^{+}).$ Let us denote $w(r)=|\ln r|^{1/4}$ for $r>0$. Then,

\begin{eqnarray*}
\Delta u(x,y) & = & \Big(w''(r) + \frac{1}{r}w'(r) \Big)\cdot y +  2\langle w'(r)\cdot\frac{(x,y)}{r}, e_{2}\rangle \\
& = &  \Big(w''(r) + \frac{3}{r}w'(r) \Big)\cdot y=:f(x,y)
\end{eqnarray*}
Now, direct computations shows that for some constant $A_{0}>0$ we have 

$$ (w'(r))^{2} \leq \frac{A_{0}}{r^{2}|\ln r|^{3/2}}  \quad \forall r\in (0,1/2),$$

$$ (w''(r))^{2} \leq \frac{A_{0}}{r^4|\ln r|^{3/2}} + \frac{A_{0}}{r^{4}|\ln r|^{7/2}}  \quad \forall r\in (0,1/2).$$ 

\noindent This way, by using polar coordinates in $\mathbb{R}^{2}$ and the change of variables $s=\ln r$
\begin{eqnarray*}
\int_{B_{1/2}} |f(x,y)|^{2}dxdy &\leq &18\Bigg(\int_{B_{1/2}} (w''(r)r)^{2}dxdy + \int_{B_{1/2}}(w'(r))^{2}dxdy\Bigg)\\\\
& \leq &  36 \pi \Bigg(\int_{0}^{1/2}(w''(r))^{2}r^{3}dr + \int_{0}^{1/2}(w'(r))^{2}rdr \Bigg)\\\\
& \leq &  36 \pi A_{0} \Bigg(2 \int_{0}^{1/2} \frac{dr}{r|\ln r|^{3/2}} + \int_{0}^{1/2} \frac{dr}{r|\ln r|^{7/2}} \Bigg)\\\\
& = & 36 \pi A_{0} \Bigg(2 \int_{-\infty}^{\ln(1/2)} \frac{ds}{|s|^{3/2}} + \int_{-\infty}^{ln(1/2)} \frac{ds}{|s|^{7/2}} \Bigg) < \infty.
\end{eqnarray*}
Thus, $f\in L^{q}(B_{1/2})$ and by the Calderon-Zygmund theory, $u\in W_{loc}^{2,q}(B_{1/2})$ with $q>2$. Hence, $u$ is a $L^{2}-$strong solution and hence a $L^{2}-$viscosity solution to $\Delta u= f $ in $B_{1/2}$ by Theorem 2.1 in \cite {CKSS}. We observe however that Theorem \ref{bdry-lip-type-estimate-scaled-version} does not hold. Indeed, otherwise 

$$\Bigg| \frac{u(0,y)}{y} \Bigg| \leq C \quad \textnormal { for every } \quad y\in B_{1/4}^{+}.$$ 

\noindent However, by the definition of $u$, it is immediate to check that  $u(0,y)/y \to \infty$ as $y\to 0^{+}.$\\

\noindent  In fact, there is a blow-up of the gradient as $(0,y)$ approaches the origin by $y >0$ since 
$$ \frac{\partial u}{\partial y}(x,y) = \langle \nabla (y\cdot w(r)), e_{2}\rangle=  \langle w(r)e_{2} + y w'(r)\frac{(x,y)}{r}, e_{2}\rangle = w(r) + \frac{w'(r)}{r}y^{2},$$
and in particular, 
$$  \frac{\partial u}{\partial y}(0,y) = w(y) + w'(y)y = |\ln y|^{1/4} - \frac{1}{4}|\ln y|^{-3/4} \to \infty \quad \textnormal { as } \quad y\to 0. $$

\section{Boundary gradient estimate with zero boundary data - 
 Lipschitz implies $C^{1,\alpha}$ on the boundary}
 
\begin{remark}[{\bf Perturbation by linear functions}]\label{perturbation-by-linear} In the sequel, we use several times perturbation of functions in the class $S^{*}(\gamma;f)$ by linear functions. In order to simplify the arguments to come, we point out that if $L$ is a linear function and $\gamma \in L_{+}^{n}(\Omega), f\in L^{n}(\Omega)$ \\
 \begin{itemize}
 \item[$i)$] $u\in\underline{S}(\gamma, f) $ in $\Omega  \Longrightarrow u+L \in \underline{S}(\gamma, f-\gamma|\nabla L|)$ in $\Omega$;\vspace{.3cm}
  \item[$ii)$] $u\in\overline{S}(\gamma, f) $ in $\Omega \Longrightarrow u+L \in \overline{S}(\gamma, f+\gamma|\nabla L|)$ in $\Omega$;\vspace{.3cm}
  \end{itemize}
  In particular, $\forall A,B \in\mathbb{R}$
\begin{equation}\label{class-perturbed-by-linear-functions}
u\in S^{*}(\gamma, f) \Longrightarrow v:=A\cdot u+ B\cdot L \in S^{*}\Big(\gamma, |A|\cdot |f| + \gamma\cdot|B|\cdot|\nabla L|\Big)  
\end{equation}
We now prove item $i)$. Item $ii)$ follows similarly and (\ref{class-perturbed-by-linear-functions}) is a simple consequence from $i)$ and $ii)$. Let $\varphi\in W_{loc}^{2,n}(\Omega)$ and suppose $(u+L)-\varphi = u-(\varphi-L)$ has a local maximum at $x_{0}\in\Omega$. Observe that $\bar\varphi=\varphi-L\in W_{loc}^{2,n}(\Omega).$ Furthermore,
\begin{equation}\label{Pucci-operator-passage}
\mathcal{P}_{\gamma}^{+}[\varphi-{L}] \leq \mathcal{P}_{\gamma}^{+}[\varphi] + \gamma(x)|\nabla L| \quad \textnormal { a.e. } x \textnormal { in } \Omega.
\end{equation} 

$$ \mathcal{P}_{\gamma}^{+}[\varphi](x) - f(x) +\gamma(x)|\nabla L|  =  \Big(\mathcal{P}_{\gamma}^{+}[\varphi](x) -\mathcal{P}_{\gamma}^{+}[\varphi - L] + \gamma(x)|\nabla L| \Big)+ \Big(\mathcal{P}_{\gamma}^{+}[\varphi - L]  - f(x)\Big)$$
The expression in the first parenthesis on the RHS is nonnegative almost everywhere in $\Omega$ by (\ref{Pucci-operator-passage}). This way, 
$$ \mathcal{P}_{\gamma}^{+}[\varphi](x) - f(x) +\gamma(x)|\nabla L|  \geq \mathcal{P}_{\gamma}^{+}[\varphi - L]  - f(x) \quad \textnormal {a.e. in } \ \Omega.$$
Thus, 
$$ ess\limsup\limits_{x\to x_{0}}  \Big(\mathcal{P}_{\gamma}^{+}[\varphi](x) - f(x) +\gamma(x)|\nabla L| \Big)\geq   ess\limsup\limits_{x\to x_{0}} \Big(\mathcal{P}_{\gamma}^{+}[\varphi - L]  - f(x) \Big) \geq 0$$ 
where for the second inequality we used that  $u\in\underline{S}(\gamma;f)$ in $\Omega.$ This finishes the proof of $i)$ and thus the Remark.
\end{remark}

\begin{proposition} [{{\bf Key step - Universal closing of the aperture of the wedge}}] \label{key-lemma}
Let $u \in C^{0}(\overline{B}_{1}^{+}) \cap S^*(\gamma, f)$ in $B_{1}^{+}$ and $f\in L^q(B_{1}^{+})$ with $q > n$ such that  $0\leq u(x)\leq x_n$ in $B_{1}^{+}$. Then, there exists a (small) universal constant $\overline{\varepsilon}_{0}>0$ such that if
\begin{equation} \label{smallness of RHS and gradient-constant}
\gamma + {\Vert f \Vert}_{L^q(B_1^{+})} \leq \overline{\varepsilon}_{0}
\end{equation}
we can find constants $L_{0}, U_{0}$ and $\overline{\delta_0} \in (0, 1)\footnote{Clearly, whenever necessary, we can assume that $\overline{\delta_{0}}\in[3/4,1).$}$ 
\begin{equation} \label{improve the cone 0}
\left \{
\begin{array}{lc}
0 \leq L_{0} \leq U_{0} \leq 1 &  \\\\
U_{0} - L_{0} \leq  \overline{\delta_0}  &
\end{array}
\right.
\end{equation}
such that
\begin{equation} \label{close of the cone}
L_{0} \cdot x_n \leq u(x) \leq U_{0} \cdot x_n  \quad \textrm{ for all } \  x\in B_{1/2}^{+}.
\end{equation}
Precisely, $\overline{\varepsilon_0}$, $\overline{\delta_0}, L_{0}$ and $ U_{0}$ depends on $n, q, \lambda, \Lambda$.
\end{proposition}

\begin{proof} We start by setting
\begin{equation}\label{setting-constants-key-step}
\overline{\varepsilon}_{0}:= \min\Bigg\{\frac{D_{2}}{8\cdot D_{3}\cdot N_{0}}, 1\Bigg\}, \quad \mu_{0}:=\frac{A_{2}}{10} \in(0,1/2), \quad N_{0}:=1+|B_{1}^{+}|^{1/q}.
\end{equation}
where $D_{2}>0$ and $D_{3}>0$ are given in Proposition \ref{Boundary Behaviour ffb}. We observe that since $\gamma\leq \overline{\varepsilon_{0}}\leq 1$, Remark \ref{dependence of universal constants of gamma} ensures that $\mu_{0}$ and $\overline{\varepsilon}_{0}$ depends only on $n, q, \lambda,\Lambda$. From now on, we divide the proof in two cases.\\

\noindent \underline{{\bf Case 1:}} Assume $u\left( \frac{1}{2} e_n \right) \geq \frac{1}{4}$.\\

\noindent By assumption and Proposition \ref{Boundary Behaviour ffb} we have
\begin{equation} \label{first main inequality}
u(x) \geq \Bigg( \frac{D_{2}}{4} - D_{3} \cdot{\Vert f\Vert}_{L^q(B_1^+)} \Bigg) \cdot x_n, \quad \forall x \in B_{1/2}^{+}.
\end{equation}
Then, (\ref{setting-constants-key-step}) implies
\begin{equation} \label{first estimate for epsilon zero}
D_{3} \cdot{\Vert f \Vert}_{L^q(B_1^+)}\leq D_{3} \cdot N_0 \cdot \Big( {\Vert f \Vert}_{L^q(B_1^+)} + \gamma \Big)\leq D_{3} \cdot N_{0} \cdot \overline{\varepsilon_{0}}  \leq \frac{D_{2}}{8}.
\end{equation}
Thus, 
\begin{equation} \label{improve of the cone I}
u(x) \geq \frac{D_{2}}{8} \cdot x_n \geq \mu_{0}\cdot x_{n}, \ \ \ \ \forall \ x \in B_{1/2}^+.
\end{equation}
\noindent We set $L_{0} := \mu_{0}$ \ and \ $U_{0} := 1$. It proves the result in this case for $\overline{\delta_{0}} := 1-\mu_{0} \in(0,1)$.\\

\noindent \underline{{\bf Case 2:}} Assume $u\left( \frac{1}{2} e_n \right) < \frac{1}{4}$.\\

\noindent  Then we define $w(x):= x_n - u(x)$ ~for $x\in B_{1}^{+}$. Remark \ref{perturbation-by-linear} implies
$$w \in S^*(\gamma, |f| + \gamma) \ \ \textrm{in} \ \ B_1^+, \quad w \left(\frac{1}{2} e_n \right) \geq \frac{1}{4} \ \ \textrm{ and } \ \ 0 \leq w(x) \leq x_n, \quad \forall x\in B_{1}^{+}.$$
Once more, by assumption and Proposition \ref{Boundary Behaviour ffb}, we have for every $x \in B_{1/2}^+$
\begin{equation} \label{second main inequality}
w(x) \geq \Bigg( \frac{D_{2}}{4} - D_{3} \cdot {\Vert |f| + \gamma \Vert}_{L^q(B_1^+)} \Bigg) \cdot x_n.
\end{equation}
 Now, (\ref{setting-constants-key-step}) implies 
\begin{equation} \label{second estimate for epsilon zero}
D_{3} \cdot{\Vert |f| +\gamma \Vert}_{L^q(B_1^+)}\leq D_{3} \cdot N_0 \cdot \Big( {\Vert f \Vert}_{L^q(B_1^+)} + \gamma \Big)\leq D_{3}\cdot N_{0}\cdot \overline{\varepsilon_{0}}  \leq \frac{D_{2}}{8}.
\end{equation}
This way, as before, 
\begin{equation} \label{improve of the cone II}
w(x) \geq \mu_0 \cdot x_n, \ \ \ \ \forall \ x \in B_{1/2}^{+}.
\end{equation}
Now from definition of $w$ we obtain
$$ 0 \leq u(x) \leq \big(1-\mu_{0}\big)\cdot x_{n}, \quad \forall \ x \in B_{1/2}^{+}.$$
We define $L_{0} := 0 $ \ and \ $U_{0} :=\big(1-\mu_{0}\big)$. Once again, the results holds for $\overline{\delta_0} = 1 - \mu_{0}$ and this finishes the proof. 
\end{proof}

\begin{proposition}[{{\bf Normalized version of Lipschitz implies $C^{1,\alpha}$ on the boundary}}] \label{iteration-scheme}
Let $u \in C^{0}(\overline{B}_{1}^{+}) \cap S^*(\gamma, f)$ in $B_{1}^+$ and $f\in L^q(B_{1}^+)$ with $q > n$. Assume, $|u(x)|\leq x_{n}$ in $B_{1}^{+}$. Then, there exists $\varepsilon_0\in(0,1)$ universal constant so that if
\begin{equation} \label{smallness of RHS and gradient-constant}
\gamma + {\Vert f \Vert}_{L^q(B_1^+)} \leq \varepsilon_0,
\end{equation}
we can find constants $A_k, B_k$ and a universal $\delta_0 \in (0, 1)$ satisfying
\begin{equation} \label{improve the cone}
\left \{
\begin{array}{lc}
A_{0}:=-1\leq \cdots \leq A_{k - 1} \leq A_k \leq \cdots \leq B_k \leq B_{k - 1} \leq \cdots \leq 1=:B_{0}, &  \\\\
 B_k - A_k \leq {\delta_0}^k \cdot (B_0 - A_0),  \quad \forall k\geq 0.

\end{array}
\right.
\end{equation}
such that
\begin{equation}\label{trapping}
A_k \cdot x_n \leq u(x) \leq B_k \cdot x_n, \ \ \ \ \textrm{ in } \ B_{2^{-k}}^+.
\end{equation}
In particular, for $0<r\leq 1,$
\begin{equation}\label{oscillation-decay}
 \underset{B_{r}^{+}}{osc}~\Bigg(\frac{u}{x_{n}}\Bigg)\leq \overline{E_{0}} \cdot r^{\alpha_{0}} \quad \textnormal { where } \quad \overline{E_{0}}=2\delta_{0}^{-1}\leq 3.
\end{equation}

\noindent This implies that there exist $\Psi_0 \in {\mathbb{R}}$ and $\alpha_0 \in (0, 1)$ such that  
\begin{equation} \label{normal derivative estimate in the zero}
\vert u(x) - \Psi_0 \cdot x_n \vert \leq E_{0} \vert x \vert^{\alpha_0}x_{n} \quad \forall x \in B_{1/2}^{+}, \qquad \vert \Psi_0 \vert \leq 1.
\end{equation}
Precisely, $\alpha_0 \in (0,1) , \varepsilon_{0}, \delta_{0}, E_{0}$ and $\overline{E_{0}}$ depend only on  $n, \lambda, \Lambda$ and $q$.
\end{proposition}

\begin{proof} We recall  $\overline{\varepsilon_{0}}$ and $\overline{\delta_{0}}$ from Proposition \ref{key-lemma} and set the following constants
\begin{equation}\label{new-constants-for-clsoing-lemma}
 K_{0}:=3\big(1+|B_{1}^{+}|^{1-n/q}\big), \quad \varepsilon_{0}:= \frac{\overline{\varepsilon}_{0}}{ K_{0}} < \overline{\varepsilon}_{0}, \quad  \delta_{0}:=\frac{\max\Big\{2^{n/q-1},\overline{\delta_{0}}\Big\} +1}{2}. 
 \end{equation}
\noindent This way, $\varepsilon_{0}=\varepsilon_{0}(n,q,\lambda, \Lambda)$  and $\delta_{0}=\delta_{0}(n,q,\lambda, \Lambda)$ and both of them are in $(0,1).$\\

\noindent Now, recall that for $0<a<b \Longrightarrow (a+b)/2 \in (a,b)$. This way, since $\overline{\delta_{0}}\in (0,1)$

$$q>n  \Longrightarrow  \max\Big\{2^{n/q-1},\overline{\delta_{0}}\Big\} < 1 \Longrightarrow  \delta_{0} \in \Bigg(\max\Big\{2^{n/q-1},\overline{\delta_{0}}\Big\}, 1\Bigg).$$
From this, we conclude
\begin{equation}\label{delta-less-delta-bar+crucial-delta-relation-n/q}
 \zeta_{0}:=\frac{2^{n/q-1}}{\delta_{0}} \in(0,1)  \quad \textnormal { and } \quad \overline{\delta_{0}}<\delta_{0}\in (0,1).
\end{equation}

\noindent We argue by induction. For the first step, we define the following renormalized function
$$ v(x) := \frac{u(x)+x_{n}}{2} \quad \textnormal { for } \quad  x\in B_{1}^{+}.$$
It is immediate that $0\leq v(x) \leq x_{n}$ in $B_{1}^{+}$ and Remark \ref{perturbation-by-linear} gives
$$v\in S^{*}\Big(\gamma; \frac{|f|+\gamma}{2}\Big) \:\textnormal{ in }B_{1}^{+}\: \textnormal { with } \ \  \Bigg|\Bigg|\frac{|f|+\gamma}{2} \Bigg|\Bigg|_{L^{q}(B_{1}^{+})} \leq K_{0} \cdot\Big(||f||_{L^{q}(B_{1}^{+})}  + \gamma\Big) \leq \overline{\varepsilon}_{0}.$$
Now,  we can apply directly Proposition \ref{key-lemma} to obtain $L_{0}\cdot x_{n} \leq v(x) \leq U_{0}\cdot x_{n}$ in $B_{1/2}^{+}.$ Setting 
$$ A_{1}:=2L_{0}-1, \quad B_{1}:=2U_{0}-1\quad \textnormal { we have }$$ 
$$ A_{1}\cdot x_{n} \leq u(x)\leq B_{1/2}^{+}\cdot x_{n} \quad \textnormal{ in } \: B_{1}^{+} \:\textnormal { with }\quad  B_{1}-A_{1}=2\cdot (U_{0}-L_{0})\leq 2\cdot\overline{\delta_{0}} \leq \delta_{0}\cdot (B_{0}-A_{0}).  $$

\noindent This proves the first inductive step. Now, we assume that the estimates in (\ref{improve the cone}) and (\ref{trapping}) hold true for all the steps $j\leq k$. We prove the step $j=k + 1$. Indeed, we set
$$u_k (x):= \frac{u(2^{-k} x)}{2^{-k}} \quad \textnormal { for } \quad x\in B_{1}^{+}.$$
Now,  $u_k \in S^*(\gamma_k, f_k)$ in $B_1^+$ where $\gamma_k = 2^{-k} \gamma$ and $f_k(x) = 2^{-k} f(2^{-k} x)$ for $x\in B_{1}^{+}$. Also,
\begin{equation}\label{kth-step-induction-estimate-ball}
A_k \cdot x_n \leq u_k(x) \leq B_k \cdot x_n, \ \ \ \ \ \forall \ x \in B_{1}^{+} \quad \textnormal { and } \quad B_{k}-A_{k}\leq \delta_{0}^{k} \cdot (B_{0}-A_{0}).
\end{equation}
Define
$$ v_{k}(x):=\Bigg(\frac{u_{k}(x)-A_{k}\cdot x_{n}}{2\delta_{0}^{k}}\Bigg) \quad \textnormal { for } \: x\in B_{1}^{+}. $$
Clearly, by (\ref{kth-step-induction-estimate-ball}) $0\leq v_{k}(x)\leq x_{n}$ in $B_{1}^{+}$. Remark \ref{perturbation-by-linear} once more implies
$$v_{k}\in S^{*}\Bigg(\gamma_{k}, \frac{|f_{k}| + \gamma_{k}\cdot |A_{k}|}{2\delta_{0}^{k}}\Bigg) \quad \textnormal { in } \: B_{1}^{+}.$$
This way, since $2^{-k} \leq 2^{-k(1-n/q)}$, we have

\begin{eqnarray*}
\gamma_{k} + \Bigg|\Bigg| \frac{|f_{k}| + \gamma_{k}\cdot |A_{k}|}{2\delta_{0}^{k}} \Bigg|\Bigg|_{L^q(B_1^+)} &\leq& 2^{-k}\gamma +\frac{2^{-k(1 - n/q)}\cdot {\Vert f \Vert}_{L^q(B_1^+)} + 2^{-k}\cdot\gamma\cdot |B_{1}^{+}|^{1-n/q} }{2\delta_{0}^{k}} \nonumber\\ \nonumber 
&\leq & \frac{2^{-k(1 - n/q)}}{2\delta_{0}^{k}}\Bigg( {\Vert f \Vert}_{L^q(B_1^+)}  + \big(2\cdot\delta_{0}^{k} + |B_{1}^{+}|^{1-n/q}\big)\cdot \gamma \Bigg)\nonumber \\ \nonumber \\ \nonumber
& \leq & \frac{2^{k(n/q - 1)}}{2\delta_{0}^{k}}\cdot K_{0} \cdot \Big(||f||_{L^q(B_1^+)} + \gamma \Big) \nonumber \\ \nonumber \\ \nonumber
& \leq & \frac{1}{2} \cdot \zeta_{0}^{k} \cdot {K_{0}\cdot \varepsilon_{0}} \quad  \nonumber \\ \nonumber \\ \nonumber
&\leq & \frac{\overline{\varepsilon}_{0}}{2} \quad (\textnormal {since } \ \zeta_{0}\in(0,1) \ \textnormal { by } \ (\ref{delta-less-delta-bar+crucial-delta-relation-n/q}) \ \textnormal { and definition of } \varepsilon_{0} \ \textnormal { in} \ (\ref{new-constants-for-clsoing-lemma})).
\end{eqnarray*}
Then, by Propositon \ref{key-lemma}, 
$$ L_{0} \cdot x_n \leq v_{k}(x) \leq U_{0} \cdot x_n  \quad \forall\  x\in B_{1/2}^{+}.$$
\noindent This translates to
$$ A_{k+1}\cdot x_{n} \leq u(x)\leq B_{k+1}\cdot x_{n} \quad \textnormal{ in } \: B_{2^{-(k+1)}}^{+} \quad \textnormal { where }$$

$$  A_{k+1}:=2\delta_{0}^{k}L_{0}+A_{k}, \quad B_{k+1}:=2\delta_{0}^{k}U_{0}+A_{k},$$ 
\noindent with 
$$B_{k+1}-A_{k+1} = 2\delta_{0}^{k}\cdot (U_{0}-L_{0}) \leq 2\delta_{0}^{k}\cdot \overline{\delta_{0}} \leq \delta_{0}^{k+1}\cdot (B_{0}-A_{0}), $$
since $\overline{\delta_{0}} \leq \delta_{0}$ by (\ref{delta-less-delta-bar+crucial-delta-relation-n/q}).\\

\noindent This finishes the proof of the inductive process. By monotonicity there exists $\Psi_0 \in {\mathbb{R}}$ such that
$$\lim_{k \rightarrow \infty} A_k = \Psi_0 = \lim_{k \rightarrow \infty} B_k \quad \textnormal { and } \quad \vert \Psi_0 \vert \leq 1.$$
Now set $\alpha_0 := -\log_{2}\delta_{0}>0$. As pointed out before, we can assume that $\overline{\delta_{0}} > 3/4$. This way, by (\ref{delta-less-delta-bar+crucial-delta-relation-n/q}), we have $\delta_{0}>3/4.$ Thus $\alpha_0$ $\in (0,1)$. Consider $x \in B_{1/2}^+$. Then, there exists $k \geq 1$ such that $2^{-(k + 1)} <\vert x \vert \leq 2^{-k}$. This way,
\begin{eqnarray}
u(x) - \Psi_0 \cdot x_n & =& u(x)+\Big(A_{k}-\Psi_{0}\Big)\cdot x_{n} - A_{k}\cdot x_{n} \nonumber\\ \nonumber
&\leq & \Big(B_{k}-A_{k}\Big)\cdot x_{n} \\ \nonumber
& \leq & 2\cdot\delta_{0}^{k}\cdot x_{n} = 2^{\alpha_0 + 1} (2^{-(k+1)})^{\alpha_0}\cdot x_{n} \\ \nonumber \vspace{.3cm}
& \leq & 2^{\alpha_0 + 1}\cdot |x|^{\alpha_0}\cdot x_{n}
\end{eqnarray}
Similarly, we prove for $x\in B_{1/2}^{+}$ that $u(x) -\Psi_{0}\cdot x_{n} \geq -2^{\alpha_0 + 1}\cdot |x|^{\alpha_0}\cdot x_{n}$. Hence, 
$$|u(x) - \Psi_{0}\cdot x_{n}| \leq 2^{\alpha_0 + 1}|x|^{\alpha_0}\cdot x_{n}, \quad \forall x\in B_{1/2}^{+}. $$
We can take $E_{0}:=2^{\alpha_0 + 1}$. Finally, observe that for $0<r<1$ we can find $k\in\mathbb{N}$ so that $2^{-(k+1)} < r \leq 2^{-k}.$ This way, setting 
$$\psi(r)=\underset{B_{r}^{+}}{osc} ~\frac{u}{x_{n}},$$
we have by (\ref{improve the cone})
$$\psi(r) \leq \psi(2^{-k})\leq B_{k}-A_{k}\leq 2\delta_{0}^{k}= 2\delta_{0}^{-1}\delta_{0}^{k+1}\leq 2\delta_{0}^{-1} 2^{-(k+1)\alpha_{0}}\leq \overline{E_{0}}\cdot r^{\alpha_{0}} $$
which proves (\ref{oscillation-decay}). This finishes the proof.
\end{proof}

\begin{remark}\label{extension-of-C-{1,alpha}-inequality-to-B1} It is trivial to see that (\ref{normal derivative estimate in the zero}) in Proposition \ref{iteration-scheme} implies
\begin{equation}\label{estimate-in-whole-B1plus}
\vert u(x) - \Psi_0 \cdot x_n \vert \leq E_{0} \vert x \vert^{1+\alpha_0} \quad \forall x \in B_{1/2}^{+}.
\end{equation}
\noindent We can easily observe, that perhaps changing the constant $E_{0}$, we can make it holds for the whole $B_{1}^{+}.$ Indeed, set $\overline{C}:= 2^{2+\alpha_0}.$ Now, (\ref{estimate-in-whole-B1plus}) holds with $C$ replaced by $\overline{E}_{0}:=\max\big\{E_{0}, \overline{C}\big\}$ for every $x\in B_{1}^{+}.$ To see this, it is enough to check it just outside $B_{1/2}^{+}.$ So, for $x\in B_{1}^{+}\setminus B_{1/2}^{+}$ we have 
$$|u(x) - \Psi_{0}\cdot x_{n}|\leq 2 = 2^{2 + \alpha_0} \cdot 2^{-{(1 + \alpha_0})}  \leq 2^{2+{\alpha_0}}\cdot |x|^{1+\alpha_0}.$$
\end{remark}

\begin{center}
{\bf \underline{Proof of Theorem \ref{boundary krylov thm Lq version}}}
\end{center}

\begin{proof} By scaling, it is enough to prove the result for $r=1$. We make the following \vspace{1mm}

\noindent \underline{\bf Claim:} There exists a constant $G_0 $ such that
\begin{equation}\label{diff-at-zero}
\left\vert u(x) - G_0 \cdot x_n \right\vert \leq \overline{E}_{1} \Big( {\Vert u \Vert}_{L^{\infty}(B_1^+)} + {\Vert f \Vert}_{L^q(B_1^+)} \Big) \vert x \vert^{\alpha_0}x_{n} \quad \forall x\in B_{1/2}^{+}.
\end{equation}
and 
\begin{equation}\label{holder-gradient}
\left\vert G_0 \right\vert \leq \overline{E}_{1} \Big( {\Vert u \Vert}_{L^{\infty}(B_1^+)} + {\Vert f \Vert}_{L^q(B_1^+)} \Big),
\end{equation}
and 
\begin{equation}\label{scaled-oscillation}
 \underset{B_{r}^{+}}{osc}~\Bigg(\frac{u}{x_{n}}\Bigg)\leq  \widehat{E_{0}} \cdot \Big(||u||_{L^{\infty}(B_{1}^{+})} + ||f||_{L^{q}(B_{1}^{+})} \Big) \cdot r^{\alpha_{0}} \quad \textnormal { for } 0<r<\bar{\varepsilon}, 
 \end{equation}
where $\bar{\varepsilon}, \widehat{E_{0}}$ and $\overline{E_{1}}$ are positive universal constants depending only on $n,q,\lambda,\Lambda,\gamma.$\\

\noindent \underline{Proof of the Claim:} Set $K:=D_{1} \Big(||u||_{L^{\infty}(B_{1}^{+})} + ||f||_{L^{q}(B_{1}^{+})} \Big)$. Then, by Proposition \ref{bdry-lip-type-estimate-scaled-version}
\begin{equation}\label{first-cone-trapping}
|u(x)| \leq K\cdot x_{n}, \quad \forall x\in B_{1/2}^{+}. 
\end{equation}
\noindent Now, we set  a universal constant given by
$$ \overline{\varepsilon}: = \min\Bigg\{\frac{1}{4}, ~~\Bigg(\frac{\varepsilon_{0}}{2(\gamma + D_{1})}\Bigg)^{\frac{1}{1-n/q}} \Bigg\}.$$
\noindent Since $1-n/q\in(0,1]$ and $\bar{\varepsilon}\in(0,1)$ we have $\bar{\varepsilon} \leq \bar{\varepsilon}^{1-n/q}$. Thus, 
\begin{equation}\label{normalizing-scaling} \bar{\varepsilon}\gamma + (\bar{\varepsilon})^{1-n/q}\cdot D_{1}^{-1} \leq 2(\bar{\varepsilon})^{1-n/q}(\gamma + D_{1}^{-1}) \leq \varepsilon_{0}.
\end{equation}
\noindent Define,
$$ v(x) :=  \frac{u(\bar{\varepsilon}x)}{\bar{\varepsilon}\cdot K}\quad \textnormal { for } \quad x\in B_{1}^{+}.$$
\noindent Now, since $ \bar{\varepsilon}< 1/2$, by (\ref{first-cone-trapping}) we observe that $|v(x)| \leq x_{n}$ in $B_{1}^{+}$,
$$v\in S^{*}\Big(\bar{\gamma}; \bar{f}\Big) \quad \textnormal { where } \quad \bar{\gamma} = \overline{\varepsilon}\cdot \gamma, \quad \bar{f}(x) = \frac{\bar{\varepsilon}\cdot f(\bar{\varepsilon}x)}{K}\quad \textnormal { for } \quad x\in B_{1}^{+}. $$
Hence, we have by (\ref{normalizing-scaling})
$$ \bar{\gamma} + ||\bar{f}||_{L^{q}(B_{1}^{+})}  =\bar{\varepsilon}\gamma + (\bar{\varepsilon})^{1-n/q}\cdot D_{1}^{-1} \leq \varepsilon_{0}.$$
\noindent Applying Proposition \ref{iteration-scheme} to $v$ and translating back in terms of  $u$ we conclude, 
\begin{equation}\label{differentiability-tiny-scale}
\vert u(x) - G_{0} \cdot x_n \vert \leq E_{0}^{*} \cdot K\cdot \vert x \vert^{\alpha_0}x_{n} \quad \forall x \in B_{\bar{\varepsilon}/2}^{+} \quad \textnormal { with }
\end{equation}
$$G_{0}:= \Psi_{0}\cdot K, \quad E_{0}^{*}=E_{0}\cdot(\bar{\varepsilon})^{-\alpha_0} \quad\textnormal { and thus }\quad  |G_{0}|\leq K. $$

\noindent Setting $C^{*}=2(2/\bar{\varepsilon})^{\alpha_0}$ we have by Proposition \ref{bdry-lip-type-estimate-scaled-version}
that for $x\in B_{1/2}^{+}\setminus B_{\varepsilon/2}^{+}$, 
\begin{equation}\label{differentiability-big-scale}
|u(x) - G_{0}\cdot x_{n}| \leq 2K\cdot x_{n} = C^{*}\Big(\frac{\bar{\varepsilon}}{2}\Big)^{\alpha_0}\cdot K \cdot x_{n} \leq C^{*}K|x|^{\alpha_0}x_{n}.
\end{equation}
Furthermore,
\begin{equation}\label{scaled-oscillation-proof}
 \underset{B_{r}^{+}}{osc}~\Bigg(\frac{u}{x_{n}}\Bigg)\leq \big(\bar{\varepsilon}K \overline{E_{0}}\big) \cdot r^{\alpha_{0}} \quad \textnormal { for } \ 0<r<\bar{\varepsilon} \quad \textnormal { and } \ \overline{E_{0}} \ \textnormal{ as in } (\ref{oscillation-decay}).
\end{equation}
\noindent Thus, setting $\widehat{E_{0}} = \bar{\varepsilon}\overline{E_{0}}D_{1}$ and $\overline{E}_{1}:= D_{1}\max\Big\{E_{0}^{*},C^{*},1 \Big\}\geq D_{1}$ the claim is proven.

\noindent Let $x_{0}\in B'_{1/2}.$ By considering $v_{0}(x):=4u(x_{0}+x/4)$ for $x\in B_{1}^{+}$ we see 
\begin{equation}\label{translation-argument}
v_{0}\in S^{*}(\gamma/4;f_{0})\subset S^{*}(\gamma;f_{0}) \quad \textnormal{ where } \quad f_{0}(x):=4^{-1}f(x_{0}+x/4) \quad \forall x\in B_{1}^{+}.
\end{equation}
\noindent Applying the claim to $v_{0}$ and translating the results back to $u$ we conclude, $\forall x_{0}\in B'_{1/2}, \forall x\in B_{1/8}^{+}(x_{0})$,
\begin{equation}\label{krylov-bdry-zero-final-1}
|u(x) - A(x_{0})\cdot x_{n}|\leq 4^{\alpha_{0}+1}\overline{E}_{1}\Big(||u||_{L^{\infty}(B_{1}^{+})} + ||f||_{L^{q}(B_{1}^{+})} \Big)|x-x_{0}|^{\alpha_0}\cdot x_{n},
\end{equation}
\begin{equation}\label{krylov-bdry-zero-final-2}
|A(x_{0})|\leq 4\overline{E}_{1}\cdot \Big(||u||_{L^{\infty}(B_{1}^{+})} + ||f||_{L^{q}(B_{1}^{+})} \Big),
\end{equation} 
and
\begin{equation}\label{final-estimate-oscillation}
 \underset{B_{r}^{+}(x_{0})}{osc}~\Bigg(\frac{u}{x_{n}}\Bigg)\leq \widehat{E_{0}} \cdot \Big(||u||_{L^{\infty}(B_{1}^{+})} + ||f||_{L^{q}(B_{1}^{+})} \Big) \cdot r^{\alpha_{0}} \quad \textnormal { for } \ 0<r<\frac{\bar{\varepsilon}}{4}.
 \end{equation}
Once more, Lipschitz estimates up to the boundary, Proposition \ref{bdry-lip-type-estimate-scaled-version} and Remark \ref{bdry-lip-type-estimate-scaled-version-extended} gives for $x\in B_{3/4}^{+}\setminus B_{1/8}^{+}$ 
\begin{eqnarray}\label{krylov-bdry-zero-final-3}
|u(x) - A(x_{0})\cdot x_{n}| &\leq&  D_{1}\Big(||u||_{L^{\infty}(B_{1}^{+})} + ||f||_{L^{q}(B_{1}^{+})} \Big)\cdot x_{n}\\ \nonumber
& \leq & 8^{-\alpha_0} \Bigg(\frac{D_{1}}{8^{-\alpha_0}}\Bigg) \Big(||u||_{L^{\infty}(B_{1}^{+})} + ||f||_{L^{q}(B_{1}^{+})} \Big)\cdot x_{n} \\ \nonumber
& \leq & \Bigg(\frac{D_{1}}{8^{-\alpha_0}}\Bigg) \Big(||u||_{L^{\infty}(B_{1}^{+})} + ||f||_{L^{q}(B_{1}^{+})} \Big)|x|^{\alpha_0}x_{n}
\end{eqnarray}
Set 
\begin{equation}\label{E1star}
E_{1}^{*} := \max\Big\{ 4^{\alpha_0+1}\overline{E}_{1}, \frac{D_{1}}{8^{-\alpha_0}} \Big\}\geq 4\overline{E}_{1}. 
\end{equation}
Now, the H\"older (gradient) estimate for $||A||_{C^{0,\alpha_{0}}(B_{1/2}^{+})}$follows directly from (\ref{krylov-bdry-zero-final-1}), (\ref{krylov-bdry-zero-final-2}), (\ref{E1star}) and Lemma \ref{taylor-2} by taking $T=E_{1}^{*}$ and $r_{0}=1/8.$\\

\noindent  Now, we prove the H\"older estimate $||u/x_{n}||_{C^{0,\alpha_{0}}(B_{1/2}^{+})}$. We follow the very nice ideas presented in the proof of Theorem 1.2 in \cite{R-O-Se-1}. We define $Q(x):=u(x)/x_{n}$ for $x\in B_{1}^{+}.$ Let $x\in B_{1/2}^{+}$. Let us denote $d_{x}=dist(x, B'_{1})$. By Krylov-Safonov  (interior) H\"older estimate, there exists a universal $\beta_{0}=\beta_{0}(n,\lambda,\Lambda, \gamma)\in (0,1)$ so that  
\begin{equation}\label{Krylov-Safonov-Interior}
d_{x}^{\beta_{0}}[u]_{C^{\beta_{0}}(B_{d_{x}/2}(x))} \leq C_{0}\Big(||u||_{L^{\infty}(B_{1}^{+})} + ||f||_{L^{q}(B_{1}^{+})}\Big)=:C_{0}M_{0}
\end{equation}
where $C_{0}=C_{0}(n,q, \lambda, \Lambda, \gamma)>0$. It is easy to check that  for any $\alpha\in (0,1)$ 
\begin{equation}\label{holder-x_{n}}
||1/x_{n}||_{L^{\infty}(B_{d_{x}/2}(x))}\leq 2d_{x}^{-1}, \quad [1/x_{n}]_{C^{\alpha}(B_{d_{x}/2}(x))} \leq 2d_{x}^{-(1+\alpha)} 
\end{equation}
Now, using the product estimate for the H\"older semi-norm $($for any $\alpha \in(0,1))$
$$[fg]_{C^{}(\Omega)} \leq  [f]_{C^{\alpha}(\Omega)}||g||_{L^{\infty}(\Omega)} + ||f||_{L^{\infty}(\Omega)}[g]_{C^{\alpha}(\Omega)}.$$
we conclude from (\ref{Krylov-Safonov-Interior}) and (\ref{holder-x_{n}}) that 
\begin{equation}\label{holder-seminorn-interior-quotient}
[Q]_{C^{\beta_{0}}(B_{d/2}(x))} \leq (2+2C_{0})M_{0}d_{x}^{-(1+\beta_{0})} = \overline{C_{0}}M_{0}d_{x}^{-(1+\beta_{0})}
\end{equation}
We observe also that (\ref{taylor-general-krylov}) implies that $Q$ is defined in $B'_{1/2}$ and in fact $Q\equiv A$ in $B'_{1/2}.$ As a matter of fact, from (\ref{taylor-general-krylov}) and (\ref{holder-estimate-krylov}), we have
\begin{equation}\label{vertical-estimate-boundary-point-frozen}
|Q(x) - A(z_{0})| \leq E_{1}M_{0}||x-z_{0}|^{\alpha_{0}} \quad\forall z_{0}\in B'_{1/2}, \forall x\in B_{1/2}^{+}
\end{equation}
\begin{equation}\label{Holder-Q-boundary}
|Q(z_{1})-Q(z_{2})| \leq E_{1}M_{0}|z_{1}-z_{2}|^{\alpha_{0}} \quad \forall z_{1},z_{2} \in B'_{1/2}
\end{equation}
Now, we are ready to prove the estimate. Let $x,y\in B_{1/2}^{+}$. We set  
$$r:=|x-y|, \quad d_{x}=d(x)=|x-x_{0}|, \quad d_{y}=d(y)=|x-y_{0}|, \quad x_{0}, y_{0}\in B'_{1/2} $$ 
We assume without losing generality that $d_{y}\leq d_{x}$. In what follows let $p\geq 1$ to be chosen a posteriori. We then analyze two cases\\

\noindent \underline{\bf Case I:} $r\geq d_{x}^{p}/2$\\

\noindent Now by (\ref{vertical-estimate-boundary-point-frozen}) and (\ref{Holder-Q-boundary}) we estimate
\begin{eqnarray*}
|Q(x)-Q(y)| &\leq&  |Q(x)-Q(x_{0})| + |Q(x_{0})-Q(y_{0})| + |Q(y_{0})-Q(y)| \\\\
&\leq & E_{1}M_{0}(d_{x}^{\alpha_{0}} + |x_{0}-y_{0}|^{\alpha_{0}} + d_{y}^{\alpha_{0}}) \\ \\
&\leq & E_{1}M_{0}(2d_{x}^{\alpha_{0}} + (d_{x} + r + d_{y})^{\alpha_{0}})\\\\
&\leq & E_{1}M_{0}(5d_{x}^{\alpha_{0}}+ r^{\alpha_{0}}) \\ \\
&\leq & 5(2^{\frac{\alpha_{0}}{p}}+1) E_{1}M_{0}\cdot  r^{\frac{\alpha_{0}}{p}}\\ \\
& = & 5(2^{\frac{\alpha_{0}}{p}}+1) E_{1}M_{0}\cdot  |x-y|^{\frac{\alpha_{0}}{p}}
\end{eqnarray*}
\noindent \underline{\bf Case II:} $r< d_{x}^{p}/2$\\

\noindent Set $\xi:=1+\beta_{0}$. Then by (\ref{holder-seminorn-interior-quotient}) 
$$|Q(x)-Q(y)| \leq \overline{C_{0}}M_{0}d_{x}^{-\xi}\cdot r^{\beta_{0}} \leq  2^{1-\frac{\xi}{p}}= \overline{C_{0}}M_{0}d_{x}^{-\xi}\cdot r^{\beta_{0}} =  2^{1-\frac{\xi}{p}}\overline{C_{0}}M_{0} |x-y|^{\beta_{0}-\frac{\xi}{p}} $$
We just choose any $p\geq 1$ for which $\beta_{0}-\xi/p>0$. In fact, $p=\beta_{0}^{-1}+2$ does the job. This way, we proved
$$x,y\in B_{1/2}^{+} \Longrightarrow  |Q(x)-Q(y)| \leq C_{00}M_{0}|x-y|^{\tau_{0}} \quad \textnormal {where } \ \tau_{0}=\min\Bigg\{\frac{\alpha_{0}}{p}, \beta_{0}-\frac{\xi}{p}\Bigg\}$$ 
for $C_{00}>0$ and $\tau_{0}$ depending only on $n,q, \lambda, \Lambda, \gamma$. This finishes the proof.
\end{proof}
\begin{center}\label{proof-PLT}
{\bf \underline{Proof of Remark \ref{PLT}} - Phargm\'en-Lindel\"of type result}
\end{center}
\begin{proof} Let us consider the function $u_{r}(x)=r^{-\beta}u(rx)$ for $x\in \mathbb{R}_{+}^{n}$. It is immediate to check that $u_{r}$ satisfies the same growth condition in (\ref{growth-condition-PL}). Moreover, from Remark \ref{scaling-remark}, $u_{r}\in S(0).$ This way,  by Theorem \ref{boundary krylov thm Lq version}, we obtain for $C=C(n,\lambda, \Lambda)>0$
\begin{eqnarray*}
\Bigg[\frac{u}{x_{n}}\Bigg]_{C^{0,\alpha_{00}}(B_{r}^{+})} &= &r^{-\alpha_{00}} \Bigg[\frac{u(rx)}{rx_{n}}\Bigg]_{C^{0,\alpha_{00}}(B_{1}^{+})}\\
& = & r^{-\alpha_{00}-1} \Bigg[\frac{u(rx)}{x_{n}}\Bigg]_{C^{0,\alpha_{00}}(B_{1}^{+})}\\
& = &  r^{-\alpha_{00}-1+\beta} \Bigg[\frac{u_{r}(x)}{x_{n}}\Bigg]_{C^{0,\alpha_{00}}(B_{1}^{+})}\\
& \leq &  C\cdot   r^{-\alpha_{00}-1+\beta}\cdot ||u_{r}||_{L^{\infty}(B_{2}^{+})}  \quad (\textnormal {by estimate } (\ref{krylov-uraltseva-oscillation-estimate}))    \\\\
& \leq & 2CC_{0}  r^{-\alpha_{00}-1+\beta} \to 0 \quad \textnormal { as } r\to\infty.
\end{eqnarray*}
This way, 
$$ \Bigg[\frac{u}{x_{n}}\Bigg]_{C^{0,\alpha_{00}}(\mathbb{R}_{+}^{n})} =0 \Longrightarrow u(x)=Kx_{n} \ \textnormal { in } \ \mathbb{R}_{+}^{n} \quad \textnormal { for some constant}~~ K\in\mathbb{R}.$$
In the case $u\geq 0$, it follows from the results in \cite{Braga-Moreira-Carleson}, that $u(x)\leq C_{0}|x|$ for $x\in \mathbb{R}_{+}^{n}$ where $C_{0}>0$.  The result follows since in this case we can take $\beta=1 < 1 + \alpha_{00}.$
\end{proof}
\section{Improvement of Flatness}

\noindent Our next goal is to extend the previous Theorem (zero boundary data) to arbitrary $C^{1, Dini}$-boundary data on the flat boundary for equations involving unbounded coefficients. In order to do that, we prove a (new) version of improvement of flatness that contemplates the case whrere $\gamma, f\in L^{q}.$ As pointed out before, because of the low regularity of the coefficients, there is no envelope class for this equation.

\begin{proposition}[{{\bf Improvement of flatness}}] \label{improvement-of-flatness}
Let $u \in C^0(\overline{B}_1^+) \cap S^*(\gamma, f)$ in $B_1^+$ where $\gamma, f \in L^{q}(B_1^+)$ with $q>n$. Let $u_{\mid_{B'_r}} = \varphi$ be the boundary data on the flat boundary and $0 \leq \alpha < \alpha_{00}$. This way, for all $\mu_{*} \in (0, \mu_{\alpha})$ we can find (a small) $\varrho_0 = \varrho_0(\alpha, \mu_{*}) > 0$ such that if
\begin{equation} \label{key small conditions}
{\Vert u \Vert}_{L^{\infty}(B_1^+)} \leq 1 \ \ \ \textrm{ and } \ \ \ \ \ {\Vert \gamma \Vert}_{L^{q}(B_1^+)}+ {\Vert f \Vert}_{L^{q}(B_1^+)} + {\Vert \varphi \Vert}_{L^{\infty}(B'_1)} \leq \varrho_0,
\end{equation}
there exist $G_0 \in [-F_{0},F_{0}]$ such that
\begin{equation} \label{contradiction argument}
\big|\big| u(x) - G_0 \cdot x_n \big|\big|_{L^{\infty}(B_{\mu_{*}}^{+})} \leq \frac{1}{4} \mu_{*}^{1 + \alpha}.
\end{equation}
Here, $F_{0}=F_{0}(n, q, \lambda, \Lambda)>0$  is a universal constant. Moreover, $\mu_{\alpha}$ is the universal constant given by 
\begin{equation} \label{choice mu beta}
\mu_{\alpha} := \min \Bigg\{ \frac{2}{3}, \left( \frac{1}{8F_{0}} \right)^{\frac{1}{\alpha_{00} - \alpha}} \Bigg\} \in (0,1).
\end{equation}
\end{proposition}

\begin{proof} We recall from Theorem \ref{boundary krylov thm Lq version} the following: 
\begin{equation}\label{proximal-scale-estimate-for-zero-case}
\left\{ \begin{array}{ll}
\forall v \in S(0) \ \textnormal { in } \ B_{2/3}^{+},\\\\
{\Vert v \Vert}_{L^{\infty}(B_{2/3}^+)} \leq 1, \\\\
v=0 \ \textnormal { on } \: B'_{2/3},
\end{array} \right. \ \Rightarrow \ \left\{ \begin{array}{ll}
\exists ~~A_{v}(0)\in\mathbb{R} \quad \textnormal { so that} \\\\
\left\vert v(x) - A_v(0) \cdot x_n \right\vert \leq F_{0} \vert x \vert^{1 + \alpha_{00}} \ \ \textrm{ for all } \ x \in B_1^+,\\\\
\left\vert A_v(0) \right\vert \leq F_{0} \quad \textnormal { where } \ F_{0}=F_{0}(n,\lambda, \Lambda)>0.
\end{array} \right. 
\end{equation} 
 By the choice of $\mu_{\alpha}$ done in  (\ref{choice mu beta}), we have
 
 \begin{equation}\label{app estimate}
\left\{ \begin{array}{ll}
\forall v \in S(0) \ \textnormal { in } \ B_{2/3}^{+},\\\\
{\Vert v \Vert}_{L^{\infty}(B_{2/3}^+)} \leq 1, \\\\
v=0 \ \textnormal { on } \: B'_{2/3},
\end{array} \right. \ \Rightarrow \ \left\{ \begin{array}{ll}
\exists ~~A_{v}(0)\in\mathbb{R} \quad \textnormal { so that} \\\\
\big|\big| v - A_{v}(0) \cdot x_n \big|\big|_{L^{\infty}(B_{\mu}^{+})} \leq \frac{1}{8} \mu^{1 + \alpha}  \textrm{ for all } \mu \in (0, \mu_{\alpha}),\\\\
\left\vert A_v(0) \right\vert \leq F_{0} \quad \textnormal { where } \ \ F_{0}=F_{0}(n, q, \lambda, \Lambda)>0.
\end{array} \right. 
\end{equation} 
 
 \noindent We now proceed to prove the Proposition \ref{improvement-of-flatness} by contradiction. So, let us suppose the statement of the Proposition is not true. This way, there exist $\mu_{*} \in (0, \mu_{\alpha})$ and a sequence $u_k \in C(\overline{B}_1^+) \cap S^*(\gamma_k, f_k)$ in $B_1^+$ and $\varrho_{k} \to 0$ with
$${\Vert u_k \Vert}_{L^{\infty}(B_1^+)} \leq 1 \ \ \ \textrm{ and } \ \ \ {\Vert \gamma_k \Vert}_{L^{q}(B_1^+)} + {\Vert f_k \Vert}_{L^{q}(B_1^+)} + {\Vert \varphi_k \Vert}_{L^{\infty}(B'_1)} \leq \varrho_k,$$
such that, for each constant $G \in [-F_{0}, F_{0}]$,
\begin{equation} \label{contradiction arg}
\big|\big| u_{k}- G \cdot x_n \big|\big|_{L^{\infty}(B_{\mu_{*}}^{+})} > \frac{1}{4} \mu_{*}^{1 + \alpha} \quad \textnormal { for every } k\geq 1.
\end{equation}
Now, by a Krylov-Safonov H\"older estimate up to the boundary type estimate (Theorem 2 in \cite{BS})\footnote{This is the same argument used in \cite{SS} Lemma 3.4 that also works for our case. As observed there, although Theorem 2 in \cite {BS} is stated for solutions, it in fact holds for the class $S^{*}(\gamma;f)$ we consider here. See Remark done in page 603 of \cite{BS}. For the precise argument (for equations of type (3) in \cite{BS}) see the proofs in page 604 of \cite{BS}.}  we obtain the equicontinuity of $(u_{k})$ in $B_{3/4}^{+}.$ Then, by the Arzela-Ascoli Theorem, we can extract a subsequence of $u_k$ which converges uniformly in $B_{3/4}^+$. Let $u_{\infty}$ be the limit of this subsequence. Now, observe that if $B_{r}\subset B_{1}^{+}$ and  $\phi\in W_{loc}^{2,n}(B_r)$ we have 

$$ \Big\|\mathcal{P}_{\gamma_{k}}^{\pm}(D^2\phi, \nabla\phi) - \mathcal{M}_{\lambda,\Lambda}^{\pm}(D^2\phi)\Big\|_{L^{n}(B_{r})}  + ||f_{k}||_{L^{n}(B_{r})} \\
 = ||(\gamma_{k}\cdot |\nabla \phi|)||_{L^{n}(B_{r})} + ||f_{k}||_{L^{n}(B_{r})} $$
The RHS above is less equal than
\begin{equation}\label{stability-going-to-zero}
 ||\gamma_{k}||_{L^{q}(B_{r})}^{n}\cdot || |\nabla \phi| ||_{L^{n\tau}(B_{r})}^{n} + |B_{r}|^{\frac{1}{n}-\frac{1}{q}}\cdot ||f_{k}||_{L^{q}(B_{r})}
 \end{equation}
where $\tau = q/(q-n)$ is the conjugate exponent of $q/n>1$.  Since, $\phi\in W_{loc}^{2,n}(B_{1}^{+})$ then by Rellich-Kondrachov embedding Theorem (Theorem 9.16 in \cite{Brezis-book}) we have $|\nabla \varphi|\in W^{1,n}(B_{r})  \hookrightarrow L^{p}(B_{r})$ for every $p\in[n,\infty)$. In particular, since the expression in  (\ref{stability-going-to-zero}) goes to zero as $k\to \infty$, we conclude that

$$ \Big\|\mathcal{P}_{\gamma_{k}}^{\pm}(D^2\phi, \nabla\phi) - \mathcal{M}_{\lambda,\Lambda}^{\pm}(D^2\phi)\Big\|_{L^{n}(B_{r})}  + ||f_{k}||_{L^{n}(B_{r})} \to 0 \quad \textnormal { as } k\to \infty.$$

\noindent  By the stability properties of $L^n$-viscosity solutions in this context (Theorem 9.4 of \cite{Koike-Swiech-WHUnbounded}) we conclude 

\begin{equation}
\left\{ \begin{array}{ll}
 u_{\infty} \in S(0) \ \textnormal { in } \ B_{2/3}^{+},\\\\
{\Vert u_{\infty} \Vert}_{L^{\infty}(B_{2/3}^+)} \leq 1, \\\\
u_{\infty}=0 \ \textnormal { on } \: B'_{2/3}.
\end{array} \right. 
\end{equation} 
 
\noindent Therefore, the RHS of the implication in (\ref{app estimate}) holds for $u_{\infty}$, i.e., there exists $A_{u_{\infty}}(0)$ such that 
\begin{equation}\label{improvement-inequality-1} 
\big|\big| u_{\infty}- A_{u_{\infty}}(0) \cdot x_n \big|\big|_{L^{\infty}(B_{\mu_{*}}^{+})} \leq \frac{1}{8} \mu_{*}^{1 + \alpha}.
\end{equation}
and 
$$  \vert A_{u_{\infty}}(0) \vert \leq F_{0}.$$ 
But $u_k \to u_{\infty}$ uniformly in $B_{2/3}^+ \supseteq B_{\mu_{*}}^+$. In particular, for $k$ sufficiently large,
\begin{equation}\label{improvement-inequality-2}
\big|\big|u_{k}-u_{\infty}\big|\big|_{L^{\infty}(B_{2/3}^{+})} \leq \frac{1}{8} {\mu_{*}}^{1 + \alpha}. 
\end{equation}
Thus, combining (\ref{improvement-inequality-1}) and (\ref{improvement-inequality-2}), we arrive at
$$|| u_{k}- A_{u_{\infty}}(0) \cdot x_n ||_{L^{\infty}(B_{\mu_{*}}^{+})} \leq \frac{1}{4} \mu_{*}^{1 + \alpha},$$
which is a contradiction to (\ref{contradiction arg}).  This finishes the proof of the Proposition.
\end{proof}

\section{ Proof of a particular case of Theorem  \ref{general-pointwise-boundary-krylov-general-boundary-data} - Zero tangent plane case} 

\begin{theorem}[{{\bf Pointwise gradient type estimate - zero tangent plane case}}] \label{pointwise-boundary-krylov-general-boundary-data}
Let $u \in C^0(\overline{B}_1^+) \cap S^*(\gamma, f)$ in $B_1^+$ where $\gamma, f \in L^{q}(B_1^+)$ with $q > n$ and $\beta_{*}:=\min\{1-n/q, \alpha_{00}^{-} \}$. Assume the boundary data $\varphi=u_{\mid_{B'_1}} \in C^{1,\omega}(0)$ with zero Taylor's polynomial at the origin where  $\omega\in\mathcal{DMC}(Q,\beta_{*})$ and $\delta_{\omega}=1$.
Let 
$$\vartheta(t):= t^{\beta_{*}}+\int_{0}^{t}\frac{\omega(s)}{s}ds \ \textnormal { for } \ t\in[0,1].  $$

 \noindent Then, there exists a unique $\Psi_0 \in {\mathbb{R}}$ such that for all $x \in B_1^+$,
\begin{equation}\label{eq-1-krylov-general-0-tangent-plane}
\left\vert u(x) - \Psi_0 \cdot x_n \right\vert \leq F_{1}\left( {\Vert u \Vert}_{L^{\infty}(B_1^+)} +[\varphi]_{C^{1, \omega}}(0) + {\Vert f \Vert}_{L^{q}(B_1^+)} \right)  {\vert x \vert}{\vartheta}(|x|),
\end{equation}
\begin{equation}\label{eq-2-krylov-general-0-tangent-plane}
\left\vert \Psi_0 \right\vert \leq F_{1} \left( {\Vert u \Vert}_{L^{\infty}(B_1^+)} +[\varphi]_{C^{1, \omega}}(0) + {\Vert f \Vert}_{L^{q}(B_1^+)} \right).
\end{equation}
Here, the universal constant 
$$F_{1}=F_{1}(n,q, \lambda, \Lambda, ||\gamma||_{L^{q}(B_{1}^{+})}, \alpha_{00}, \delta_{\omega}^{*}, \beta_{*},Q,  \int_{0}^{1}\omega(s)s^{-1}ds).$$
\end{theorem}
\begin{remark} In the previous Theorem, in the case $q=n$, we have that $\vartheta(t)=1+\int_{0}^{t}\omega(s)/sds$ which is not a modulus of continuity since $\vartheta(0)\geq1.$ In particular, this would not even imply the differentiability of $u$ at the origin on the boundary.
\end{remark}
\begin{remark}\label{specific-form-F1} In fact, in the Theorem \ref{pointwise-boundary-krylov-general-boundary-data} above, $F_{1}$ can be taken as
$$F_{1}=J_{1}\cdot\bigg(1+||\gamma||_{L^{q}(B_{1}^{+})}^{\frac{2+\beta_{*}}{1-{n}/{q}}}\bigg), \quad J_{1}=J_{1}(n,q, \lambda, \Lambda, \alpha_{00}, \delta_{\omega}^{*}, \beta_{*},Q,  \int_{0}^{1}\omega(s)s^{-1}ds).$$
It follows from the proof below that the dependence of $J_{1}$ (and thus of $F_{1}$) on $\int_{0}^{1}\omega(s)s^{-1}ds$ is monotonically increasing.
\end{remark}
\begin{proof} Let us start by the uniqueness of $\Psi_{0}$. Suppose $\overline{\Psi}_{0}$ is another number satisfying (\ref{eq-1-krylov-general-0-tangent-plane}). This way, for all $x\in B_{1}^{+}$
$$ \left\vert \Psi_0 \cdot x_n - \overline{\Psi}_{0}\cdot x_{n} \right\vert \leq 2F_{1}\left( {\Vert u \Vert}_{L^{\infty}(B_1^+)} +[\varphi]_{C^{1, \omega}}(0) + {\Vert f \Vert}_{L^{q}(B_1^+)} \right)  {\vert x \vert}{\vartheta}(|x|),
$$
Taking $x=te_{n}$ for $t>0$  in the inequality above and letting $t\to 0^{+}$ we readily obtain $\Psi_{0}=\overline{\Psi}_{0}.$\\

\noindent We observe from the statement that $0<\delta_{\omega}^{*}\leq \delta_{\omega}=1.$ Also, since $n>q$ we have $\beta_{*}>0$.\\

\noindent  We divide the proof in two parts. \\

\noindent \underline{\bf Part I:} Here we do the following claim.\\

\noindent \underline{\bf Claim:} There exists a universal constant $\gamma_{0}>0$ such that Theorem \ref{pointwise-boundary-krylov-general-boundary-data} holds if  $||\gamma||_{L^{q}(B_{1}^{+})} \leq \gamma_{0}$. Precisely, 
 $$F_{1}=F_{1}(n,q, \lambda, \Lambda, \alpha_{00}, Q, \beta_{*}, \delta_{\omega}^{*}, \int_{0}^{1}\omega(s)s^{-1}ds).$$ 

\noindent\underline{Proof of the Claim:}\
\noindent  According to Proposition \ref{improvement-of-flatness}, once $0<\beta_{*}<\alpha_{00}$ is fixed  we choose 
%\begin{equation} \label{important inequality}
%\mu_{*} = \min \left \lbrace {\overline{\mu_{\beta_{*}}}}, \Big(\frac{1}{N_{0}}\Big)^{\frac{1}{1-\beta_{*}}}\right \rbrace \in (0,\mu_{\beta_{*}}), 
%\end{equation} 
\begin{equation}\label{choice-of-scale}
\mu_{*} := \min \Bigg\{\delta_{\omega}^{*}, \frac{1}{e},  \left( \frac{1}{8F_{0}} \right)^{\frac{1}{\alpha_{00} - \beta_{*}}} \Bigg\}<  \min \Bigg\{\frac{2}{3},   \left( \frac{1}{8F_{0}} \right)^{\frac{1}{\alpha_{00} - \beta_{*}}} \Bigg\} = \mu_{\beta_{*}},
\end{equation}
and set 
\begin{equation}\label{N_{0}}
N_{0}:=2Q\bigg(\frac{1}{\beta_{*}} + \int_{0}^{1}\frac{\omega(t)}{t}dt\bigg).
\end{equation}
Here 
\begin{equation}\label{E1-dependent-gamma-zero}
F_{0}= F_{0}(n, q, \lambda, \Lambda) \quad \textnormal { given in the statement of Proposition \ref{improvement-of-flatness}}.
 \end{equation}
Observe that $\beta_{*}\in(0,1).$ After the choice of $\mu_{*}, $ the number $\varrho_{0}=\varrho_{0}(\beta_{*}, \mu_{*}) > 0$ given in Proposition \ref{improvement-of-flatness} is completely determined and universal.  Now, we set
\begin{equation}\label{definition-of-n}
N := \bigg({\Vert u \Vert}_{L^{\infty}(B_1^+)} + \frac{1}{\varrho_{0}} \Big( [\varphi]_{C^{1, \omega}}(0) + 2{\Vert f \Vert}_{L^{q}(B_1^+)} \Big)\bigg)
\end{equation}
and introduce
\begin{equation}\label{auxillary-modulus}
\omega_{0}(t):=t^{\beta_{*}} + \omega(t), \quad\forall t\in[0,1]. 
\end{equation}
Taking into consideration Lemma \ref{properties-modulus-of-continuity} item $i)$ (recall $Q\geq 1$), direct computation shows that
\begin{equation}\label{omega-zero-DMC-Q-beta_{*}}
\omega_{0}\in\mathcal{DMC}(Q,\beta^{*}) \quad \textnormal { with } \quad \delta_{\omega_{0}}^{*}=\delta_{\omega}^{*}.
\end{equation}
\begin{equation}\label{vartheta-dominates-integral-omega-zero}
\int_{0}^{t}\frac{\omega_{0}(s)}{s}ds\leq \bigg(1+\frac{1}{\beta_{*}}\bigg)\vartheta(t) \quad \forall t\in [0,1],
\end{equation}
\begin{equation}\label{omega-zero-dominated-vartheta}
\omega_{0}(t) \leq t^{\beta_{*}} + Q\int_{0}^{t}\frac{\omega(s)}{s}ds \leq Q\vartheta(t) \quad \forall t\in [0,1].
\end{equation}
 We can assume without lost of generality that $k_{0}=0$ in the definition of $\mathcal{DMC}(Q,\beta_{*})$. Our goal is to prove, by an inductive process,  that there exist a sequence of real numbers $\{A_k\}_{k \geq 0}$ such that for all $k\geq 0$ we have
\begin{equation}\label{osc seq real numbers}
\big|\big| u- A_{k} \cdot x_n \big|\big|_{L^{\infty}(B_{\mu_{*}^{k}}^{+})} \leq N\mu_{*}^{k}\omega_{0}(\mu_{*}^{k}),
\end{equation}
\begin{equation}\label{modulus real numbers}
\vert A_{k + 1} - A_{k} \vert \leq 2 F_{0} N \omega_{0}(\mu_{*}^{k}) \quad \textnormal {where }\:\: F_{0} \:\: \textnormal { is given by (\ref{E1-dependent-gamma-zero}).}
\end{equation}
Putting $A_0 = 0$ we conclude that (\ref{osc seq real numbers}) holds for $k = 0$. Now, we assume that for all $0 \leq j \leq k$ the conditions (\ref{osc seq real numbers}) and (\ref{modulus real numbers}) hold. We need to show that there exists a real number $A_{k + 1}$ such that (\ref{osc seq real numbers}) and (\ref{modulus real numbers}) hold for $j=k + 1$. In order to do that we define
\begin{equation} \label{main sequence}
u_k (x) := \frac{1}{2 N \mu_{*}^{k}\omega_{0}(\mu_{*}^{k})} \Big( u(\mu_{*}^{k} x) - A_k \mu_{*}^{k} \cdot x_n \Big) \quad \textnormal { for } x\in B_{1}^{+}.
\end{equation}
Remark \ref{perturbation-by-linear} together with Remark \ref{scaling-remark} imply that $u_k \in S^*(\gamma_{k}, f_{k})$ in  $B_1^{+}$ where
$$\gamma_{k}(x):=\mu_{*}^{k}\gamma(\mu_{*}^{k}x) \quad \textnormal { for } x\in B_{1}^{+}, $$
and 
$$f_k(x) := \frac{\mu_{*}^{k}}{2 N \omega_{0}(\mu_{*}^{k})}\Big(\vert f(\mu_{*}^{k} x) \vert + \gamma(\mu_{*}^{k}x) \vert A_k \vert \Big) \quad \textnormal { for } x\in B_{1}^{+}.$$
Also, 
$$ \varphi_{k}(x) = \frac{\varphi(\mu_{*}^{k} x)}{2 N \mu_{*}^{k}\omega_{0}(\mu_{*}^{k})} \quad \textnormal { for } \quad x\in B'_{1}.$$ 
From the hypothesis of induction in (\ref{osc seq real numbers}), we obtain that  ${\Vert u_k \Vert}_{L^{\infty}(B_{1}^{+})} \leq 1$.\\

\noindent Now, in order to apply the improvement of flatness, we need to estimate the following quantity $$M_k := ||\gamma_{k}||_{L^{q}(B_{1}^{+})} + {\Vert f_k \Vert}_{L^{q}(B_{1}^{+})} + {\Vert \varphi_k \Vert}_{L^{\infty}(B'_1)}.$$ 

\noindent From the definition of $N$ and $\omega_{0}$ in (\ref{definition-of-n}) and (\ref{auxillary-modulus}), we have for all $x\in B'_{1}$
\begin{eqnarray*}
 |\varphi_{k}(x)| = \frac{|\varphi(\mu_{*}^{k} x)|}{2 N \mu_{*}^{k}\omega_{0}(\mu_{*}^{k})}  &\leq &  \frac{1}{2N} \frac{|\varphi(\mu_{*}^{k} x)|}{\mu_{*}^{k}\omega(\mu_{*}^{k})} \\
 & \leq & \frac{1}{2N} \frac{|\varphi(\mu_{*}^{k} x)||x|}{(\mu_{*}^{k}|x|)\omega(\mu_{*}^{k}|x|)} \quad( \omega \textnormal { is non-decreasing})\\
 & \leq & \frac{1}{2N}\Big([\varphi]_{C^{1, \omega}}(0)\Big) |x| \\
 & \leq & \frac{[\varphi]_{C^{1, \omega}}(0)}{2N}\\
 &\leq& \frac{\varrho_{0}}{2}.
\end{eqnarray*}
Thus, 
\begin{equation}\label{control-phi-k}
{\Vert \varphi_k \Vert}_{L^{\infty}(B'_1)} \leq \varrho_{0}/2.
\end{equation}
 Also, since $q>n$

\begin{equation}\label{control-gamma-k}
||\gamma_{k}||_{L^{q}(B_{1}^{+})} \leq  (\mu_{*}^{k})^{1-\frac{n}{q}}||\gamma||_{L^{q}(B_{1}^{+})} \leq ||\gamma||_{L^{q}(B_{1}^{+})}.
\end{equation}

\noindent Moreover, since $\omega_{0}(t)\geq t^{\beta_{*}}$ for all $t\in [0,1]$,  we have 

$$ |f_{k}(x) | \leq \frac{\mu_{*}^{k}}{2 N (\mu_{*}^{k})^{\beta_{*}}}\Big(\vert f(\mu_{*}^{k} x) \vert + \gamma(\mu_{*}^{k}x) \vert A_k \vert \Big) \quad \textnormal { for } x\in B_{1}^{+}.$$ 
Also, we observe that  by Lemma \ref{properties-modulus-of-continuity} item $ii)$ and by the choice made in (\ref{choice-of-scale}) $(\mu_{*}\in (0,\delta_{\omega}^{*}))$
\begin{equation}\label{N0-control-dini}
N_{0} = 2Q\int_{0}^{1}\frac{\omega_{0}(s)}{s}ds\geq  \sum_{j = 0}^{\infty} \omega_{0}(\mu_{*}^{j}).
\end{equation}
Since $\beta_{*}\leq 1 - n/q $ ~then $(\mu_{*}^{k})^{1-\beta_{*}-\frac{n}{q}} \leq 1$  for any $k\geq 0$ and thus we find 
\begin{eqnarray}
||f_{k}||_{L^{q}(B_{1}^{+})} & \leq & (\mu_{*}^{k})^{1 - \beta_{*} - \frac{n}{q}}\Bigg(\frac{{\Vert f \Vert}_{L^{q}(B_{1}^{+})}}{2N}   +  \frac{||\gamma||_{L^{q}(B_{1}^{+})}}{2N}\cdot |A_{k}|\Bigg) \nonumber\\
& \leq & \frac{\varrho_{0}}{4} + \frac{||\gamma||_{L^{q}(B_{1}^{+})}}{2 N} \cdot \sum_{j = 0}^{k - 1} \vert A_{j + 1} - A_j \vert \nonumber\\
&\leq & \frac{\varrho_{0}}{4} +   \Big( F_{0} \cdot||\gamma||_{L^{q}(B_{1}^{+})}\Big) \cdot \sum_{j = 0}^{\infty} \omega_{0}(\mu_{*}^{j})  \ (\textnormal {by the hypothesis of induction in} \ (\ref{modulus real numbers})) \nonumber\\
& \leq & \frac{\varrho_{0}}{4} + N_{0}F_{0} ||\gamma||_{L^{q}(B_{1}^{+})}\label{control-f-k}
\end{eqnarray}
Adding up (\ref{control-phi-k}), (\ref{control-gamma-k}) and (\ref{control-f-k}) we obtain 
\begin{equation}
M_k \leq \frac{3}{4} \varrho_{0} + \Big(1 +N_{0}F_{0}\Big) \cdot ||\gamma||_{L^{q}(B_{1}^{+})}
\end{equation}
Setting 
$${\gamma_0} := \frac{\varrho_{0}}{4(N_{0}F_{0} + 1)},$$ 
\noindent we arrive at 
$$  ||\gamma||_{L^{q}(B_{1}^{+})}\leq\gamma_{0} \Longrightarrow M_k \leq \varrho_{0}.$$
\noindent Now by the improvement of flatness, Proposition \ref{improvement-of-flatness}, there exists a constant $G_{0}$ such that $\vert G_{0} \vert \leq F_{0}$ and
\begin{equation} \label{modulus-ineq}
\big|\big| u_{k} - G_0 \cdot x_n \big|\big|_{L^{\infty}(B_{\mu_{*}}^{+})} \leq \frac{1}{4} \mu_{*}^{1 + \beta_{*}}.
\end{equation}
We define, 
\begin{equation} \label{Ak+1 definition}
A_{k + 1} := A_k + 2N\omega_{0}(\mu_{*}^{k}) G_{0}.
\end{equation}
Clearly, (\ref{modulus real numbers}) holds for $A_{k + 1}.$  By (\ref{Ak+1 definition}) and (\ref{main sequence}),  we have the following identities for $x\in B_{\mu_{*}^{k}},$
$$u(x) =  2 N \mu_{*}^{k}\omega_{0}(\mu_{*}^{k})u_k(x/\mu_{*}^{k}) +A_k \cdot x_n  $$

\noindent So, 
\begin{eqnarray*}
u(x) -A_{k+1}\cdot x_{n}  & = & 2 N \mu_{*}^{k}\omega_{0}(\mu_{*}^{k})u_k(x/\mu_{*}^{k}) - 2N\omega_{0}(\mu_{*}^{k})G_{0}\cdot x_{n} \\\\
&=& 2 N \mu_{*}^{k}\omega_{0}(\mu_{*}^{k})u_k(x/\mu_{*}^{k})- 2N\mu_{*}^{k}\omega_{0}(\mu_{*}^{k})G_{0}\cdot \bigg(\frac{x_{n}}{\mu_{*}^{k}}\bigg)\\\\
& = & 2 N \mu_{*}^{k}\omega_{0}(\mu_{*}^{k})\Bigg( u_k(x/\mu_{*}^{k}) -G_{0}\cdot \bigg(\frac{x_{n}}{\mu_{*}^{k}}\bigg)\Bigg)
\end{eqnarray*}
\noindent Thus,  by (\ref{modulus-ineq}) and  since $\omega_{0}$ satisfies the $\beta_{*}$ compatibility condition by (\ref{omega-zero-DMC-Q-beta_{*}}), we find
\begin{eqnarray*}
 \big|\big| u - A_{k+1} \cdot x_n \big|\big|_{L^{\infty}(B_{\mu_{*}^{k+1}}^{+})} & \leq & 2 N \mu_{*}^{k}\omega_{0}(\mu_{*}^{k}) \Big(4^{-1} \mu_{*}^{1 + \beta_{*}}\Big)  \\
 &= &  \frac{N}{2}\mu_{*}^{k+1}\omega_{0}(\mu_{*}^{k})\mu_{*}^{\beta} \\
 & \leq & N \mu_{*}^{k+1}\omega_{0}(\mu_{*}^{k+1}) \ (\textnormal {since by the choice done in } (\ref{choice-of-scale}) ~~\textnormal { we have } \ \mu_{*}\in (0,\delta_{\omega}^{*})).
\end{eqnarray*} 
This finishes the inductive construction and (\ref{osc seq real numbers}) and (\ref{modulus real numbers}) hold for all $k\in\mathbb{N}.$\\

\noindent Now,  
\begin{eqnarray}
|A_{m+k}  - A_{k}| &\leq& \sum _{j=k}^{m+k-1} |A_{j+1}-A_{j}| \label{convergence-of-As}\\
& \leq & 2F_{0}N\sum_{j = k}^{\infty} \omega_{0}(\mu_{*}^{j}) \quad (\textnormal {by } (\ref{modulus real numbers})) \nonumber\\
&\leq & 4QF_{0}N \int_{0}^{\mu_{*}^{k}}\frac{\omega_{0}(t)}{t}dt \quad (\textnormal {by Lemma} \  \ref{properties-modulus-of-continuity}\ \textnormal  {item}  \ ii))\nonumber\\
& \leq & 4QF_{0}N\bigg(1+\frac{1}{\beta_{*}}\bigg)\vartheta(\mu_{*}^{k})\quad \textnormal { by } (\ref{vartheta-dominates-integral-omega-zero}). \nonumber
\end{eqnarray}
\noindent By the chain of inequalities above, we see that the Dini continuity of $\omega_{0}$ implies that $\{A_{k}\}_{k\geq 0}$ is a Cauchy sequence. So, let $\widehat{\Psi_0} := \lim_{k \rightarrow \infty} A_k$. Thus, from (\ref{osc seq real numbers}) and also from the second and the last inequality in the chain (\ref{convergence-of-As}) (passing the limit as $m\to\infty$) we obtain
\begin{eqnarray*}
\big|\big| u - \widehat{\Psi_0} \cdot x_n \big|\big|_{L^{\infty}(B_{\mu_{*}^{k}}^{+})}& \leq & \big|\big| u - A_{k} \cdot x_n \big|\big|_{L^{\infty}(B_{\mu_{*}^{k}}^{+})} + {\mu}_{*}^k \vert A_k - \widehat{\Psi_{0}} \vert \\
& \leq &   N\mu_{*}^{k}\omega_{0}(\mu_{*}^{k}) + {\mu}_{*}^k\Big|\sum\limits_{j=k}^{\infty}(A_{j+1}-A_{j})\Big| \\
& \leq &QN {\mu}_{*}^{k}\vartheta(\mu_{*}^{k}) +4QF_{0}N\bigg(1+\frac{1}{\beta_{*}}\bigg)\mu_{*}^{k}\vartheta(\mu_{k}^{*}) \quad \textnormal { by } (\ref{omega-zero-dominated-vartheta})\\
& \leq & M_{0}N{\mu}_{*}^{k}\vartheta(\mu_{*}^{k}), \quad M_{0}:=Q+4QF_{0}\bigg(1+\frac{1}{\beta_{*}}\bigg).
\end{eqnarray*}

\noindent We also observe from the the chain of inequalities in (\ref{convergence-of-As}) by letting $m\to\infty$ and taking $k=0$ that
$$|\widehat{\Psi_{0}}| = \Big|\sum\limits_{j=0}^{\infty}(A_{j+1}-A_{j})\Big| \leq  4QF_{0}N\bigg(1+\frac{1}{\beta_{*}}\bigg)\vartheta(1) = 4QF_{0}\bigg(1+\frac{1}{\beta_{*}}\bigg)\Bigg(1 + \int_{0}^{1}\frac{\omega(s)}{s}ds\Bigg)N.$$
Since $\vartheta(t)=\int_{0}^{t} {\omega^{*}(s)}{s}^{-1}ds$ where $\omega^{*}(s)=\beta^{*}s^{\beta_{*}} + \omega(s)$ and $\omega^{*}$ satisfies the $Q-$decreasing quotient property in Definition \ref{definition-MC} $($recall $\beta_{*}\in(0,1))$, we have by  Lemma \ref{properties-modulus-of-continuity} item $i)$ applied to $\omega^{*}$  that
\begin{equation}\label{final-solving-passing-scale}
\vartheta(\mu_{*}^{k}) \leq \frac{2Q^{2}}{\mu_{*}}\vartheta(\mu_{*}^{k+1}) \quad \forall k\in\mathbb{N}.
\end{equation}
\noindent Now, for $x\in B_{1}^{+}$, we can choose $k\in \mathbb{N}$ so that $\mu_{*}^{k+1}\leq |x| <\mu_{*}^{k}.$ Thus,

\begin{eqnarray*}
\big|u(x) - \widehat{\Psi_{0}}\cdot x_{n} \big| &\leq& \big|\big| u - \widehat{\Psi_0} \cdot x_n \big|\big|_{L^{\infty}(B_{\mu_{*}^{k}}^{+})}\\
& \leq & M_{0}N{\mu}_{*}^{k}\vartheta(\mu_{*}^{k})\\
& = &M_{0}N{\mu}_{*}^{k+1}\frac{2Q^{2}}{\mu_{*}^{2}}\vartheta(\mu_{*}^{k+1}) \quad \textnormal { by } ~~(\ref{final-solving-passing-scale})\\
&\leq&\Bigg(\frac{2M_{0}\cdot Q^{2}}{\mu_{*}^{2}}\Bigg)\cdot N\cdot |x|\vartheta(|x|)
\end{eqnarray*}

\noindent This way,  since 

$$ N\leq \bigg(1 + 2\varrho_{0}^{-1} \bigg)\bigg({\Vert u \Vert}_{L^{\infty}(B_1^+)}  +[\varphi]_{C^{1, \omega}}(0) +{\Vert f \Vert}_{L^{q}(B_1^+)}\bigg)$$ 
the claim is now proven with 
\begin{equation}\label{F1-star}
F_{1}^{*}:= \bigg(1 + 2\varrho_{0}^{-1}\bigg)\Bigg(4QF_{0}\bigg(1+\frac{1}{\beta_{*}}\bigg)\Bigg(1 + \int_{0}^{1}\frac{\omega(s)}{s}ds\Bigg)+\Bigg(\frac{2M_{0}\cdot Q^{2}}{\mu_{*}^{2}}\Bigg)\Bigg),
\end{equation}
replacing $F_{1}$ in the statement of the Theorem. This concludes the proof if the claim and Part I of the proof. \\

\noindent \underline{\bf Part II:} Suppose $||\gamma |_{L^{q}(B_{1}^{+})}>{\gamma_{0}}.$\\

\noindent In this case, we use a scaling argument. Let us set  
\begin{equation}\label{gamma-is-big}
r_{0}:=\Bigg(\frac{\gamma_{0}}{||\gamma||_{L^{q}(B_{1}^{+})}}\Bigg)^{\frac{1}{1-{n}/{q}}} <1.
\end{equation}
\noindent We define 
$$v(x) := u \left(r_{0} x \right), \ \ \ \ x \in B_{1}^+.$$
We denote $\varphi_{v}:=v_{\mid_{B'_1}}.$ In this case, by Remark \ref{scaling-remark},  $v \in S^* \left( \gamma_{00}, f_{00} \right)$ in $B_1^+$, where 
$$\gamma_{00}(x):= r_{0}\gamma(r_{0}x), \quad f_{0}(x) := r_{0}^{2} f \left(r_{0} x \right) \quad \textnormal {  for  } \quad x\in B_{1}^{+}.$$

\noindent  Also, by Remark \ref{scaling-pointwise-dini}, $\varphi_{v}$ is $C^{1,\omega}$ at $0$ and 
$$[\varphi_{v}]_{C^{1,\omega}}(0)\leq [\varphi]_{C^{1,\omega}}(0).$$
Moreover, 
$$||\gamma_{00}||_{L^{q}(B_{1}^{+})} \leq r_{0}^{1-n/q}||\gamma||_{L^{q}(B_{1}^{+})} \leq \gamma_{0},$$
$${\Vert v \Vert}_{L^{\infty}(B_1^+)} \leq {\Vert u \Vert}_{L^{\infty}(B_{1}^+)}, \quad {\Vert f_{00} \Vert}_{L^{q}(B_1^+)}\leq r_{0}^{2-n/q}{\Vert f \Vert}_{L^{q}(B_1^+)}\leq {\Vert f \Vert}_{L^{q}(B_1^+)}.  $$
By the previous claim,  there exists $\widehat{\Psi_{0}}$ such that  for all $x \in B_1^+$,
\begin{equation}
\left\vert v(x) - \widehat{\Psi_0} \cdot x_n \right\vert \leq F_{1}^{*} \left( {\Vert v \Vert}_{L^{\infty}(B_1^+)} +[\varphi_{v}]_{C^{1, \omega}}(0) + {\Vert f_{00} \Vert}_{L^{q}(B_1^+)} \right)  {\vert x \vert}\vartheta(|x|),
\end{equation}
\begin{equation}
\vert \widehat{\Psi_0} \vert \leq F_{1}^{*} \left( {\Vert v \Vert}_{L^{\infty}(B_1^+)} + [\varphi_{v}]_{C^{1, \omega}}(0) + {\Vert f_{00} \Vert}_{L^{q}(B_1^+)} \right),
\end{equation}
where $F_{1}^{*}=F_{1}^{*}(n,q, \lambda, \Lambda, \alpha_{00}, \delta_{\omega}^{*}, \beta_{*}, Q, \int_{0}^{1}\omega(s)s^{-1}ds).$ Translating this back in terms of $u$ we find 

\begin{equation}\label{taylor-pointwise-krylov}
 \left\vert u(x) - \widetilde{\Psi_0} \cdot x_n \right\vert \leq \widetilde{F_{1}} \left( {\Vert u \Vert}_{L^{\infty}(B_1^+)} + [\varphi]_{C^{1, \omega}}(0)+ {\Vert f \Vert}_{L^{q}(B_1^+)} \right)  {\vert x \vert}\cdot\vartheta \Big(r_{0}^{-1}|x|\Big) \quad \textnormal { in } \quad B_{r_{0}}^{+}
\end{equation}
where 
\begin{equation}\label{new-gradient-estimate}
|\widetilde{\Psi_{0}}| \leq \widetilde{F_{1}}\left( {\Vert u \Vert}_{L^{\infty}(B_1^+)} +[\varphi]_{C^{1, \omega}}(0) + {\Vert f \Vert}_{L^{q}(B_1^+)} \right) , \quad \widetilde{F_{1}}:=r_{0}^{-1}F_{1}^{*}.
\end{equation}

\noindent Clearly, from  Lemma \ref{properties-modulus-of-continuity} item $i)$ since $\Theta=r_{0}^{-1}>1$ and again $\vartheta(t)=\int_{0}^{t} \frac{\omega^{*}(s)}{s}ds$ where  $\omega^{*}(s)=\beta^{*}s^{\beta_{*}} + \omega(s)$ and $\omega^{*}$ satisfies the $Q-$decreasing quotient property in Definition \ref{definition-MC}, we have

$$ \vartheta\Big(r_{0}^{-1}|x|\Big) \leq 2Q^{2}r_{0}^{-1}\vartheta(|x|) \quad \textnormal { for all } \ |x|\leq r_{0}.$$

\noindent So, setting $\widehat{F_{1}}=2Q^{2}r_{0}^{-1}\widetilde{F_{1}}=2Q^{2}r_{0}^{-2}F_{1}^{*},$ (\ref{taylor-pointwise-krylov}) becomes, 
\begin{equation}\label{taylor-pointwise-krylov-new}
 \left\vert u(x) - \widetilde{\Psi_0} \cdot x_n \right\vert \leq \widehat{F_{1}} \left( {\Vert u \Vert}_{L^{\infty}(B_1^+)} + [\varphi]_{C^{1, \omega}}(0)+ {\Vert f \Vert}_{L^{q}(B_1^+)} \right)  {\vert x \vert}\cdot\vartheta (|x|) \quad \textnormal { in } \quad B_{r_{0}}^{+}.
 \end{equation}

\noindent Finally to treat points  $x$ in $B_{1}^+ \setminus B_{r_{0}}^+$, we set 

$$M:= \left( {\Vert u \Vert}_{L^{\infty}(B_1^+)} +[\varphi]_{C^{1, \omega}}(0) + {\Vert f \Vert}_{L^{q}(B_1^+)} \right).$$

\noindent This way, we estimate (since $\vartheta(t) \geq t^{\beta_{*}} \ \forall t\in[0,1])$
\begin{eqnarray}
\left\vert u(x) - \widetilde{\Psi_0} \cdot x_n \right\vert &\leq & \left( {\Vert u \Vert}_{L^{\infty}(B_1^+)}  + | \widetilde{\Psi_0}|\right) \nonumber\\
&\leq& (\widetilde{F_{1}}+1)\cdot M\nonumber\\
&= & (\widetilde{F_{1}}+1)\cdot r_{0}\cdot r_{0}^{-1} \cdot \vartheta\big(r_{0}\big)\cdot \big(\vartheta\big(r_{0}\big)\big)^{-1}\cdot M\nonumber\\
&\leq & (\widetilde{F_{1}}+1) \cdot  r_{0}^{-1}\cdot \big(\vartheta\big(r_{0}\big)\big)^{-1}\cdot M\cdot |x|\vartheta (|x|) \nonumber\\
&\leq &  (\widetilde{F_{1}}+1) r_{0}^{-(1+\beta_{*})}.\label{trick-to-extend-outside}
\end{eqnarray}
\noindent Now, adding up the estimates (\ref{F1-star}), (\ref{new-gradient-estimate}), (\ref{taylor-pointwise-krylov-new}) and (\ref{trick-to-extend-outside}), the Theorem is proven for 
\begin{eqnarray*}
F_{1}:=F_{1}^{*}+\widetilde{F_{1}}+\widehat{F_{1}} + (\widetilde{F_{1}}+1)r_{0}^{-(1+\beta_{*})} &= &F_{1}^{*} + r_{0}^{-(2+\beta_{*})}\Big( F_{1}^{*} (2 + 2Q^{2})+1\Big) \\
&  = &  F_{1}^{**}\Bigg(1+ \Bigg(\frac{||\gamma||_{L^{q}(B_{1}^{+})}}{\gamma_{0}}\Bigg)^{\frac{2+\beta_{*}}{1-{n}/{q}}}\Bigg) \\
& \leq & F_{1}^{**}\Bigg( \Big(1+ (\gamma_{0})^{\frac{-(2+\beta_{*})}{1-{n}/{q}}}\Big)\bigg(1+||\gamma||_{L^{q}(B_{1}^{+})}^{\frac{2+\beta_{*}}{1-{n}/{q}}}\bigg)\Bigg)\\
& = & F_{1}^{***}\bigg(1+||\gamma||_{L^{q}(B_{1}^{+})}^{\frac{2+\beta_{*}}{1-{n}/{q}}}\bigg)
\end{eqnarray*}
where $F_{1}^{**}:= F_{1}^{*} (2 + 2Q^{2})+1$ and $F_{1}^{***}:=F_{1}^{**}\Big(1+ (\gamma_{0})^{\frac{-(2+\beta_{*})}{1-{n}/{q}}}\Big).$
\end{proof}

\section{Proof of Theorem \ref{general-pointwise-boundary-krylov-general-boundary-data}} \label{proof-main-result}

\begin{proof} Uniqueness can be proven as in the previous Theorem. In order to make the proof more transparent, we divide it in two steps. First, we assume $\delta_{\omega}=1.$ We then consider the function $v(x):=u(x) -L(x',0)$ for $x\in B_{1}^{+}.$ Let $\varphi_{v}=v_{\mid_{B'_1}}.$ Clearly, $[\varphi_{v}]_{C^{1,\omega}}(0) \leq [\varphi]_{C^{1,\omega}}(0).$ Moreover, $v\in S^{*}(\gamma; \overline{f})$ in $B_{1}^{+}$ where $\overline{f} = |f| + \gamma\cdot |\nabla\varphi(0)|$ by Remark \ref{perturbation-by-linear}.\\

\noindent This way, by applying Theorem \ref{pointwise-boundary-krylov-general-boundary-data} to $v$, we obtain for all $|x|\leq 1$
\begin{equation}\label{eq-1-krylov-general-0-tangent-plane-XX}
\left\vert v(x) - \Psi_0 \cdot x_n \right\vert \leq F_{1}\left( {\Vert v \Vert}_{L^{\infty}(B_1^+)} +[\varphi_{v}]_{C^{1, \omega}}(0) + {\Vert \overline{f} \Vert}_{L^{q}(B_1^+)} \right)  {\vert x \vert}\vartheta(|x|),
\end{equation}
\begin{equation}\label{eq-2-krylov-general-0-tangent-plane-XX}
\left\vert \Psi_0 \right\vert \leq F_{1} \left( {\Vert v \Vert}_{L^{\infty}(B_1^+)} +[\varphi_{v}]_{C^{1, \omega}}(0) + {\Vert \overline{f} \Vert}_{L^{q}(B_1^+)} \right).
\end{equation}
where $F_{1}$ as in Theorem \ref{pointwise-boundary-krylov-general-boundary-data} above. Translating the estimates above back to $u$, the Theorem is proven with 
$$\overline{T_{0}}=\Big(1+||\gamma||_{L^{q}(B_{1}^{+})}\Big)\cdot F_{1}$$
 in the place of $T_{0}$. Now, we treat the general case $\delta_{\omega}\in[0,1]$ proceeding by scaling. Consider $v_{0}(x)=\delta_{\omega}^{-1}u(\delta_{\omega}x)$ for $x\in B_{1}^{+}.$ Now, $v_{0}\in C^{0}(\overline{B}_{1})\cap S^{*}(\gamma_{00}, f_{00})$ where $\gamma_{00}(x):=\delta_{\omega}\gamma(\delta_{\omega}x)$ and $f_{00}(x)=\delta_{\omega}f({\delta_{\omega}}x)$ for $x\in B_{1}^{+}.$ Setting $\overline{L}(x):=\delta_{\omega}^{-1}L(\delta_{\omega}x)$ for $x\in B'_{1}$ and  $\varphi_{0}:={v_{0}}_{\mid_{B'_1}}$ as the new boundary data, we have by Remark \ref{restriction-renormalization-MC} that $\varphi_{0}\in C^{1,\overline{\omega}}(0)$ and $[\varphi_{0}]_{C^{1,\overline{\omega}}}(0)\leq [\varphi]_{C^{1,{\omega}}}(0)$  where $\overline{\omega}(t)=\omega(\delta_{\omega}t)$ for $t\in[0,1].$  We recall that $\overline{\omega}\in\mathcal{DMC}(Q,\beta_{*})$ with $ \delta_{\overline{\omega}}=1, \delta_{\overline{\omega}}^{*}=\delta_{\omega}^{*}\in(0,1]$ and 
 \begin{equation}\label{preserving-dependence-dininess}
 \int_{0}^{\delta_{\omega}}\frac{w(s)}{s}ds=\int_{0}^{1}\frac{\overline{w}(s)}{s}ds.
 \end{equation}
 \noindent Furthermore, 
 $$||\gamma_{00}||_{L^{q}(B_{1}^{+})} \leq \delta_{\omega}^{1-n/q}||\gamma||_{L^{q}(B_{1}^{+})}, \quad ||f_{00}||_{L^{q}(B_{1}^{+})} \leq \delta_{\omega}^{1-n/q}||f||_{L^{q}(B_{1}^{+})}, \quad ||v_{0}||_{L^{\infty}(B_{1}^{+})} \leq \delta_{\omega}^{-1}||u||_{L^{\infty}(B_{1}^{+})}.$$
 We set
\begin{equation}\label{defining-vartheta-bar}
\overline{\vartheta}(t):= t^{\beta_{*}}+\int_{0}^{t}\frac{\overline{\omega}(s)}{s}ds = t^{\beta_{*}} + \int_{0}^{\delta_{\omega}t}\frac{\omega(s)}{s}ds \quad \textnormal { for  }\ t\in[0,1].
\end{equation} 

\noindent In particular since $\delta_{\omega}\leq 1,$

\begin{equation}\label{vartheta-bar-vartheta}
\overline{\vartheta}(\delta_{\omega}^{-1}t) = \delta_{\omega}^{-\beta_{*}} t^{\beta_{*}} + \int_{0}^{t}\frac{\omega(s)}{s}ds \leq \delta_{\omega}^{-\beta_{*}}\vartheta(t) \quad \forall t\in [0,\delta_{\omega}].
\end{equation}

 \noindent We can now apply the previous case to $v_{0}$ $($since $\delta_{\overline{\omega}}=1)$ to obtain $($once $\delta_{\omega}\in(0,1])$ for $|x|\leq \delta_{\omega}$
 \begin{eqnarray*}
\left\vert u(x) - L(x',0) - \Psi_0 \cdot x_n \right\vert &\leq& \overline{T_{0}}\delta_{\omega}^{-1}\left( {\Vert u \Vert}_{L^{\infty}(B_1^+)} +||\varphi||_{C^{1, \omega}}(0) +{\Vert {f} \Vert}_{L^{q}(B_1^+)} \right)  {\vert x \vert}\overline{\vartheta}(|x|\delta_{\omega}^{-1})\\\\
& \leq & \overline{T_{0}}\delta_{\omega}^{-1-\beta_{*}}\left( {\Vert u \Vert}_{L^{\infty}(B_1^+)} +||\varphi||_{C^{1, \omega}}(0) +{\Vert {f} \Vert}_{L^{q}(B_1^+)} \right)  {\vert x \vert}{\vartheta}(|x|) \end{eqnarray*}
{by } (\ref{vartheta-bar-vartheta}) and 

$$\left\vert \Psi_0 \right\vert \leq \delta_{\omega}^{-1}\overline{T_{0}} \left( {\Vert u \Vert}_{L^{\infty}(B_1^+)} +||\varphi||_{C^{1, \omega}}(0) + {\Vert {f} \Vert}_{L^{q}(B_1^+)} \right).$$
The proof is finished by observing (\ref{preserving-dependence-dininess}) and by taking 
$$T_{0}:=\delta_{\omega}^{-(1+\beta_{*})}\overline{T_{0}}=\delta_{\omega}^{-(1+\beta_{*})}\Big(1+||\gamma||_{L^{q}(B_{1}^{+})}\Big)\cdot F_{1}.$$
  \end{proof}

\section{ Proof of Corollary \ref{smooth-bdry-data}} \label{proof-corollary-smooth-data}
\begin{proof} Let us consider a generic point $x_{0}\in B'_{1/2}$. Define $\delta_{*}:=\min\{\delta_{\omega}, 1/2\}$ and $v(x):=\delta_{*}^{-1}\big(u(x_{0}+\delta_{*} x)\big)$ for $x\in B_{1}^{+}.$ Clearly, $v\in S^{*}(\overline{\gamma}; \overline{f})$ in $B_{1}^{+}$ where $\overline{\gamma}(x)=\delta_{*}\gamma(x_{0}+\delta_{*} x)$and $\overline{f}(x)=\delta_{*} f(x_{0}+\delta_{*} x)$  for $x\in B_{1}^{+}.$ Moreover, for $y\in H_{n-1}$, let 
$${L}_{x_{0}}(y)=\varphi(x_{0}) + \nabla\varphi(x_{0})\cdot (y-x_{0}) \quad (\textnormal {in our notation here} \quad \nabla\varphi(x',0) \in \mathbb{R}^{n-1}\times\{0\}) \ \textnormal {see Remark} \ \ref{taylor-on-the-fiber})$$ 
be the (classical) Taylor's polynomial of $\varphi$ at $x_{0}$ and $\varphi_{v}=v_{\mid_{B'_{1}}}.$ Setting, for $y\in H_{n-1}$
$$ \overline{L}_{0}(y):=\delta_{*}^{-1}{L}_{x_{0}}(x_{0}+\delta_{*} y),$$
we have for any $x\in B'_{1}$ by Remark \ref{classical-dini-taylor}
$$|\varphi_{v}(x) - \overline{L}_{0}(x)| = \delta_{*}^{-1}|\varphi(x_{0}+\delta_{*} x) - {L}_{x_{0}}(x_{0}+\delta_{*} x)| \leq[\nabla\varphi]_{C^{0,\omega}(B_{1})}|x|\overline{\omega}(|x|)$$
\noindent where $\overline{\omega}(t):=\omega(\delta_{*} t)$ for $t\in[0,1]$. From, the estimate above, we obtain 
$$[\varphi_{v}]_{C^{1,\overline{\omega}}}(0)\leq [\nabla\varphi]_{C^{0,\omega}(B'_{1})}.$$
Again, we observe by Rermark \ref{restriction-renormalization-MC}  that $\overline{\omega}\in\mathcal{DMC}(Q,\beta_{*})$ with $ \delta_{\overline{\omega}}=1, \delta_{\overline{\omega}}^{*}=\delta_{\omega}^{*}\in(0,1].$ Moreover, 

\begin{equation}\label{dininess-monotonicity}
\int_{0}^{1}\frac{\overline{\omega}(s)}{s}ds = \int_{0}^{\delta_{*}}\frac{\omega(s)}{s}ds \leq  \int_{0}^{\delta_{\omega}}\frac{\omega(s)}{s}ds 
\end{equation}
  \noindent  This way, applying Theorem \ref{general-pointwise-boundary-krylov-general-boundary-data} and recalling definition of $\overline{\vartheta}$ in (\ref{defining-vartheta-bar})\footnote{with $\delta_{\omega}$ replaced by $\delta_{*}=\min\{\delta_{\omega}, 1/2\}.$},  we obtain for $\Psi_{0}=\Psi_{0}(x_{0})$
 
 \begin{equation}\label{taylor-shifted}
  \left\vert v(x) - \overline{L}_{0}(x',0) - \Psi_0(x_{0}) \cdot x_n \right\vert \leq {T_{0}}\overline{M} \cdot {\vert x \vert}\overline{\vartheta}(|x|) \quad \forall x\in B_{1}^{+}, 
\end{equation}
$$  \left\vert \Psi_0(x_{0}) \right\vert \leq  T_{0}\overline{M},
$$ 
$$ \overline{M} := \delta_{*}^{-1}\left( {\Vert u\Vert}_{L^{\infty}(B_{1}^+)} + {\Vert {f} \Vert}_{L^{q}(B_{1}^+)} + ||\varphi||_{C^{1, \omega}(B'_{1})}  \right) =: \delta_{*}^{-1}\cdot N.$$
as before.\\

\noindent Setting $$A(x_{0}):= \nabla\varphi(x_{0})+ \Psi_{0}(x_{0})\cdot e_{n} \quad \textnormal { and } $$ 
\begin{equation}\label{taylor-to-match-appendix}
P_{x_{0}}(x)= \varphi(x_{0}) +  A(x_{0})\cdot (x-x_{0}).
\end{equation}
 \noindent  Rewriting (\ref{taylor-shifted}) in terms of $u$ and taking (\ref{vartheta-bar-vartheta}) into account, we obtain
\begin{equation}\label{final-taylox-each-point}
 \left\vert u(x) - P_{x_{0}}(x) \right\vert \leq T_{0}\delta_{*}^{-(1+\beta_{*})}N|x-x_{0}|\vartheta(|x-x_{0}|)\quad \forall x\in B_{\delta_{*}}^{+}(x_{0}), 
\end{equation}
Also, $\forall (x',0)\in B_{\delta_{*}}^{+}(x_{0}),$ we trivially have
\begin{equation}\label{tangent-part}
\left\vert u(x',0) - P_{x_{0}}(x',0) \right\vert = \left\vert \varphi(x',0) -L_{x_{0}}(x',0) \right\vert \\
 \leq  (T_{0}\delta_{*}^{-(1+\beta_{*})}+1)N|x-x_{0}|\vartheta(|x-x_{0}|).
 \end{equation}
 Moreover, 
\begin{equation}\label{gradient-estimate-final-taylor-each-point}
 |A(x_{0})| \leq |\Psi_{0}(x_{0})| + |\nabla \varphi(x_{0})| \leq  (T_{0}\delta_{*}^{-1}+1)N.
 \end{equation}
 Observe here that we can replace $\int_{0}^{1}{\overline{\omega}(s)}{s}^{-1}ds$ that appears embedded in $T_{0}$ by $\int_{0}^{\delta_{\omega}}{\omega(s)}{s}^{-1}ds$ due to (\ref{dininess-monotonicity}) and the monotonically increasing dependence of $T_{0}$ in that variable (Remark \ref{specific-form-F1} and Remark \ref{dependence-T0}). Now, we are exactly in the conditions of Lemma \ref{taylor-2} if we take 
 $$T:=(T_{0}\delta_{*}^{-(1+\beta_{*})} + T_{0}\delta_{*}^{-1}+2)N=T_{*}\cdot N \quad \textnormal { and }  \quad  r_{0}=\delta_{*}.$$
 
\noindent Inequalities (\ref{gradient-estimate-final-taylor-each-point}) and (\ref{holder-semi-norm-mixed-bdry}) imply

\begin{equation}\label{using-appendix-to-norm}
[A]_{C^{0,\omega}(\overline{B}'_{1/2})} \leq F\bigg(T +\frac{T}{\omega(\delta_{*}/4)}\bigg) = F\bigg(T_{*} +\frac{T_{*}}{\omega(\delta_{*}/4)}\bigg)\cdot N,
\end{equation}
where $F$ is a dimensional constant. The Theorem is proven by adding the constants in expressions (\ref{using-appendix-to-norm}) and (\ref{final-taylox-each-point})
$$ T_{1}:= F\bigg(T_{*} +\frac{T_{*}}{\omega(\delta_{*}/4)}\bigg) + T_{*}\leq T_{*}(F+1)\bigg(1+\frac{1}{\omega(\delta_{*}/4)}\bigg)=2(T_{0}\delta_{*}^{-(1+\beta_{*})} +1)(F+1)\bigg(1+\frac{1}{\omega(\delta_{*}/4)}\bigg).$$
\end{proof}

\section{Appendix: Uniform control on Taylor's expansion versus interior and boundary regularity}
In this Appendix, we present some estimates relating pointwise $C^{1,\omega}$ behavior and classical $C^{1,\omega}$ regularity in the interior and boundary case. These estimates are known (specially in the $C^{1,\alpha}$ case). However,  it is not so easy to find a reference for their proofs specially on the generality discussed here. We present the proofs in full details for completeness.

\begin{lemma}[{\bf $C^{1,\omega}-$ interior regularity by uniform control on Taylor's expansion}]\label{taylor-1} Let $u$ be defined in $B_{r}$ and  $\omega:[0,\delta_{\omega}]\to[0,\infty)$ a modulus of continuity. Moreover, let $r_{0}\leq \min\{r/2, \delta_{\omega} \}.$
Assume that for every $x_{0}\in \overline{B}_{r/2}$ there exists an affine function $P_{x_{0}}$ such that
\begin{equation}\label{uniform-polynomial-control} 
|u(x) - P_{x_{0}}(x)| \leq T |x-x_{0}|\omega(|x-x_{0}|) \quad\forall x\in B_{r} \ \ \textnormal {such that } \ |x-x_{0}|\leq r_{0}. 
\end{equation}
Then, $u\in C^{1,\omega}(\overline{B}_{r/2})$ with the following estimates 
\begin{equation}\label{holder-semi-norm-mixed}
[\nabla u]_{C^{0,\omega}(\overline{B}_{r/2})} \leq E\bigg(T + \frac{||\nabla u||_{L^{\infty}(\overline{B}_{r/2})}}{\omega(r_{0}/4)}\bigg).
\end{equation}
\begin{equation}\label{holder-semi-norm-pure}
[\nabla u]_{C^{0,\omega}(\overline{B}_{r_{0}/8})} \leq ET.
\end{equation}
where $E>0$ is a dimensional constant.
\end{lemma}
\begin{proof} Clearly, (\ref{uniform-polynomial-control}) implies that $u$ is differentiable at any point in $\overline{B}_{r/2}$. Suppose $x_{0},y_{0}\in \overline{B}_{r/2}$ are such that $d_{0}:=|x_{0}-y_{0}|> r_{0}/4.$ Then, 
\begin{equation}\label{holder-away}
 \frac{|\nabla u(x_{0})-\nabla u(y_{0})|}{\omega(|x_{0}-y_{0}|)}\leq \Bigg(\frac{2}{\omega(r_{0}/4)}\Bigg)||\nabla u||_{L^{\infty}(\overline{B}_{r/2})}.
 \end{equation}
Now, assume $x_{0},y_{0}\in \overline{B}_{r/2}$ with $2d_{0}=2|x_{0}-y_{0}|\leq r_{0}/2.$ Set $z_{0}:=(x_{0}+y_{0})/2 \in \overline{B}_{r/2}.$ Clearly, $|z_{0}-x_{0}|=|z_{0}-y_{0}|=d_{0}/2$. This way, $\overline{B}_{d_{0}/2}(z_{0})\subset \overline{B}_{d_{0}}(x_{0})\cap \overline{B}_{d_{0}}(y_{0})\subset B_{r}.$ By assumption (\ref{uniform-polynomial-control}) we have,
\begin{equation}\label{control-in-x_{0}}
||u - P_{x_{0}}||_{L^{\infty}(\overline{B}_{\frac{d_{0}}{2}}(z_{0}))} \leq  ||u - P_{x_{0}}||_{L^{\infty}(\overline{B}_{d_{0}}(x_{0}))}\leq T d_{0} \omega(d_{0}) 
\end{equation}
\begin{equation}\label{control-in-y_{0}}
||u - P_{y_{0}}||_{L^{\infty}(\overline{B}_{\frac{d_{0}}{2}}(z_{0}))} \leq  ||u - P_{y_{0}}||_{L^{\infty}(\overline{B}_{d_{0}}(y_{0}))}\leq T d_{0} \omega(d_{0})
\end{equation}
Now, recall that the affine function $P_{x_{0}}-P_{y_{0}}$ is harmonic. So, using gradient estimates, we have for a dimensional constant $\overline{C}>0$ that
\begin{eqnarray}\label{estimating-gradient-close-points}
\Big| \nabla{u}(x_{0}) - \nabla{u}(y_{0}) \Big| &\leq &\frac{2\overline{C}}{d_{0}}\Big|\Big|P_{x_{0}} - P_{y_{0}} \Big|\Big|_{L^{\infty}(\overline{B}_{d_{0}/2}(z_{0}))} \nonumber \\
 &\leq & \frac{2\overline{C}}{d_{0}} \Bigg( \Big|\Big| u - P_{x_{0}}\Big|\Big|_{L^{\infty}(\overline{B}_{d_{0}}(x_{0}))} +  \Big|\Big| u - P_{x_{0}}\Big|\Big|_{L^{\infty}(\overline{B}_{d_{0}}(y_{0}))} \Bigg) \nonumber \\
 & \leq & 4\overline{C}T\omega(d_{0}) =  4\overline{C}T\omega(|x_{0}-y_{0}|),
\end{eqnarray}
\noindent where we added (\ref{control-in-x_{0}}) and (\ref{control-in-y_{0}}) in the last inequality. The first estimate in the Lemma follows by adding (\ref{holder-away}) and (\ref{estimating-gradient-close-points}) by taking $E=4\overline{C}+2.$ The second estimate is a immediate consequence of the second part of the proof since for $x_{0},y_{0}\in \overline{B}_{r_{0}/8}$ we have $2d_{0}:=2|x_{0}-y_{0}|\leq r_{0}/2.$
\end{proof}

\begin{lemma}[{\bf $C^{1,\omega}-$boundary regularity by uniform control on Taylor's expansion}]\label{taylor-2}  Let $u$ be defined in $B_{r}$ and  $\omega:[0,\delta_{\omega}]\to[0,\infty)$ a modulus of continuity. Moreover, let $r_{0}\leq \min\{r/2, \delta_{\omega} \}.$  Suppose that for every $x_{0}\in \overline{B}'_{r/2}$ there exists an affine function $P_{x_{0}}$ such that for $x=(x',x_{n})$
\begin{equation}\label{uniform-polynomial-control-bdry} 
|u(x) - P_{x_{0}}(x)| \leq T |x-x_{0}|\omega(|x-x_{0}|) \quad \forall x\in B_{r}^{+} \ \textnormal { such that } \  |x-x_{0}|\leq r_{0}. 
\end{equation}
\begin{equation}\label{tangential-uniform-polynomial-control-bdry} 
|u(x',0) - P_{x_{0}}(x',0)| \leq T |x'-x'_{0}|\omega(|x'-x'_{0}|) \quad \forall (x',0)\in B'_{r} \ \textnormal { such that } \  |x'-x'_{0}|\leq r_{0}.
\end{equation}
Then, there is a unique function $A:B'_{r/2}\to\mathbb{R}^{n}$ such that 
\begin{equation}\label{expression-for-Px0}
P_{x_{0}} (x):= A(x_{0})(x-x_{0}) + u(x_{0}) \quad \forall x\in \mathbb{R}^{n}.
\end{equation}
Furthermore,  $A\in C^{0,\omega}(\overline{B}'_{r/2})$ and 
\begin{equation}\label{holder-semi-norm-mixed-bdry}
[A]_{C^{0,\omega}(\overline{B}'_{r/2})} \leq F\bigg(T +\frac{||A||_{L^{\infty}(\overline{B}'_{r/2})}}{\omega(r_{0}/4)}\bigg).
\end{equation}
\begin{equation}\label{holder-semi-norm-pure-bdry}
[A]_{C^{0,\omega}(\overline{B}'_{r_{0}/8})} \leq FT.
\end{equation}
where $F>0$ is a dimensional constant. The vector field $A$ can be thought as the gradient of $u$ along $B'_{r}$.
\end{lemma}
\begin{proof} For each $x_{0}\in B'_{r/2}$, we can write 
$$P_{x_{0}} (x):= A(x_{0})(x-x_{0}) + u(x_{0}) \quad \forall x\in \mathbb{R}^{n}.$$
We observe that by denoting $A(x_{0}) = (A_{T}(x_{0}), A_{n}(x_{0}))\in \mathbb{R}^{n-1}\times\mathbb{R}$ then for all $x\in\mathbb{R}^{n}$
$$ P_{x_{0}}(x) = A_{T}(x_{0})(x'-x'_{0}) + A_{n}(x_{0})x_{n} +u(x_{0})= P_{x_{0}}(x',0) + A_{n}(x_{0})x_{n}$$
\begin{equation}\label{relating-gradients-taylor-original-function}
 P_{x_{0}}(x',0) = A_{T}(x_{0})(x'-x'_{0}) + u(x_{0})
 \end{equation}
Now, by setting $v(x):=u(x)-P_{x_{0}}(x',0)$ we can rewrite expression (\ref{uniform-polynomial-control-bdry}) and (\ref{tangential-uniform-polynomial-control-bdry}) as:\\

 \noindent $\forall x_{0}\in \overline{B}'_{r/2}$  we have:
\begin{equation}\label{uniform-polynomial-control-bdry-rw} 
|v(x) - A_{n}(x_{0})x_{n}| \leq T |x-x_{0}|\omega(|x-x_{0}) \quad \forall x\in B_{r}^{+} \ \textnormal { such that } \  |x-x_{0}|\leq r_{0}. 
\end{equation}
\begin{equation}\label{tangential-uniform-polynomial-control-bdry-rw} 
|v(x',0)| \leq T |x'-x'_{0}|\omega(|x'-x'_{0})\quad \forall (x',0) \in B'_{r} \ \textnormal { such that } \  |x'-x'_{0}|\leq r_{0}. 
\end{equation}
Suppose $x_{0},y_{0}\in \overline{B}'_{r/2}$ are such that $d_{0}:=|x_{0}-y_{0}|> r_{0}/4.$ Then, 
\begin{equation}\label{holder-away-bdry}
 \frac{|A_{n}(x_{0})-A_{n}(y_{0})|}{\omega(|x_{0}-y_{0}|)}\leq 2(\omega(r_{0}/4))^{-1}||A||_{L^{\infty}(\overline{B}_{r/2})}.
 \end{equation}
Now, assume $x_{0},y_{0}\in \overline{B}'_{r/2}$ with $2d_{0}=2|x_{0}-y_{0}|\leq r_{0}/2.$ Set $z_{0}:=(x_{0}+y_{0})/2 \in \overline{B}'_{r/2}.$ Clearly, $|z_{0}-x_{0}|=|z_{0}-y_{0}|=d_{0}/2$. This way, $\overline{B}_{d_{0}/2}^{+}(z_{0})\subset \overline{B}_{d_{0}}^{+}(x_{0})\cap \overline{B}_{d_{0}}^{+}(y_{0})\subset B_{r}^{+}.$ By assumption (\ref{uniform-polynomial-control-bdry-rw}) we have,
\begin{equation}\label{control-in-x_{0}-bdry}
||v- A_{n}({x_{0}})x_{n}||_{L^{\infty}(\overline{B}_{\frac{d_{0}}{2}}^{+}(z_{0}))} \leq  ||v - A_{n}({x_{0}})x_{n}||_{L^{\infty}(\overline{B}_{d_{0}}^{+}(x_{0}))}\leq T d_{0} \omega(d_{0}). 
\end{equation}
\begin{equation}\label{control-in-y_{0}-bdry}
||v - A_{n}({y_{0}})x_{n}||_{L^{\infty}(\overline{B}_{\frac{d_{0}}{2}}^{+}(z_{0}))} \leq  ||v - A_{n}({y_{0}})x_{n}||_{L^{\infty}(\overline{B}_{d_{0}}^{+}(y_{0}))}\leq T d_{0} \omega(d_{0}).
\end{equation}
Now, recall that for a function of the type $z_{\mu}(x)=\mu\cdot x_{n} \ (\mu\in \mathbb{R})$  we always have for any $a\in\mathbb{R}^{n-1}$ that 
$$||z_{\mu}||_{L^{\infty}(B_{r}(a,0))}= ||z_{\mu}||_{L^{\infty}(B_{r}^{+}(a,0))}, \quad  |\mu| \leq \overline{C}\cdot \frac{||z_{\mu}||_{L^{\infty}(B_{r}^{+}(a,0))}}{r},$$
\noindent by gradient estimates since $z_{\mu}$ is a harmonic function.  Here, $\overline{C}=\overline{C}(n).$ \\

\noindent In particular, by (\ref{control-in-x_{0}-bdry}) and (\ref{control-in-y_{0}-bdry}), we arrive at 
\begin{eqnarray}\label{estimating-gradient-close-points-bdry}
| A_{n}(x_{0})-A_{n}(y_{0})|&\leq &\frac{2{\overline{C}}}{d_{0}}\big\|(A_{n}(x_{0})-A_{n}(y_{0}))x_{n}\big\|_{L^{\infty}(\overline{B}_{d_{0}/2}^{+}(z_{0}))} \nonumber \\
 &\leq & \frac{2{\overline{C}}}{d_{0}} \Big( \big\| v - A_{n}(x_{0})x_{n}\big\|_{L^{\infty}(\overline{B}_{d_{0}}(x_{0}))} +  \big\| v - A_{n}(y_{0})x_{n}\big\|_{L^{\infty}(\overline{B}_{d_{0}}(y_{0}))} \Big) \nonumber \\
 & \leq & 4\overline{C}T\omega(d_{0}) =  4\overline{C}T\omega(|x_{0}-y_{0}|).
\end{eqnarray}
So, by (\ref{holder-away-bdry}) and (\ref{estimating-gradient-close-points-bdry})
\begin{equation}\label{seminorm-A_{n}}
 [A_{n}]_{C^{0,\omega}(\overline{B}'_{r/2})} \leq 4\overline{C}T + 2(\omega(r_{0}/4))^{-1}||A||_{L^{\infty}(\overline{B}_{r/2})}.
\end{equation}
Now, if we denote $\varphi(x')=u(x',0)$ and $\overline{P}_{x_{0}}(x')=P_{x_{0}}(x',0)$  for $x'\in B_{r}^{n-1}(0)$ where
$$B_{r}^{n-1}(0):=\big\{x'\in\mathbb{R}^{n-1}; (x',0)\in B'_{r}\big\},$$

\noindent We can actually rewrite  (\ref{tangential-uniform-polynomial-control-bdry-rw}) as
$$|\varphi(x') - \overline{P}_{x_{0}}(x')| \leq T |x'-x'_{0}|\omega(|x'-x'_{0})\quad \forall x'\in B_{r}^{n-1}(0) \ \textnormal { such that } \  |x'-x'_{0}|\leq r_{0}. 
 $$
 By Lemma \ref{taylor-1}, we have $\varphi\in C^{1,\omega}(\overline{B}_{r/2})$. From  (\ref{relating-gradients-taylor-original-function}) we arrive to
 $$\nabla \varphi(x_{0})= \nabla \overline{P}_{x_{0}}(x_{0})=A_{T}(x_{0}).$$
 This way, from the estimates (\ref{holder-semi-norm-mixed}), (\ref{holder-semi-norm-pure}) and (\ref{seminorm-A_{n}})
  \begin{eqnarray*}
[A]_{C^{0,\omega}(\overline{B}'_{r/2})} &=& [A_{T}+A_{n}\cdot e_{n}]_{C^{0,\omega}(\overline{B}'_{r/2})}\\\\
 &\leq& [A_{T}]_{C^{0,\omega}(\overline{B}'_{r/2})}  + [A_{n}]_{C^{0,\omega}(\overline{B}'_{r/2})}  \\\\
 &\leq& [\nabla \varphi]_{C^{0,\omega}(\overline{B}'_{r/2})} +4\overline{C}T + 2(\omega(r_{0}/4))^{-1}||A||_{L^{\infty}(\overline{B}_{r/2})}\\\\
 &\leq &   E(T + \omega(r_{0}/4)^{-1}||A_{T}||_{L^{\infty}(\overline{B}_{r/2})}) + 4\overline{C}T + 2(\omega(r_{0}/4))^{-1}||A||_{L^{\infty}(\overline{B}_{r/2})}\\\\
 &\leq & (E+4\overline{C})T + 3\omega(r_{0}/4)^{-1}||A||_{L^{\infty}(\overline{B}_{r/2})}.
\end{eqnarray*}
Futhremore, by (\ref{holder-semi-norm-mixed}) and (\ref{estimating-gradient-close-points-bdry}), we obtain
\begin{eqnarray*}
[A]_{C^{\alpha}(\overline{B}'_{r_{0}/8})} &=& [A_{T}]_{C^{\alpha}(\overline{B}'_{r_{0}/8})} +[A_{n}]_{C^{\alpha}(\overline{B}'_{r_{0}/8})} \\\\
&\leq &[\nabla\varphi]_{C^{\alpha}(\overline{B}'_{r_{0}/8})} + 4\overline{C}T\\\\
&\leq & ET +4\overline{C}T.
\end{eqnarray*}
Uniqueness of $A$ follows easily.
\end{proof}
\section{Acknowledgements}
\noindent The authors are thankful to Luis Caffarelli for sharing thoughts, ideas and comments about the topics treated in this paper.  The authors are grateful to O. Savin for nice suggestions and for bringing Lemma \ref{lemma-control-by-the-distance} to their attention. The authors would like to thank B. Sirakov for kindly pointing out to them how his results in \cite{Boyan-1, Boyan-2} relate to IHOL discussed here and in \cite{BM-IHOL-PartI}. The first author was supported by CAPES (Brazil) by grant PNPD-1412370.  The second author was supported by CNPq (Brazil) by grants  PQ-310986/2013-3 and Universal-447536/2014-1. The third author was partially supported by NSFC (China) by grant 11371249.

\end{document}